\documentclass[11pt]{article}
\usepackage{amssymb,amsmath,amsfonts,amsthm,mathrsfs}
\usepackage{graphicx,graphics,psfrag,epsfig,calrsfs,cancel}
\usepackage[colorlinks=true, pdfstartview=FitV, linkcolor=blue, citecolor=red, urlcolor=blue]{hyperref}

\usepackage{caption,subfig,color}
\usepackage{tikz}
\usetikzlibrary{calc}

\usepackage{geometry}
\newtheorem{assumption}{Assumption}

\newtheorem{lemma}{Lemma}

\newtheorem{theorem}{Theorem}
\newtheorem{corollary}{Corollary}

\newcommand{\N}{{\mathbb N}}
\newcommand{\Z}{{\mathbb Z}}
\newcommand{\R}{{\mathbb R}}
\newcommand{\C}{{\mathbb C}}

\newcommand{\A}{{\mathbb A}}

\newcommand{\E}{{\mathbb E}}

\newcommand{\M}{{\mathbb M}}
\newcommand{\D}{{\mathbb D}}
\newcommand{\Dbar}{\overline{\mathbb D}}
\newcommand{\U}{{\mathcal U}}
\newcommand{\Ubar}{\overline{\mathcal U}}
\newcommand{\cercle}{{\mathbb S}^1}

\newcommand{\G}{{\mathcal G}}
\newcommand{\md}{\mathrm{d}}
\newcommand{\bqs}{\begin{equation*}}
\newcommand{\eqs}{\end{equation*}}
\newcommand{\bqq}{\begin{equation}}
\newcommand{\eqq}{\end{equation}}
\renewcommand{\Re}{\mathrm{Re}}
\renewcommand{\Im}{\mathrm{Im}}

\begin{document}

\title{Generalized Gaussian bounds \\
for discrete convolution powers}

\author{Jean-Fran\c{c}ois {\sc Coulombel} \& Gr\'egory {\sc Faye}\thanks{Institut de Math\'ematiques de Toulouse - UMR 5219, Universit\'e de
Toulouse ; CNRS, Universit\'e Paul Sabatier, 118 route de Narbonne, 31062 Toulouse Cedex 9 , France. Research of J.-F. C. was supported
by ANR project Nabuco, ANR-17-CE40-0025. G.F. acknowledges support from the ANR via the project Indyana under grant agreement ANR-21-
CE40-0008-01, from an ANITI (Artificial and Natural Intelligence Toulouse Institute)
Research Chair and from Labex CIMI under grant agreement ANR-11-LABX-0040. Emails: {\tt jean-francois.coulombel@math.univ-toulouse.fr},
{\tt gregory.faye@math.univ-toulouse.fr}}}
\date{\today}
\maketitle

\begin{abstract}
We prove a uniform generalized Gaussian bound for the powers of a discrete convolution operator in one space dimension. Our bound is
derived under the assumption that the Fourier transform of the coefficients of the convolution operator is a trigonometric rational function,
which generalizes previous results that were restricted to trigonometric polynomials. We also allow the modulus of the Fourier transform
to attain its maximum at finitely many points over a period.
\end{abstract}
\bigskip


\bigskip
\bigskip


For $1 \le q < + \infty$, we let $\ell^q(\Z;\C)$ denote the Banach space of complex valued sequences indexed by $\Z$ and such that the  $\ell^q$ norm,
defined for $u : \Z \to \C$ by
$$
\| \, u \, \|_{\ell^q} \, := \, \left( \, \sum_{j \in \Z} \, |u_j|^q \, \right)^{1/q} \, ,
$$
is finite. We also let $\ell^\infty(\Z;\C)$ denote the Banach space of bounded complex valued sequences indexed by $\Z$ and equipped with
the norm:
$$
\| \, u \, \|_{\ell^\infty} \, := \, \sup_{j \in \Z} \, |u_j| \, .
$$
Throughout this article, we use the notations
\begin{align*}
&\U := \{\zeta \in \C ~|~ |\zeta|>1 \}\, ,\quad \D := \{\zeta \in \C  ~|~ |\zeta|<1 \}\, ,\quad \cercle := \{\zeta \in \C  ~|~ |\zeta|=1 \} \, ,\\
&\Ubar := \U \cup \cercle \, ,\quad \Dbar := \D \cup \cercle \, .
\end{align*}
If $w$ is a complex number and $\rho$ a positive real number, the notation $B_\rho(w)$ stands for the open ball in $\C$ centered at $w$ and with
radius $\rho$, that is $B_\rho(w) := \{ z \in \C \, | \, |z-w|<\rho \}$.

The notation $\sigma (T)$ stands for the spectrum of a bounded operator $T$ acting on a Banach space $E$. We also use the notation
$\| \, \cdot \, \|_{E \to E}$ for the operator norm on a Banach space $E$.

\section{Introduction and main result}
\label{section1}

\subsection{A reminder on Laurent operators}

Let us recall a few facts about Laurent operators on $\ell^q(\Z;\C)$. If $a \in \ell^1(\Z;\C)$, we let $L_a$ denote the so-called Laurent (or convolution)
operator associated with the sequence $a$ \cite{TE,Nikolski}, which is defined by:
\begin{equation}
\label{laurent}
L_a \quad : \quad \big( u_j \big)_{j \in \Z} \, \longmapsto \, \left( \, \sum_{\ell \in \Z} \, a_\ell \, u_{j-\ell} \right)_{j \in \Z} \, = \, a \star u \, ,
\end{equation}
whenever the defining formula \eqref{laurent} for the sequence $L_a \, u$ makes sense. Here and below, $\star$ always stands for the convolution
product of two sequences indexed by $\Z$. In particular, Young's inequality shows that $L_a$ acts boundedly on $\ell^q(\Z;\C)$ for any $q \in [1,+\infty]$:
$$
\forall \, u \in \ell^q(\Z;\C) \, ,\quad \| \, L_a \, u \, \|_{\ell^q} \, \le \, \| \, a \, \|_{\ell^1} \, \| \, u \, \|_{\ell^q} \, .
$$
The spectrum of $L_a$ is also well-understood since the celebrated Wiener-Levy theorem, see \cite{newman}, characterizes the invertible elements
of $\ell^1(\Z;\C)$ for the convolution product (and we have the morphism property $L_a \circ L_b=L_{a \star b}$). Namely, the spectrum of $L_a$ as
an operator acting on $\ell^q(\Z;\C)$ does not depend on $q$ and is nothing but the image of the Fourier transform of the sequence $a$:
$$
\sigma \, (L_a) \, = \, \left\{ \, \sum_{\ell \in \Z} \, a_\ell \, {\rm e}^{\mathbf{i} \, \ell \, \xi} \, | \, \xi \in \R \right\} \, .
$$
Since $a$ belongs to $\ell^1(\Z;\C)$, its Fourier transform is continuous on $\R$. It actually belongs to the so-called Wiener algebra.

Following, among other works, \cite{Thomee,Diaconis-SaloffCoste,RSC1,RSC2}, we are interested here in giving uniform pointwise bounds for the
$n$-th iterated convolution product $a \star \cdots \star a=a^{\star \, n}$ as the number $n$ gets large. We use the convention $a^{\star \, 1}:=a$ and
$a^{\star \, n} :=a^{\star \, (n-1)} \star a$ for $n \ge 2$. Note that, by the morphism property $L_a \circ L_b=L_{a \star b}$, we have $(L_a)^n =L_{a^{\star n}}$
for any $n \in \N$. Beyond their own analytical interest, sharp bounds for the coefficients $(a^{\star \, n})_j$ or the precise description of their asymptotic
behavior are useful in probability theory and in numerical analysis. In probability theory, the coefficients $a_\ell$ correspond to the probability $\mathbb{P}
(X=\ell)$ where $X$ is a random variable with values in $\Z$. Considering the random walk:
$$
Y_n \, := \, X_1 \, + \, \cdots \, + \, X_n \, ,
$$
where the $X_m$'s are identically distributed, independent and follow the same law as $X$, then
$$
(a^{\star \, n})_j \, = \, \mathbb{P}(Y_n=j) \, ,
$$
for any $n \in \N^*$ and $j \in \Z$. In this context, asymptotic expansions for $(a^{\star \, n})_j$ are referred to as \emph{local limit theorems} and may be
found in \cite[Chapter VII]{Petrov}. These expansions involve the Gaussian function and Hermite polynomials. In numerical analysis, the study of the iterates
$(L_a)^n$, $n \in \N$, arises when one discretizes an evolutionary linear partial differential equation (set on the real line $\R$) by means of a finite difference
scheme. The transport equation or the heat equation are typical examples. We refer for instance to \cite{RM,gko}. From Young's inequality:
$$
\| \, (L_a)^n \, u \, \|_{\ell^\infty} \, = \, \| \, L_{a^{\star n}} \, u \, \|_{\ell^\infty} \, \le \, \| \, a^{\star n} \, \|_{\ell^1} \, \| \, u \, \|_{\ell^\infty} \, ,
$$
one observes that boundedness of the sequence $(\| \, a^{\star n} \, \|_{\ell^1})_{n \in \N}$ is a sufficient condition\footnote{It is actually a necessary
and sufficient condition, see \cite{Thomee}.} for what is referred to, in this context, as \emph{stability in the maximum norm}, that is:
$$
\sup_{n \in \N} \, \| \, (L_a)^n \, \|_{\ell^\infty \to \ell^\infty} \, < \, + \infty \, .
$$

The fundamental result in \cite{Thomee} characterizes, under suitable assumptions, the elements $a \in \ell^1(\Z;\C)$ such that the geometric sequence
$(a^{\star \, n})_{n \in \N}$ is bounded in $\ell^1(\Z;\C)$, see also \cite{hedstrom,Despres1,Diaconis-SaloffCoste,RSC2} and references therein for further
developments.  The sufficient part of the characterization in \cite{Thomee} (see also \cite{strang3,Despres1}) is performed by deriving a suitable ``algebraic''
pointwise bound for the coefficient $(a^{\star \, n})_j$ (see \cite[Lemma 2.4]{Thomee}). This bound is obtained by integrating by parts the Fourier transform
of $a^{\star \, n}$, and this manipulation requires the Fourier transform of $a$ to be $\mathcal{C}^2$. For the derivation of these algebraic bounds, the
support of the sequence $a$ may be arbitrary. Refining and optimizing this approach, the algebraic bound in \cite{Thomee} was turned in
\cite{Diaconis-SaloffCoste} into a generalized Gaussian bound, for finitely supported sequences $a$, thanks to a suitable contour deformation. The contours
chosen in \cite{Diaconis-SaloffCoste} can go arbitrarily far away from the real line (where the Fourier transform of $a$ is defined at first), which is the reason
why the authors in \cite{Diaconis-SaloffCoste} assume $a$ to have finite support, so that its Fourier transform extends to a holomorphic function on the
whole complex plane. The necessity of the conditions of \cite[Theorem 1]{Thomee} for $a$ to be power bounded in $\ell^1(\Z;\C)$ is proved under the
assumption that $a$ is finitely supported. This assumption on the support was removed in \cite{hedstrom}.

Our goal in this article is to extend the results in \cite{Diaconis-SaloffCoste} in two directions: we first wish to consider sequences $a$ with
\emph{infinite} support, since such sequences arise when one considers \emph{implicit} discretizations of partial differential equations. We
also wish to relax the assumption made in \cite{Diaconis-SaloffCoste} that the modulus of the Fourier transform of $a$ attains its maximum
at only one point over each period (say, at $0$, in the interval $[-\pi,\pi]$). When the modulus of the Fourier transform attains its maximum
at more than one point over a period, the arguments in either \cite{Thomee} or \cite{Diaconis-SaloffCoste}, rely on ($\mathcal{C}^\infty$)
partitions of unity which destroy the holomorphy of the Fourier transform. This is the reason why the bounds obtained in \cite{Diaconis-SaloffCoste}
in that situation are only of ``sub-Gaussian'' type (compare for instance \cite[Theorem 3.1]{Diaconis-SaloffCoste} with
\cite[Theorem 3.5]{Diaconis-SaloffCoste}). Actually, the results in \cite{Diaconis-SaloffCoste} were first refined and extended in \cite{RSC1},
where local limit theorems are proved for complex valued sequences $a$, and then further extended in \cite{RSC2} to deal with \emph{multidimensional}
situations. Another extension that is achieved in \cite{RSC2} is the proof of generalized Gaussian bounds for finitely supported sequences $a$
whose modulus of the Fourier transforms attains its maximum at several points (under a technical assumption that is discussed below).

Our approach in this article is quite different from the one in \cite{Thomee,Despres1,Diaconis-SaloffCoste,RSC2} where the coefficient $(a^{\star \, n})_j$
is represented by an integral involving the Fourier transform of $a$ (to the $n$-th power). Here we rather follow an approach which is commonly referred
to in the partial differential equation community as ``spatial dynamics'', which amounts to representing $(a^{\star \, n})_j$ in terms of the resolvent of the
operator $L_a$. The link between the two comes from the so-called functional calculus \cite[Chapter VII]{Conway} which expresses the \emph{temporal}
Green's function (here the coefficient $(a^{\star \, n})_j$) in terms of the \emph{spatial} Green's function, which is the solution to the resolvent equation:
$$
\big( \, z \, I \, - \, L_a \, \big) \, u \, = \, \boldsymbol{\delta} \, ,\qquad z \not \in \sigma (L_a) \, ,
$$
where $\boldsymbol{\delta}$ stands for the ``discrete'' Dirac mass ($\boldsymbol{\delta}_j=1$ if $j=0$, and $0$ otherwise). A detailed analysis
of the spatial Green's function with sharp holomorphic extensions and bounds is provided in Section \ref{section2} of the present article under
conditions that are similar to but, to some extent, less restrictive than those in \cite{Diaconis-SaloffCoste,RSC2}. Our main technical assumption
is the fact that we consider nonzero \emph{drift} velocities, which enables us to pass smoothly (that is, holomorphically) from ``temporal'' to
``spatial'' representations. Some cases with vanishing \emph{group velocities} (in the terminology of \cite{trefethen1}) are discussed at the end
of this article.

Once we have sharp holomorphic extensions and bounds for the spatial Green's function, our final argument relies on a suitable choice
of contours in the defining expression of the temporal Green's function. The choice of contours can be interpreted as an application of the
saddle point method \cite{debruijn}. A fundamental contribution in this direction is \cite{ZH98} (for the stability analysis of viscous shock
profiles) and we also refer to \cite{godillon} for an application of this method to the stability analysis of \emph{discrete} shock profiles. As
a matter of fact, our motivations for deriving generalized Gaussian bounds in the broadest possible context stems from the stability analysis
of discrete shock profiles but also from the theory of numerical boundary conditions for hyperbolic equations. An
application of the techniques developed in this article to finite rank perturbations of Toeplitz operators (on $\ell^2(\N;\C)$ rather than $\ell^2(\Z;\C)$)
is given in \cite{CFbord}. Discrete shock profiles will be considered in a forthcoming work. We now make several assumptions and state our
main result.

\subsection{Assumptions and main result}

This work is much inspired by the theory of partial differential equations and its numerical approximations. Hence, instead of sticking to the convolution
operators $L_a$ of the introduction, we shall rather use operators in the form:
$$
S_b \quad : \quad \big( u_j \big)_{j \in \Z} \, \longmapsto \, \left( \, \sum_{\ell \in \Z} \, b_\ell \, u_{j+\ell} \right)_{j \in \Z} \, ,
$$
with $b \in \ell^1(\Z;\C)$. One of the simplest such operators is the so-called shift operator ${\bf S}$ defined by:
$$
{\bf S} \quad : \quad \big( u_j \big)_{j \in \Z} \, \longmapsto \, \big( u_{j+1} \big)_{j \in \Z} \, .
$$
The two definitions of operators $L_a$ and $S_b$ are closely related. Namely, given $b \in \ell^1(\Z;\C)$, we have $S_b = L_a$ where the sequence $a \in
\ell^1(\Z;\C)$ is defined by $a_\ell := b_{-\ell}$ for all $\ell \in \Z$. Our convention, which is different from \cite{Diaconis-SaloffCoste,RSC1,RSC2}, is
the reason for the minus sign in \eqref{hyp:stabilite2} in the term $- \, \mathbf{i} \, \alpha_k \, \xi$ (compare, for instance, with
\cite[equation (3.3)]{Diaconis-SaloffCoste}).

We thus consider from now on two ``convolution'' operators $Q_0$ and $Q_1$ on $\Z$ with \emph{finite support}:
\begin{equation}
\label{operateurs}
\forall \, \sigma \, = \, 0,1 \, , \quad \forall \, j \in \Z \, , \quad (Q_\sigma \, u)_j \, := \, \sum_{\ell=-r}^p \, a_{\ell,\sigma} \, u_{j+\ell} \, ,
\end{equation}
where $r,p \in \N$ and the $a_{\ell,\sigma}$'s are complex numbers\footnote{When discretizing partial differential equations with real coefficients, these numbers
are real.}. In what follows, we always write:
\begin{equation}
\label{defphij}
Q_0 \, = \, L_{\phi_0} \, ,\quad \text{\rm and } \quad Q_1 \, = \, L_{\phi_1} \, ,
\end{equation}
where $\phi_0$ and $\phi_1$ are \emph{finitely supported} elements of $\ell^1(\Z;\C)$. Both operators $Q_0$ and $Q_1$ act boundedly on any $\ell^q(\Z;\C)$,
$1 \le q \le +\infty$. Our main focus below is on the three cases $q=1$, $q=2$ and $q=+\infty$. The integers $r,p$ in \eqref{operateurs} define the common
\emph{stencil} of the operators $Q_0,Q_1$. They are fixed by enforcing the conditions:
$$
| \, a_{-r,1} \, | \, + \, | \, a_{-r,0} \, | \, > \, 0 \, ,\quad | \, a_{p,1} \, | \, + \, | \, a_{p,0} \, | \, > \, 0 \, .
$$
Our first assumption is the following.

\begin{assumption}
\label{hyp:0}
The operator $Q_1$ is an isomorphism on $\ell^2(\Z;\C)$, that is:
\begin{equation}
\label{inversibilite}
\forall \, \kappa \in \cercle \, ,\quad \widehat{Q}_1 (\kappa) \, := \, \sum_{\ell=-r}^p \, a_{\ell,1} \, \kappa^\ell \, \neq \, 0 \, ,
\end{equation}
and it satisfies furthermore the index condition:
\begin{equation}
\label{indice}
\dfrac{1}{2 \, \mathbf{i} \, \pi} \, \int_{\cercle} \dfrac{\widehat{Q}_1' (\kappa)}{\widehat{Q}_1 (\kappa)} \, {\rm d}\kappa \, = \, 0 \, .
\end{equation}
\end{assumption}

The function $\widehat{Q}_1$ in \eqref{inversibilite} is referred to below as the \emph{symbol} of the convolution operator $Q_1$. We can similarly
define the symbol $\widehat{Q}_0$ associated with $Q_0$:
$$
\forall \, \kappa \in \cercle \, ,\quad \widehat{Q}_0 (\kappa) \, := \, \sum_{\ell=-r}^p \, a_{\ell,0} \, \kappa^\ell \, .
$$

Recalling the definition \eqref{defphij} of the sequence $\phi_1$, the condition \eqref{inversibilite} implies that $\phi_1$ is invertible in $\ell^1(\Z;\C)$
(thanks to the Wiener-Levy theorem \cite{newman}). We are then interested in the operator $\mathcal{L} := Q_1^{-1} \, Q_0$ and more specifically in its powers
$\mathcal{L}^n$ as $n$ becomes large. Since $\phi_1$ is invertible in $\ell^1(\Z;\C)$, we can write $\mathcal{L}=L_\phi$ with $\phi := \phi_1^{-1} \star
\phi_0$. Since we are interested in $\mathcal{L} = Q_1^{-1} \, Q_0$, we can always multiply $Q_0$ and $Q_1$ by the same nonzero complex number,
which does not modify $\mathcal{L}$. In view of Assumption \ref{hyp:0}, we thus always assume $\widehat{Q}_1 (1)=1$ from now on.

We briefly discuss the support of the sequence $\phi$ in order to compare our framework with that of \cite{Diaconis-SaloffCoste} or \cite{RSC2}.
The generalized Gaussian bounds in \cite{Diaconis-SaloffCoste} or \cite{RSC2} are obtained by assuming that $\phi$ has finite support (which
makes its Fourier transform an entire function). In our case, two situations occur:
\begin{itemize}
 \item The support of $\phi_1$ is a singleton, that is $Q_1 = a_{\underline{\ell},1} \, {\bf S}^{\underline{\ell}}$ for some integer $\underline{\ell}$ between
 $-r$ and $p$ (recall the notation ${\bf S}$ for the shift operator). Because of the condition \eqref{indice}, we have $\underline{\ell}=0$ so $Q_1$ is a
 nonzero multiple of the identity, and our normalization convention $\widehat{Q}_1 (1)=1$ makes $Q_1$ be the identity. In that case, $\phi=\phi_0$ is
 finitely supported. In numerical analysis, this situation corresponds to \emph{explicit schemes}. In probability theory, this situation corresponds to
 random walks with finite range.
 \item The support of $\phi_1$ contains at least two elements. Then the inverse $\phi_1^{-1}$ of $\phi_1$ for the convolution product has an infinite
 support. Apart from ``trivial'' cases where a factorization is possible, $\phi=\phi_1^{-1} \star \phi_0$ will also have infinite support. An example of this
 situation is provided in Section \ref{section4}. In numerical analysis, this situation corresponds to \emph{implicit schemes}.
\end{itemize}

As can be expected, a crucial role is played below by the symbol of $\mathcal{L}$ which is defined by:
\begin{equation}
\label{defF}
\forall \, \kappa \in \cercle \, , \quad F(\kappa) \, := \, \dfrac{\widehat{Q}_0 (\kappa)}{\widehat{Q}_1 (\kappa)} \, .
\end{equation}
The main difference between \cite{Diaconis-SaloffCoste} or \cite{RSC2} and the present work is that we allow $F(\exp (\mathbf{i}\, \xi))$ to be a trigonometric
rational function of $\xi$ rather than just a trigonometric polynomial in $\xi$. (Other results, such as local limit theorems, are derived in \cite{RSC2}
under the assumption that $F$ is of class $\mathcal{C}^\infty$ on $\cercle$, but we focus here on the derivation of generalized Gaussian bounds.)
In other words, we deal here with the class of sequences in $\ell^1(\Z;\C)$ whose Fourier transforms are trigonometric rational functions. The following
assumption on $F$ is inspired by the fundamental contribution \cite{Thomee}. The link between our Assumption \ref{hyp:1} below and the classification
obtained in \cite[page 280]{Thomee} in the case of trigonometric polynomials is discussed in Appendix \ref{appendix} at the end of this article (see Lemma
\ref{lem:comportementF}).

\begin{assumption}
\label{hyp:1}
The function $F$ defined in \eqref{defF} satisfies $\max_{\kappa \in \cercle} \, |F(\kappa)|=1$. Furthermore, there exists a finite set of points
$\{ \underline{\kappa}_1,\dots,\underline{\kappa}_K \}$, $K \ge 1$, in $\cercle$ such that:
\begin{equation}
\label{hyp:stabilite1}
\forall \, \kappa \in \cercle \setminus \big\{ \underline{\kappa}_1,\dots,\underline{\kappa}_K \big\} \, ,\quad \big| \, F(\kappa) \, \big| \, < \, 1 \, ,
\end{equation}
and for all index $k=1,\dots,K$, $F(\underline{\kappa}_k)$ belongs to $\cercle$. Moreover, for any $k=1,\dots,K$, there exist a nonzero real number
$\alpha_k$, an even integer $2\, \mu_k \ge 2$ and a complex number $\beta_k$ with positive real part such that:
\begin{equation}
\label{hyp:stabilite2}
\dfrac{F \Big( \underline{\kappa}_k \, {\rm e}^{\, \mathbf{i} \, \xi} \Big)}{F(\underline{\kappa}_k)} \, = \,
\exp \left( \, - \, \mathbf{i} \, \alpha_k \, \xi \, - \, \beta_k \, \xi^{\, 2 \, \mu_k} \, + \, O \Big( \xi^{\, 2 \, \mu_k+1} \Big) \right) \, ,
\end{equation}
as $\xi$ tends to $0$.
\end{assumption}

As in \cite{Thomee,Diaconis-SaloffCoste,RSC2}, the maximum of $F$ on the unit circle $\cercle$ is normalized to be $1$. Thanks to Beurling's
result (see \cite[page 428]{riesznagy}):
$$
\forall \, a \in \ell^1(\Z;\C) \, ,\quad \lim_{n \to \infty} \, \| \, a^{\star n} \, \|_{\ell^1}^{1/n} \, = \,
\max_{\theta \in \R} \, \left| \sum_{\ell \in \Z} \, a_\ell \, {\rm e}^{\mathbf{i} \, \ell \, \theta} \, \right| \, ,
$$
the case where the maximum equals $1$ is the limit case where the question of stability is not straightforward. In view of the result of Lemma
\ref{lem:comportementF} in Appendix \ref{appendix}, we just wish to exclude in Assumption \ref{hyp:1} the case where $F$ has constant
modulus on $\cercle$, and we then only consider the so-called points of type $\gamma$ in the terminology of \cite{Thomee}. The number $\alpha_k$
in \eqref{hyp:stabilite2} is necessarily real since $F(\kappa)$ belongs to $\Dbar$ for all $\kappa \in \cercle$. The fact that all real numbers $\alpha_k$
are nonzero is a major assumption that we make. It is fundamental below in the description of the so-called spatial Green's function. Examples of 
operators $Q_0,Q_1$ for which Assumptions \ref{hyp:0} and \ref{hyp:1} are satisfied are provided in Section \ref{section4} at the end of this article.

Since $F(\kappa)$ belongs to $\Dbar$ for all $\kappa \in \cercle$, the operator $\mathcal{L}$ is a contraction on $\ell^2(\Z;\C)$, that is:
$$
\forall \, u \in \ell^2 (\Z \, ; \, \C) \, ,\quad \| \, \mathcal{L} \, u \, \|_{\ell^2} \, \le \, \| \, u \, \|_{\ell^2} \, ,
$$
since the $\ell^2$ norms on both sides can be computed by the Parseval-Bessel identity. Of course, this implies that every power of $\mathcal{L}$
is also a contraction on $\ell^2(\Z;\C)$. In the field of numerical analysis, this property is referred to as $\ell^2$-stability\footnote{For scalar problems,
this is even equivalent to the so-called von Neumann stability condition \cite{RM,gko}.}, or strong stability \cite{strang1,tadmor}, for the ``numerical
scheme'':
$$
\begin{cases}
Q_1 \, u^{n+1} \, = \, Q_0 \, u^n \, ,& n \in \N \, ,\\
u^0 \in \ell^2(\Z) \, .&
\end{cases}
$$

Let us now define the quantities:
\begin{equation}
\label{defAl}
\forall \, z \in \C \, ,\quad \forall \, \ell=-r,\dots,p \, ,\quad \A_\ell (z) \, := \, z \, a_{\ell,1} \, - \, a_{\ell,0} \, .
\end{equation}
The following assumption already appears in several works devoted to the stability analysis of \emph{numerical boundary conditions} for discretized
hyperbolic equations, see, e.g., \cite{kreiss1,osher1,gks,goldberg-tadmor,jfcnotes} and references therein. Not only does it determine the minimal
integers $r$ and $p$ in \eqref{operateurs} (by prohibiting to add artificial zero coefficients), but it is also crucially used below to analyze the so-called
resolvent equation \eqref{defGz}. It might be relaxed though, but a more elaborate analysis would be required.

\begin{assumption}
\label{hyp:2}
The functions $\A_{-r}$ and $\A_p$ defined in \eqref{defAl} do not vanish on $\Ubar$.
\end{assumption}

For instance, if $a_{-r,1}$ and $a_{p,1}$ are nonzero, Assumption \ref{hyp:2} means that $a_{-r,0}/a_{-r,1}$ and $a_{p,0}/a_{p,1}$ belong to $\D$.
If $a_{-r,1}$ and $a_{p,1}$ are both zero, then $a_{-r,0}$ and $a_{p,0}$ should both be nonzero.

Thanks to Assumption \ref{hyp:2}, we can define the following companion matrix:
\begin{equation}
\label{defM}
\M (z) \, := \, \begin{bmatrix}
-\dfrac{\A_{p-1}(z)}{\A_p(z)} & \dots & \dots & -\dfrac{\A_{-r}(z)}{\A_p(z)} \\
1 & 0 & \dots & 0 \\
0 & \ddots & \ddots & \vdots \\
0 & 0 & 1 & 0 \end{bmatrix} \in \mathcal{M}_{p+r}(\C) \, ,
\end{equation}
which is holomorphic on the set $\{ z \in \C \, | \, |z| \, > \, \exp(-\underline{\eta}) \}$ for some parameter $\underline{\eta}>0$ which only depends
on the location of the root of $\A_p$ (if it exists). A crucial observation is that the upper right coefficient of $\M(z)$ is always nonzero, because of
Assumption \ref{hyp:2} and up to restricting $\underline{\eta}$, so the matrix $\M(z)$ is invertible for all relevant values of $z$. We shall repeatedly
use the inverse matrix $\M(z)^{-1}$ in what follows.

The analysis in this article heavily relies on a precise description of the spectrum of $\M(z)$ as $z$ runs through $\Ubar$ (and even sometimes slightly
through $\D$). This description is given in Lemma \ref{lem:1} below, and uses the following two assumptions.

\begin{assumption}
\label{hyp:3}
Either $Q_1$ is the identity, or $a_{-r,1}$ and $a_{p,1}$ are nonzero.
\end{assumption}

\noindent If $Q_1$ is the identity, Assumption \ref{hyp:0} is trivially satisfied since $\widehat{Q}_1 \equiv 1$. In the other case,
the complex numbers $a_{-r,1}$ and $a_{p,1}$ are both nonzero. In that case, $\widehat{Q}_1$ is a meromorphic function on $\C$ with a single pole
(that is located at $0$) of order $r$. (When $r$ equals zero, there is no pole.) By the residue theorem \cite{rudin}, the index condition \eqref{indice} is
equivalent to $\widehat{Q}_1$ having $r$ zeros (counted with multiplicity) in $\D \setminus \{ 0 \}$. Because $\kappa^r \, \widehat{Q}_1(\kappa)$ is a
polynomial of degree $p+r$, $\widehat{Q}_1$ then has $p$ zeros in $\U$.

\begin{assumption}
\label{hyp:4}
For all index $k=1,\dots,K$, let us define $\underline{z}_k \, := \, F(\underline{\kappa}_k) \in \cercle$. Then for any $k=1,\dots,K$, the set:
\begin{equation}
\label{defIk}
\mathcal{I}_k \, := \, \Big\{ \nu \in \{ 1,\dots,K \} \, | \, \underline{z}_\nu \, = \, \underline{z}_k \Big\}
\end{equation}
has either one or two elements\footnote{Note that $\mathcal{I}_k$ always contains $\{ k \}$.}. Furthermore, in case it has two elements, which we denote
$\nu_{k,1},\nu_{k,2}$, then $\alpha_{\nu_{k,1}} \, \alpha_{\nu_{k,2}} \, < \, 0$. (Let us recall that the drift parameters $\alpha_k$ are given in Assumption
\ref{hyp:1}.)
\end{assumption}

From now on, we always make Assumptions \ref{hyp:0}, \ref{hyp:1}, \ref{hyp:2}, \ref{hyp:3} and \ref{hyp:4}. Our main result is a partial extension
of \cite[Theorem 3.1]{Diaconis-SaloffCoste} and \cite[Theorem 1.8]{RSC2}. It gives a uniform, generalized Gaussian bound for the convolution
coefficients of the powers $\mathcal{L}^n$. A precise statement is the following.

\begin{theorem}
\label{thm:1}
Let the operators $Q_0,Q_1$ in \eqref{operateurs} satisfy Assumptions \ref{hyp:0}, \ref{hyp:1}, \ref{hyp:2}, \ref{hyp:3} and \ref{hyp:4} and the normalization
condition $\widehat{Q}_1(1)=1$. According to the above two cases in Assumption~\ref{hyp:3}, we have:
\begin{itemize}
\item {\bf Explicit case.} If $Q_1$ is the identity, then there
exist two constants $C>0$ and $c>0$ such that the operator $\mathcal{L}=Q_0$ satisfies the uniform generalized Gaussian bound:
\begin{equation}
\label{bound}
\forall \, n \in \N^* \, ,\quad \forall \, j \in \Z \, ,\quad \big| \, (\mathcal{L}^n \, \boldsymbol{\delta})_j \, \big| \, \le \, C \, \, \sum_{k=1}^K \,
\dfrac{1}{n^{1/(2\, \mu_k)}} \, \exp \left( - \, c \, \left( \dfrac{|j \, - \, \alpha_k \, n|}{n^{1/(2\, \mu_k)}} \right)^{\frac{2\, \mu_k}{2\, \mu_k-1}} \right) \, ,
\end{equation}
where $\boldsymbol{\delta}$ denotes the discrete Dirac mass defined by $\boldsymbol{\delta}_j=1$ if $j=0$ and $\boldsymbol{\delta}_j=0$
otherwise.
\item {\bf Implicit case.} If $Q_1$ is not the identity, then there
exist constants $C>0$, $L>0$ and $c>0$ such that the operator $\mathcal{L}=Q_1^{-1} \, Q_0$ satisfies the  bounds:
\begin{subequations}
\begin{align}
\forall \, n \in \N^* \, ,\quad \forall \, |j| \leq L \, n \, ,\quad \big| \, (\mathcal{L}^n \, \boldsymbol{\delta})_j \, \big| \, & \le \, C \, \, \sum_{k=1}^K \,
\dfrac{1}{n^{1/(2\, \mu_k)}} \, \exp \left( - \, c \, \left( \dfrac{|j \, - \, \alpha_k \, n|}{n^{1/(2\, \mu_k)}} \right)^{\frac{2\, \mu_k}{2\, \mu_k-1}} \right) \, , \\
\forall \, n \in \N^* \, ,\quad \forall \, |j| >  L \, n \, ,\quad \big| \, (\mathcal{L}^n \, \boldsymbol{\delta})_j \, \big| \, & \le \, C \, \,  \exp \left( -c \, n \, - \, c \, |j| \right) \, .
\end{align}
\label{boundimp}
\end{subequations}
\end{itemize}
\end{theorem}

A comparison between Theorem \ref{thm:1} and the analogous results in \cite{Thomee,Diaconis-SaloffCoste,RSC2} is provided in the next Subsection.
Otherwise, the rest of this article is organized as follows. In Section \ref{section2}, we prove sharp bounds on the so-called spatial Green's function.
This is where Assumptions \ref{hyp:2}, \ref{hyp:3} and \ref{hyp:4} are used. Then we use these preliminary bounds in Section \ref{section3} to obtain
the uniform bounds \eqref{bound} and \eqref{boundimp} for what we call the temporal Green's function. Examples and possible extensions are given
in Section \ref{section4}. The proofs of some intermediate and related results are gathered in Appendix \ref{appendix}.

\subsection{What is new ? and what is not ?}

Let us first observe that the sequence $\mathcal{G} := \mathcal{L} \, \boldsymbol{\delta}$ corresponds to the Laurent series expansion of $F$ near $\cercle$:
$$
\forall \, \kappa \in \cercle \, ,\quad F(\kappa) \, = \, \sum_{j \in \Z} \, \mathcal{G}_j \, \kappa^j \, .
$$
In particular, if $\mathcal{G}$ satisfies a generalized Gaussian bound of the form:
$$
\forall \, j \in \Z \, ,\quad | \, \mathcal{G}_j \, | \, \le \, C \, \exp ( -\, c \, |j|^s \, ) \, ,
$$
for some positive constants $C$ and $c$ and some exponent $s>1$, then $F$ extends to a holomorphic function on $\C \setminus \{ 0 \}$. In our framework,
we have $F(\kappa)=\widehat{Q}_0(\kappa)/\widehat{Q}_1(\kappa)$. Assuming that $\widehat{Q}_0(\kappa)$ and $\widehat{Q}_1(\kappa)$ have no common
factor, the only possible case where $F$ extends to a holomorphic function on $\C \setminus \{ 0 \}$ is when $\widehat{Q}_1$ does not vanish on $\C \setminus
\{ 0 \}$. Because of the form of $\widehat{Q}_1$ and the index condition \eqref{indice}, the only situation in which $F$ extends to a holomorphic function on $\C
\setminus \{ 0 \}$ is when $Q_1$ is the identity\footnote{We assume, of course, that $Q_0$ and $Q_1$ are irreducible.}. This argument explains why, in
\eqref{boundimp}, the bound for $(\mathcal{L}^n \, \boldsymbol{\delta})_j$ ``degenerates'' to $\exp (- \, c \, |j|)$ for any fixed $n$ (e.g., $n=1$) and for large $j$'s.

We now compare Theorem \ref{thm:1} with \cite[Theorem 3.1]{Diaconis-SaloffCoste} and \cite[Theorem 1.8]{RSC2} which, to our knowledge, are the two prior
references on generalized Gaussian bounds for convolution powers of complex sequences. In \cite[Theorem 3.1]{Diaconis-SaloffCoste}, the authors consider
(in our notation) the \emph{explicit} case ($Q_1$ is the identity) with $K=1$, but they make no assumption on the drift parameter $\alpha_1$ (while we assume
$\alpha_1 \neq 0$ in Theorem \ref{thm:1}). When specifying \cite[Theorem 1.8]{RSC2} to one space dimension, the result in \cite[Theorem 1.8]{RSC2} covers
the \emph{explicit} case ($Q_1$ is the identity) with $K \ge 1$, but it is then further assumed:
$$
\alpha_1 \, = \, \cdots \, = \, \alpha_K \, ,\quad \mu_1 \, = \, \cdots \, = \, \mu_K \, ,\quad \beta_1 \, = \, \cdots \, = \, \beta_K \, .
$$
As in \cite[Theorem 3.1]{Diaconis-SaloffCoste}, it is not assumed in \cite[Theorem 1.8]{RSC2} that the common value of the $\alpha_k$'s should be nonzero.

In our opinion, the main novelty here consists in considering sequences $a$ such that the $\alpha_k$'s, $\mu_k$'s and $\beta_k$'s are, to some extent, ``arbitrary''.
This is not entirely true since we assume that the $\underline{z}_k$'s are not ``too much'' equal (Assumption \ref{hyp:4}) and we further assume that all $\alpha_k$'s
are nonzero. We believe that this restriction on the $\alpha_k$'s in Theorem \ref{thm:1} is purely technical. For instance, in Corollary~\ref{coro2}, we give an
example of a general situation in the explicit case where some $\alpha_k$ can be zero and where the result of Theorem \ref{thm:1} can be used to obtain the same 
bound as in \eqref{bound}. We also note that the theory of numerical boundary conditions in \cite{gks} (see \cite{jfcnotes} for a thorough exposition) covers cases with 
vanishing group velocities (by making use in several occurrences of Puiseux expansions). In the continuous setting, reaction diffusion equations are another occurence 
where spatial dynamics and pointwise Green's function bounds cover some problems with zero drift velocities, see, e.g., \cite{FayeHolzer}. We thus hope that we shall 
be able to fully remove the restriction $\alpha_k \neq 0$ as well as Assumption \ref{hyp:4} in the future.

Our second improvement is to consider sequences in $\ell^1(\Z;\C)$ whose Fourier transforms are trigonometric rational functions, which is relevant for implicit
numerical schemes. Possible extensions of our work are listed in Section \ref{section4}.

\section{The spatial Green's function}
\label{section2}

The spectrum of $\mathcal{L}$ as an operator on $\ell^2(\Z;\C)$ is the parametrized curve $F(\cercle)$. We know from Assumption \ref{hyp:1}
that this curve touches the unit circle $\cercle$ at the points $\underline{z}_k$, $k=1,\dots,K$, and that it is located inside the open unit disk
$\D$ otherwise. Hence the resolvent set of $\mathcal{L}$ contains at least $\Ubar \setminus \{ \underline{z}_1,\dots,\underline{z}_K \}$. For
such values of $z$, we can thus define the sequence $G(z) \in \ell^2(\Z;\C)$ (the capital $G$ letter stands for Green, as in Green's function)
by the formula:
\begin{equation}
\label{defGz}
\big( z \, I \, - \, \mathcal{L} \big) \, G(z) \, = \, \boldsymbol{\delta} \, ,
\end{equation}
where we recall that $\boldsymbol{\delta}$ stands for the Dirac mass ($\boldsymbol{\delta}_j=1$ if $j=0$ and $\boldsymbol{\delta}_j=0$ if
$j \in \Z \setminus \{ 0 \}$).

From the definition $\mathcal{L} = Q_1^{-1} \, Q_0$, the equation \eqref{defGz} can be equivalently rewritten:
$$
\big( z \, Q_1 \, - \, Q_0 \big) \, G(z) \, = \, Q_1 \, \boldsymbol{\delta} \, ,
$$
and the definitions \eqref{operateurs}, \eqref{defAl} give the final form:
\begin{equation}
\label{resolvent}
\forall \, j \in \Z \, ,\quad \sum_{\ell=-r}^p \, \A_\ell (z) \, G_{j+\ell}(z) \, = \, (Q_1 \, \boldsymbol{\delta})_j \, ,
\end{equation}
together with the integrability conditions at infinity $G(z) \in \ell^2(\Z;\C)$.

\subsection{Spectral properties}

We introduce the augmented vectors:
$$
\forall \, j \in \Z \, ,\quad W_j(z) \, := \, \begin{bmatrix}
G_{j+p-1}(z) \\
\vdots \\
G_{j-r}(z) \end{bmatrix} \in \C^{p+r} \, ,\quad {\bf e} \, := \, \begin{bmatrix}
\, 1 \, \\
\, 0 \, \\
\vdots \\
\, 0 \, \end{bmatrix} \in \C^{p+r} \, ,
$$
and rewrite equivalently \eqref{resolvent} as:
\begin{equation}
\label{resolvent'}
\forall \, j \in \Z \, ,\quad W_{j+1}(z) \, - \, \M(z) \, W_j(z) \, = \, \dfrac{(Q_1 \, \boldsymbol{\delta})_j}{\A_p(z)} \, {\bf e} \, .
\end{equation}

The construction and analysis of the solution to the recurrence relation \eqref{resolvent'} relies on the following spectral splitting lemma, which is
originally due to Kreiss \cite{kreiss1} in the context of finite difference approximations.

\begin{lemma}[Spectral splitting]
\label{lem:1}
Let $z \in \Ubar \setminus \{ \underline{z}_1,\dots,\underline{z}_K \}$ and let the matrix $\M(z)$ be defined as in \eqref{defM}. Then $\M(z)$ has:
\begin{itemize}
 \item no eigenvalue on $\cercle$,
 \item $r$ eigenvalues in $\D \setminus \{ 0 \}$,
 \item $p$ eigenvalues in $\U$ (eigenvalues are counted with multiplicity).
\end{itemize}

Let now $k \in \{1,\dots,K\}$ be such that the set $\mathcal{I}_k$ in \eqref{defIk} is the singleton $\{ k \}$. Then if $\alpha_k>0$, the matrix
$\M(\underline{z}_k)$ has $\underline{\kappa}_k \in \cercle$ as a simple eigenvalue, it has $r-1$ eigenvalues in $\D$ and $p$ eigenvalues in
$\U$. If $\alpha_k<0$, the matrix $\M(\underline{z}_k)$ has $\underline{\kappa}_k \in \cercle$ as a simple eigenvalue, it has $r$ eigenvalues
in $\D$ and $p-1$ eigenvalues in $\U$.

Eventually, let now $k \in \{1,\dots,K\}$ be such that the set $\mathcal{I}_k$ in \eqref{defIk} has two elements $\nu_{k,1},\nu_{k,2}$. Then the matrix
$\M(\underline{z}_k)$ has $\underline{\kappa}_{\nu_{k,1}}$ and $\underline{\kappa}_{\nu_{k,2}}$ as simple eigenvalues on $\cercle$, it has $r-1$
eigenvalues in $\D$ and $p-1$ eigenvalues in $\U$.
\end{lemma}

\noindent The arguments are basically the same as in \cite{kreiss1} but we give them here for the sake of completeness.

\begin{proof}[Proof of Lemma \ref{lem:1}]
We first recall that the matrix $\M(z)$ is given by \eqref{defM} and that it is invertible for all $z$ satisfying $|z|>\exp(-\underline{\eta})$ (thanks to
Assumption \ref{hyp:2}). Hence $0$ will never be an eigenvalue of $\M(z)$ for the relevant values of $z$. Let us then observe that $\kappa \in \C
\setminus \{ 0 \}$ is an eigenvalue of $\M(z)$ for $z\in \overline{\mathcal{U}}$ if and only if $z$ and $\kappa$ satisfy the so-called dispersion
relation:
$$
\sum_{\ell=-r}^p \, \A_\ell(z) \, \kappa^\ell \, = \, 0 \, ,
$$
and the definition \eqref{defAl} of the functions $\A_\ell$ yields the equivalent form:
\begin{equation}
\label{relation-valeurs-propres}
\widehat{Q}_1(\kappa) \, z \, = \, \widehat{Q}_0(\kappa) \, .
\end{equation}
In particular, for any $z$ in the connected set $\Ubar \setminus \{ \underline{z}_1,\dots,\underline{z}_K \}$, $\M(z)$ has no eigenvalue on the unit circle
$\cercle$ for otherwise we would have $z=F(\kappa)$ for some $\kappa \in \cercle$ and $z \not \in \{ \underline{z}_1,\dots,\underline{z}_K \}$, which is
precluded by Assumption \ref{hyp:1}. To obtain the first statement of Lemma \ref{lem:1}, it thus remains to count the number of eigenvalues of $\M(z)$
in $\D \setminus \{ 0 \}$ (we shall call such eigenvalues the \emph{stable} ones). By the connectedness of $\Ubar \setminus \{ \underline{z}_1,\dots,\underline{z}_K
\}$, the number of stable eigenvalues of $\M(z)$ does not depend on $z \in \Ubar \setminus \{ \underline{z}_1,\dots,\underline{z}_K \}$. In order to compute
the precise number of such eigenvalues, we shall let $z$ tend to infinity and determine the asymptotic behavior of these eigenvalues. This asymptotic
behavior differs completely between the explicit and implicit cases (though the number of stable eigenvalues will be the same in both cases), which is
the reason why we now deal with those two cases separately.

\paragraph{The explicit case ($Q_1=I$).} The dispersion relation \eqref{relation-valeurs-propres} then reduces to:
\begin{equation}
\label{relation-valeurs-propres'}
z \, = \, \sum_{\ell=-r}^p \, a_{\ell,0} \, \kappa^\ell \, .
\end{equation}
If $r=0$, then there are no eigenvalues in $\D \setminus \{ 0 \}$ for any $z$ for otherwise there would be at least one eigenvalue in $\D \setminus \{ 0 \}$
for all $z \in \U$ and the triangle inequality in \eqref{relation-valeurs-propres'} would imply:
$$
| \, z \, | \, \le \, \sum_{\ell=0}^p \, | \, a_{\ell,0} \, | \, .
$$
which is impossible because $|z|$ can be arbitrarily large. The result is thus proved in the case $r=0$ so we assume $r \ge 1$ from now on (Assumption
\ref{hyp:2} then yields $a_{-r,0} \neq 0$). Following \cite{kreiss1} (see also \cite{jfcnotes} for the complete details), the number of eigenvalues of $\M(z)$
in $\D \setminus \{ 0 \}$ is computed by letting $z$ tend to infinity for in that case, all such (stable) eigenvalues of $\M(z)$ collapse to zero. Indeed, an
eigenvalue of $\M(z)$ in $\D \setminus \{ 0 \}$ cannot remain uniformly away from the origin for otherwise the right hand side of \eqref{relation-valeurs-propres'}
would remain bounded while the left hand side tends to infinity.

The final argument is the following (see \cite[Theorem 4.2.1]{hinrichsen2005mathematical} for a general statement). For any $z \in \Ubar \setminus \{
\underline{z}_1,\dots,\underline{z}_K \}$, the eigenvalues of $\M(z)$ are those $\kappa \neq 0$ such that:
$$
\kappa^r \, = \, \dfrac{1}{z} \, \sum_{\ell=-r}^p \, a_{\ell,0} \, \kappa^{r+\ell} \, ,
$$
which is just an equivalent way of writing \eqref{relation-valeurs-propres'}. Hence for $z$ large, the small eigenvalues of $\M(z)$ behave at the leading
order like the roots of the reduced equation:
$$
\kappa^r \, = \, \dfrac{a_{-r,0}}{z} \, ,
$$
and there are exactly $r$ distinct roots close to $0$ of that equation. Hence $\M(z)$ has $r$ eigenvalues in $\D \setminus \{ 0 \}$ for any $z \in \Ubar
\setminus \{ \underline{z}_1,\dots,\underline{z}_K \}$.
\bigskip

\paragraph{The implicit case ($Q_1 \neq I$).} We then know that $a_{-r,1} \neq 0$ and $a_{p,1} \neq 0$. Moreover, the function $\widehat{Q}_1$ satisfies
the index condition \eqref{indice}. By the residue theorem \cite{rudin}, this means that $\widehat{Q}_1$ has as many poles as roots in $\D$ and since it
only has a pole of order $r$ at $0$, we can conclude that $\widehat{Q}_1$ has $r$ roots in $\D \setminus \{ 0 \}$. Since $\kappa^r \, \widehat{Q}_1(\kappa)$
is a polynomial of degree $p+r$, we also conclude that $ \widehat{Q}_1$ has $p$ roots in $\U$, as already explained in the introduction.

From the definition \eqref{defAl}, we compute:
$$
\lim_{z \to \infty} \, \M(z) \, = \, \begin{bmatrix}
-\dfrac{a_{p-1,1}}{a_{p,1}} & \dots & \dots & -\dfrac{a_{-r,1}}{a_{p,1}} \\
1 & 0 & \dots & 0 \\
0 & \ddots & \ddots & \vdots \\
0 & 0 & 1 & 0 \end{bmatrix} \, ,
$$
and the eigenvalues of that (invertible) matrix are exactly those $\kappa \neq 0$ that satisfy $\widehat{Q}_1(\kappa)=0$. Hence for any sufficiently large
$z$, $\M(z)$ has $r$ eigenvalues in $\D \setminus \{ 0 \}$ and $p$ eigenvalues in $\U$ (which are close to the roots of $\widehat{Q}_1$). This completes
the proof of the first statement in Lemma \ref{lem:1}. It now remains to examine the situation at the points $\underline{z}_k$, $k=1,\dots,K$. The arguments
below are the same for the explicit and implicit cases so we stop distinguishing between the two from now on. We thus consider a point $\underline{z}_k$
for $1 \le k \le K$.
\bigskip

\paragraph{Case I.} We assume that the index $k \in \{ 1,\dots,K \}$ is such that the set $\mathcal{I}_k$ in \eqref{defIk} is the singleton $\{ k \}$ and
we assume for now $\alpha_k>0$ in \eqref{hyp:stabilite2}. Since the eigenvalues of $\M(\underline{z}_k)$ are the roots of the dispersion relation:
$$
\widehat{Q}_1(\kappa) \, \underline{z}_k \, = \, \widehat{Q}_0(\kappa) \, ,
$$
we first observe that the only eigenvalue of $\M(\underline{z}_k)$ on $\cercle$ is $\underline{\kappa}_k$ and we are now going to show that this
eigenvalue is algebraically (and therefore geometrically) simple. The relation \eqref{hyp:stabilite2} gives:
$$
F'(\underline{\kappa}_k) \, = \, - \, \dfrac{\underline{z}_k \, \alpha_k}{\underline{\kappa}_k} \neq 0 \, .
$$
Moreover, the characteristic polynomial of $\M(z)$ at $\kappa$ equals $z-F(\kappa)$ up to a nonvanishing holomorphic function of $(z,\kappa)$
close to $(\underline{z}_k,\underline{\kappa}_k)$. This means that $\underline{\kappa}_k$ is an algebraically simple eigenvalue of
$\M(\underline{z}_k)$ and can therefore be extended holomorphically with respect to $z$ in a sufficiently small neighborhood of $\underline{z}_k$.
We let $\kappa_k(z)$ denote this holomorphic extension, which satisfies $z=F(\kappa_k(z))$ for any $z$ close to $\underline{z}_k$. Performing a
Taylor expansion, we compute:
$$
\kappa_k \big( \underline{z}_k \, (1+\epsilon) \big) \, =\underline{\kappa}_k \left( \, 1 \, - \, \dfrac{\epsilon}{\alpha_k}\right) \, + \, O(\epsilon^2) \, .
$$
In particular, $\kappa_k \big( \underline{z}_k \, (1+\epsilon) \big)$ belongs to $\D$ for $\epsilon>0$ small enough.

To conclude, we observe that the $p+r-1$ eigenvalues of $\M(\underline{z}_k)$ which differ from $\underline{\kappa}_k$ lie in $\D \cup \U$.
Those eigenvalues remain in $\D \cup \U$ as $\underline{z}_k$ is perturbed into $\underline{z}_k \, (1+\epsilon)$ for a sufficiently small
$\epsilon>0$. Using the previous step of the analysis, we know that $\M \big( \underline{z}_k \, (1+\epsilon) \big)$ has $r$ eigenvalues
in $\D$ and $p$ eigenvalues in $\U$ so the reader will easily get convinced that the only possible situation for the location of the eigenvalues
of $\M(\underline{z}_k)$ is the one stated in Lemma \ref{lem:1}.
\bigskip

\paragraph{Cases II and III.} It remains to deal with the case where $\mathcal{I}_k$ is the singleton $\{ k \}$ and $\alpha_k<0$ (Case II), and
the final case where $\mathcal{I}_k$ has two elements (Case III). The argument for Case II is the same as for Case I except that now the Taylor
expansion of $\kappa_k$ shows that $\kappa_k \big( \underline{z}_k \, (1+\epsilon) \big)$ belongs to $\U$ for $\epsilon>0$ small enough.
The remaining details for that case are easily filled in. For Case III, $\M(\underline{z}_k)$ has two eigenvalues on $\cercle$, which are, in our
usual notation, $\underline{\kappa}_{\nu_{k,1}}$ and $\underline{\kappa}_{\nu_{k,2}}$. The same argument as in Case I or Case II shows that one
of these eigenvalues moves into $\D$ as $z$ is perturbed from $\underline{z}_k$ to $\underline{z}_k \, (1+\epsilon)$, and the other
eigenvalue moves into $\U$. This situation thus mixes Cases I and II. The conclusion follows and the proof of Lemma \ref{lem:1} is now complete.
\end{proof}

\subsection{Estimates for the spatial Green's function}

This section is devoted to the analysis of the solution to the recurrence relation \eqref{resolvent'}, which we recall is an equivalent
formulation of \eqref{resolvent}. More precisely, our aim is to derive pointwise estimates on the spatial Green's function $G_j(z)$.
We will divide the analysis depending on the position of $z$ in the complex plane. Away from the tangency points $\underline{z}_1,
\dots, \underline{z}_K$, we expect to obtain a uniform exponential decay while near the tangency points only some kind of local
boundedness is expected. We will rely on the spectral splitting given by Lemma~\ref{lem:1} to compute pointwise estimates of the
augmented vector $W_j(z)$.

We start with the estimates away from the tangency points.

\begin{lemma}[Bounds away from the tangency points]
\label{lem:2}
Let $\underline{z} \in \Ubar \setminus \{ \underline{z}_1,\dots,\underline{z}_K \}$. Then there exists an open ball $B_\delta (\underline{z})$,
$\delta>0$, centered at $\underline{z}$ and there exist two constants $C>0$, $c>0$ such that:
$$
\forall \, z \in B_\delta (\underline{z}) \, ,\quad \forall \, j \in \Z \, ,\quad \big| \, G_j(z) \, \big| \, \le \, C \, \exp \big( -c \, |j| \, \big) \, .
$$
\end{lemma}

\begin{proof}
We first introduce some notation. Let $\underline{z} \in \Ubar \setminus \{ \underline{z}_1,\dots,\underline{z}_K \}$ be fixed. We know that $\M(z)$ in
\eqref{defM} is well-defined and holomorphic in a sufficiently small neighborhood of $\underline{z}$ (including in the case where $\underline{z}$ belongs
to $\cercle$). Moreover, because of the spectral splitting shown in Lemma \ref{lem:1}, the matrix $\M(z)$ has no eigenvalue on $\cercle$ for $z$ close
to $\underline{z}$, and it has $r$, resp. $p$, eigenvalues in $\D$, resp. $\U$, for $z$ close to $\underline{z}$. Consequently, for $z$ close to $\underline{z}$,
the so-called stable subspace, which is spanned by the generalized eigenvectors of $\M(z)$ associated with eigenvalues in $\D$, has constant dimension
$r$. Similarly, the unstable subspace, which is spanned by the generalized eigenvectors of $\M(z)$ associated with eigenvalues in $\U$, has constant
dimension $p$. We let $\E^s(z)$, resp. $\E^u(z)$, denote the stable, resp. unstable, subspace of $\M(z)$ for $z$ close to $\underline{z}$. We have
the decomposition:
$$
\forall \, z \in B_\delta (\underline{z}) \, ,\quad \C^{p+r} \, = \, \E^s(z) \oplus \E^u(z) \, ,
$$
for some sufficiently small radius $\delta>0$. The associated projectors are denoted $\pi^s(z)$ and $\pi^u(z)$. The dynamics of \eqref{resolvent'} is therefore
of hyperbolic type for any $z \in B_\delta (\underline{z})$.

The projectors $\pi^s(z)$ and $\pi^u(z)$ are given by contour integrals. For instance, we have:
$$
\pi^s(z) \, = \, \dfrac{1}{2 \, \mathbf{i} \, \pi} \, \int_\gamma \, (w \, I \, - \, \M(z))^{-1} \, {\rm d}w \, ,
$$
where $\gamma$ is a contour that encloses the stable eigenvalues (those in $\D$) of $\M(z)$ (for instance, $\cercle$ is such a contour). A similar formula
holds for $\pi^u(z)$ with a contour that encloses the unstable eigenvalues. This formula shows that $\pi^s(z)$ depends holomorphically on $z$ in the ball
$B_\delta (\underline{z})$ and consequently, the stable and unstable subspaces $\E^s(z)$ and $\E^u(z)$ depend holomorphically\footnote{Following the
analysis of spectral projectors in \cite{Kato}, we shall say that a vector space $\E(z) \subset \C^N$ that depends on a complex variable $z$ for $z$ in an
open set $\mathcal{O} \subset \C$ and that has constant dimension $n$, depends holomorphically on $z$ if, for any $\underline{z} \in \mathcal{O}$, there
exists a neighborhood $\mathcal{V}$ of $\underline{z}$ in $\mathcal{O}$ and a basis $e_1(z), \dots, e_n(z)$ of $\E(z)$ that depends holomorphically on
$z$ in $\mathcal{V}$. This amounts to saying that the vector bundle defined by $\E$ over $\mathcal{O}$ is holomorphic. A typical example is the case
$\E(z) =P(z) \, \C^N$ where $P(z)$ is a projector on $\C^N$ that depends holomorphically on $z$ in an open set $\mathcal{O}$.} on $z$.

Up to restricting $\delta$, any complex number $z$ in the open ball $B_\delta (\underline{z})$ lies in the resolvent set of the operator $\mathcal{L}$,
hence there exists a unique solution $(W_j(z))_{j \in \Z} \in \ell^2(\Z;\C^{p+r})$ to \eqref{resolvent'}. Since the dynamics of the iteration \eqref{resolvent'} for
such $z$ enjoys a hyperbolic dichotomy, the solution to \eqref{resolvent'} is given by integrating either from $j$ to $+\infty$, or from $-\infty$ to $j-1$,
depending on whether we compute the unstable or stable components of the vector $W_j(z)$. This leads to the expression:
\begin{equation}
\label{composanteinstablej}
\forall \, j \in \Z \, ,\quad \pi^u(z) \, W_j(z) \, = \, - \, \dfrac{1}{\A_p(z)} \,
\sum_{\ell \ge 0} \, \, (Q_1 \, \boldsymbol{\delta})_{j+\ell} \, \, \M(z)^{-1-\ell} \, \pi^u(z) \, {\bf e} \, ,
\end{equation}
for the unstable components, and to the expression:
\begin{equation}
\label{composantestablej}
\forall \, j \in \Z \, ,\quad \pi^s(z) \, W_j(z) \, = \, \dfrac{1}{\A_p(z)} \,
\sum_{\ell =-\infty}^{j-1} \, \, (Q_1 \, \boldsymbol{\delta})_\ell \, \, \M(z)^{j-1-\ell} \, \pi^s(z) \, {\bf e} \, ,
\end{equation}
for the stable components.

At this stage, we observe that the sequence $Q_1 \, \boldsymbol{\delta}$ only has finitely many nonzero coefficients, which are given by:
$$
\forall \, j \in \Z \, ,\quad (Q_1 \, \boldsymbol{\delta})_j \, = \, \begin{cases}
a_{-j,1} \, ,& \text{\rm if $j \in \{ -p,\dots,r \}$,} \\
0 \, ,& \text{\rm otherwise.}
\end{cases}
$$
Hence we see from \eqref{composanteinstablej} that $\pi^u(z) \, W_j(z)$ vanishes for $j \ge r+1$, and we see from \eqref{composantestablej}
that $\pi^s(z) \, W_j(z)$ vanishes for $j \le -p$. For $j \le r$, we get:
$$
\pi^u(z) \, W_j(z) \, = \, - \, \dfrac{1}{\A_p(z)} \, \sum_{\ell =\max(-p-j,0)}^{r-j} \, a_{-j-\ell,1} \, \M(z)^{-1-\ell} \, \pi^u(z) \, {\bf e} \, ,
$$
and since the sequence $(\M(z)^{-\ell} \, \pi^u(z))_{\ell \ge 1}$ is exponentially decreasing, uniformly with respect to $z \in B_\delta (\underline{z})$,
we get the uniform bound\footnote{Here we also use Assumption \ref{hyp:2} to get a uniform local bound for $\A_p(z)^{-1}$, including in the case
$\underline{z} \in \cercle$ for which $z$ can come inside the unit disk.}:
\begin{equation}
\label{borneslem2-1}
\forall \, z \in B_\delta (\underline{z}) \, ,\quad \forall \, j \in \Z \, ,\quad \Big| \, \pi^u (z) \, W_j(z) \, \Big| \, \le \, \begin{cases}
0  \, ,& \text{\rm if $j \, \ge \, r+1$,} \\
C \, \exp (- \, c \, |j|) \, ,& \text{\rm if $j \, \le \, r$.}
\end{cases}
\end{equation}

Similar arguments, using the uniform exponential decay of the sequence $(\M(z)^{\ell} \, \pi^s(z))_{\ell \ge 1}$, yield the bound:
\begin{equation}
\label{borneslem2-2}
\forall \, z \in B_\delta (\underline{z}) \, ,\quad \forall \, j \ge 1 \, ,\quad \Big| \, \pi^s (z) \, W_j(z) \, \Big| \, \le \, \begin{cases}
0 \, ,& \text{\rm if $j \, \le \, -p-1$,} \\
C \, \exp (- \, c \, |j|) \, ,& \text{\rm if $j \, \ge -p$.}
\end{cases}
\end{equation}
Adding \eqref{borneslem2-1} and \eqref{borneslem2-2} gives the claim of Lemma \ref{lem:2} since the spatial Green's function $G_j(z)$ is just
one coordinate of the vector $W_j(z) \in \C^{p+r}$.
\end{proof}

We are now going to examine the behavior of the spatial Green's function $G(z)$ close to any of the points $\underline{z}_k$, $k=1,\dots,K$,
where the spectrum of $\mathcal{L}$ is tangent to the unit circle. Let us first recall that the exterior $\U$ of the unit disk belongs to the resolvent
set of $\mathcal{L}$ hence the spatial Green's function $G(z)$ is well-defined in the ``half-ball'' $B_\delta (\underline{z}_k) \cap \U$ for any radius
$\delta>0$. Our goal below is to extend holomorphically $G(z)$ to a whole neighborhood of $\underline{z}_k$ for each $k$, which amounts to
passing through the (essential) spectrum of $\mathcal{L}$. Our results are the following two lemmas.

\begin{lemma}[Bounds close to the tangency points -- Cases I and II]
\label{lem:3}
Let $k \in \{1,\dots,K\}$ be such that the set $\mathcal{I}_k$ in \eqref{defIk} is the singleton $\{ k \}$. Then there exists an open ball $B_\epsilon
(\underline{z}_k)$ and there exist two constants $C>0$ and $c>0$ such that, for any integer $j \in \Z$, the component $G_j(z)$ defined on
$B_\epsilon(\underline{z}_k) \cap \U$ extends holomorphically to the whole ball $B_\epsilon(\underline{z}_k)$ with respect to $z$,
and the holomorphic extension satisfies the bound:
$$
\forall \, z \in B_\epsilon(\underline{z}_k) \, ,\quad \forall \, j \in \Z \, ,\quad \big| \, G_j(z) \, \big| \, \le \, \begin{cases}
C \, \exp \big( -c \, |j| \, \big) \, ,& \text{\rm if $j \, \le \, 0$,} \\
C \, \Big| \kappa_k(z) \Big|^{j} \, ,& \text{\rm if $j \, \ge \, 1$,}
\end{cases} \quad \text{\rm if $\alpha_k>0 \quad$ (Case I),}
$$
and
$$
\forall \, z \in B_\epsilon(\underline{z}_k) \, ,\quad \forall \, j \in \Z \, ,\quad \big| \, G_j(z) \, \big| \, \le \, \begin{cases}
C \, \Big| \kappa_k(z) \Big|^{|j|} \, ,& \text{\rm if $j \, \le \, 0$,} \\
C \, \exp \big( -c \, j \, \big) \, ,& \text{\rm if $j \, \ge \, 1$,}
\end{cases} \quad \text{\rm if $\alpha_k<0 \quad$ (Case II),}
$$
where, in either case, $\kappa_k(z)$ denotes the (unique) holomorphic eigenvalue of $\M(z)$ defined close to $\underline{z}_k$ and that satisfies
$\kappa_k(\underline{z}_k)=\underline{\kappa}_k$.
\end{lemma}

\begin{lemma}[Bounds close to the tangency points -- Case III]
\label{lem:4}
Let now $k \in \{1,\dots,K\}$ be such that the set $\mathcal{I}_k$ in \eqref{defIk} has two elements $\{ \nu_{k,1},\nu_{k,2} \}$ which are fixed by the convention
$\alpha_{\nu_{k,1}}<0<\alpha_{\nu_{k,2}}$. Then there exists an open ball $B_\epsilon (\underline{z}_k)$ centered at $\underline{z}_k$ and there exists a
constant $C>0$ such that, for any integer $j \in \Z$, the component $G_j(z)$ defined on $B_\epsilon(\underline{z}_k) \cap \U$ extends
holomorphically to the whole ball $B_\epsilon(\underline{z}_k)$ with respect to $z$, and the holomorphic extension satisfies the bound:
$$
\forall \, z \in B_\epsilon(\underline{z}_k) \, ,\quad \forall \, j \in \Z \, ,\quad \big| \, G_j(z) \, \big| \, \le \, \begin{cases}
C \, \Big| \kappa_{\nu_{k,1}}(z) \Big|^{|j|} \, , & \text{\rm if $j \, \le \, 0$,} \\
C \, \Big| \kappa_{\nu_{k,2}}(z) \Big|^{j} \, ,& \text{\rm if $j \, \ge \, 1$,}
\end{cases} \qquad \text{\rm (Case III),}
$$
where $\kappa_{\nu_{k,1}}(z)$, resp. $\kappa_{\nu_{k,2}}(z)$, denotes the (unique) holomorphic eigenvalue of $\M(z)$ defined close to $\underline{z}_k$
and that satisfies $\kappa_{\nu_{k,1}}(\underline{z}_k)=\underline{\kappa}_{\nu_{k,1}}$, resp. $\kappa_{\nu_{k,2}}(\underline{z}_k)=\underline{\kappa}_{\nu_{k,2}}$.
\end{lemma}

\noindent The proofs of Lemma \ref{lem:3} and Lemma \ref{lem:4} are mostly identical so we just give the proof of Lemma \ref{lem:3} and indicate the minor
refinements for the proof of Lemma \ref{lem:4}.

\begin{proof}[Proof of Lemma \ref{lem:3}]
Most ingredients of the proof are similar to what we have already done in the proof of Lemma \ref{lem:2}. We assume from now on $\alpha_k>0$,
the case $\alpha_k<0$ being left to the interested reader. We just need to slightly adapt the notation used in the proof of Lemma \ref{lem:2} since
the hyperbolic dichotomy of $\M(z)$ does not hold any longer in a whole neighborhood of $\underline{z}_k$. Since $\underline{\kappa}_k$ is a
simple eigenvalue of $\M(\underline{z}_k)$, we can extend it holomorphically to a simple eigenvalue $\kappa_k(z)$ of $\M(z)$ in a neighborhood
of $\underline{z}_k$. This eigenvalue is associated with the eigenvector:
$$
E_k(z) \, := \, \begin{bmatrix}
\, \kappa_k(z)^{p+r-1} \, \\
\vdots \\
\kappa_k(z) \\
1 \end{bmatrix} \in \C^{p+r} \, ,
$$
which also depends holomorphically on $z$ in a neighborhood of $\underline{z}_k$. The vector $E_k(z)$ contributes to the stable subspace of
$\M(z)$ for $z \in \U$ close to $\underline{z}_k$ but the situation is unclear as $z$ goes inside $\D$ (it actually depends on the position of $z$ with
respect to the spectrum of $\mathcal{L}$). The remaining $p+r-1$ eigenvalues of $\M(z)$ enjoy the now familiar hyperbolic dichotomy, uniformly
with respect to $z$ close to $\underline{z}_k$. We let below $\E^{ss}(z)$, resp. $\E^u(z)$, denote the strongly stable, resp. unstable, subspace of
$\M(z)$ associated with those eigenvalues that remain uniformly inside $\D$, resp. $\U$, as $z$ belongs to a neighborhood of $\underline{z}_k$.
In particular, $\E^{ss}(z)$, resp. $\E^u(z)$, has dimension $r-1$, resp. $p$, thanks to Lemma \ref{lem:1}, and we have the decomposition:
\begin{equation}
\label{decomposition-lem3}
\forall \, z \in B_\epsilon(\underline{z}_k) \, ,\quad \C^{p+r} \, = \, \E^{ss}(z) \, \oplus \, \text{\rm Span } E_k(z) \, \oplus \, \E^u(z) \, ,
\end{equation}
for a sufficiently small radius $\epsilon>0$. We let below $\pi^{ss}(z)$, $\pi_k(z)$ and $\pi^u(z)$ denote the holomorphic projectors associated
with the decomposition \eqref{decomposition-lem3}.
\bigskip

We first consider a point $z \in B_\epsilon(\underline{z}_k) \cap \U$ so that the decomposition \eqref{decomposition-lem3} holds and the Green's
function $G(z) \in \ell^2(\Z;\C)$ is well-defined as the only solution to \eqref{resolvent}. We use the equivalent formulation \eqref{resolvent'} and derive
the following expressions that are entirely similar to those found in the proof of Lemma \ref{lem:2}:
\begin{subequations}
\begin{align}
\forall \, j \in \Z \, ,\quad \pi^u(z) \, W_j(z) \, &= \, - \, \dfrac{1}{\A_p(z)} \,
\sum_{\ell \ge 0} \, \, (Q_1 \, \boldsymbol{\delta})_{j+\ell} \, \, \M(z)^{-1-\ell} \, \pi^u(z) \, {\bf e} \, , \label{composanteinstablej'} \\
\pi^{ss}(z) \, W_j(z) \, &= \, \dfrac{1}{\A_p(z)} \,
\sum_{\ell =-\infty}^{j-1} \, \, (Q_1 \, \boldsymbol{\delta})_\ell \, \, \M(z)^{j-1-\ell} \, \pi^{ss}(z) \, {\bf e} \, , \label{composantestablej'} \\
\pi_k(z) \, W_j(z) \, &= \, \dfrac{1}{\A_p(z)} \,
\sum_{\ell =-\infty}^{j-1} \, \, (Q_1 \, \boldsymbol{\delta})_\ell \, \, \kappa_k(z)^{j-1-\ell} \, \pi_k(z) \, {\bf e} \, . \label{composantescalairej}
\end{align}
\end{subequations}

The strongly stable ($\pi^{ss}(z) W_j(z)$) and unstable ($\pi^u(z) W_j(z)$) components obviously extend holomorphically to the whole neighborhood
$B_\epsilon(\underline{z}_k)$ of $\underline{z}_k$ since the projectors $\pi^{ss}(z)$ and $\pi^u(z)$ depend holomorphically on $z$ on that set
and the sums on the right hand side of \eqref{composanteinstablej'} and \eqref{composantestablej'} are, at most, finite. Furthermore, by using the
same type of bounds as in the proof of Lemma \ref{lem:2}, we obtain:
\begin{equation}
\label{borneslem3-1}
\forall \, z \in B_\epsilon(\underline{z}_k) \, ,\quad \forall \, j \in \Z \, ,\quad \big| \, \pi^u(z) \, W_j(z) \, + \, \pi^{ss}(z) \, W_j(z) \, \big| \, \le \,
C \, \exp (- \, c \, |j|) \, ,
\end{equation}
for some appropriate constants $C>0$ and $c>0$. We now focus on the vector $\pi_k(z) \, W_j(z)$ in \eqref{composantescalairej} which is aligned
with the eigenvector $E_k(z)$. We see from \eqref{composantescalairej} that $\pi_k(z) \, W_j(z)$ vanishes for $j \le -p$. For $j$ in the finite set
$\{ -p+1,\dots,r\}$, we have:
$$
\pi_k(z) \, W_j(z) \, = \, \dfrac{1}{\A_p(z)} \, \sum_{\ell =-p}^{j-1} \, a_{-\ell,1} \, \, \kappa_k(z)^{j-1-\ell} \, \pi_k(z) \, {\bf e} \, ,
$$
and for $j \ge r+1$, we have:
$$
\pi_k(z) \, W_j(z) \, = \, \dfrac{1}{\A_p(z)} \, \sum_{\ell =-p}^r \, a_{-\ell,1} \, \, \kappa_k(z)^{j-1-\ell} \, \pi_k(z) \, {\bf e} \, .
$$
In either case, we see that the component $\pi_k(z) \, W_j(z)$ extends holomorphically to the whole neighborhood $B_\epsilon(\underline{z}_k)$
of $\underline{z}_k$ and we have a bound of the form\footnote{Here we use again that $\A_p(z)$ does not vanish in the ball $B_\epsilon
(\underline{z}_k)$ up to restricting the radius $\epsilon$.}:
\begin{equation}
\label{borneslem3-2}
\forall \, z \in B_\epsilon(\underline{z}_k) \, ,\quad \forall \, j \in \Z \, ,\quad \big| \, \pi_k(z) \, W_j(z) \, \big| \, \le \,
\begin{cases}
\, 0 \, ,& \text{\rm if $j \le -p$,} \\
\, C \, |\kappa_k(z)|^j \, ,& \text{\rm if $j \ge 1-p$.}
\end{cases}
\end{equation}

In order to conclude, we can always assume that the ball $B_\epsilon(\underline{z}_k)$ is so small that the modulus $|\kappa_k(z)|$
belongs to the interval\footnote{The constant $c$ here refers to the same one as in \eqref{borneslem3-1}.} $[\exp (-c),\exp c]$ (it equals $1$ at
$\underline{z}_k$). It then remains to add the bounds in \eqref{borneslem3-1} and \eqref{borneslem3-2} and to compare which is the largest.
This completes the proof of Lemma \ref{lem:3} in the case $\alpha_k>0$. The remaining Case II ($\alpha_k<0$) is handled similarly except
that now the eigenvector $E_k(z)$ contributes to the unstable subspace of $\M(z)$ for $|z|>1$. The minor modifications are left to the reader.
\end{proof}

\begin{proof}[Proof of Lemma \ref{lem:4}]
The proof of Lemma \ref{lem:4} is a mixture between Cases I and II in which now $\M(z)$ has two (holomorphic) eigenvalues whose modulus
equals $1$ at $\underline{z}_k$. One contributes to the stable subspace of $\M(z)$ and the other one contributes to the unstable subspace of
$\M(z)$ for $|z|>1$. The same ingredients as in the proof of Lemma \ref{lem:3} can then be applied with minor modifications.
\end{proof}

Let us remark that all the claims in Lemma \ref{lem:2}, Lemma \ref{lem:3} and Lemma \ref{lem:4} do not distinguish between the explicit and
implicit case since they only rely on Lemma \ref{lem:1}. In the implicit case, for the forthcoming estimates of the temporal Green's function of
Section~\ref{section3}, we will also need to obtain bounds of the spatial Green's function $G(z)$ for large values of $z$. These bounds are
provided by the following result.

\begin{lemma}[Bounds at infinity -- Implicit case]
\label{lem:5}
If $Q_1$ is not the identity, then there exist a radius $R \ge 2$ and two constants $C>0$, $c>0$ such that there holds:
$$
\forall \, z \not \in B_R(0) \, ,\quad \forall \, j \in \Z \, ,\quad \big| \, G_j(z) \, \big| \, \le \, C \, \exp \big( -c \, |j| \, \big) \, .
$$
\end{lemma}

\begin{proof}
The proof is basically the same as that of Lemma \ref{lem:2}. Indeed, we recall that in the implicit case, the matrix $\M(z)$ has a limit at infinity,
given by:
$$
\begin{bmatrix}
-\dfrac{a_{p-1,1}}{a_{p,1}} & \dots & \dots & -\dfrac{a_{-r,1}}{a_{p,1}} \\
1 & 0 & \dots & 0 \\
0 & \ddots & \ddots & \vdots \\
0 & 0 & 1 & 0 \end{bmatrix} \, ,
$$
and this matrix has a hyperbolic dichotomy because of Assumption \ref{hyp:0}. We can thus apply the same argument as in the proof of Lemma
\ref{lem:2} for $z$ in a neighborhood of infinity. In particular, we can use the fact that the sequences $(\M(z)^{-\ell} \, \pi^u(z))_{\ell \ge 1}$ and
$(\M(z)^\ell \, \pi^s(z))_{\ell \ge 1}$ are exponentially decreasing, uniformly with respect to $z$ in a neighborhood $B_R(0)^c$ of infinity. The
conclusion of Lemma \ref{lem:5} follows.
\end{proof}

Let us observe that we can extend holomorphically each scalar component $G_j(z)$ to a neighborhood $B_\epsilon(\underline{z}_k)$ of
$\underline{z}_k$, but that does not necessarily mean that the extended sequence $G(z)$ lies in $\ell^2(\Z;\C)$. For instance, in Case I of
Lemma \ref{lem:3}, the eigenvalue $\kappa_k(z)$ contributes to the stable subspace of $\M(z)$ for $|z|>1$ but it starts contributing to the
unstable subspace of $\M(z)$ as $z$ crosses the spectrum of $\mathcal{L}$ (which coincides with the curve $F(\cercle)$). Hence the
holomorphic extension $G(z)$ ceases to be in $\ell^2(\Z;\C)$ as $z$ crosses the spectrum of $\mathcal{L}$ for it then has an exponentially growing
mode in $j$ at $+\infty$.

We finally end this section with the following corollary, which is a direct consequence of Lemma \ref{lem:2},  Lemma \ref{lem:3}, Lemma
\ref{lem:4} and Lemma \ref{lem:5} above by applying a standard compactness argument. We refer to Figure~\ref{fig:recouv} for a geometrical 
representation of this result.

\begin{corollary}
\label{cor1}
There exists some $\epsilon_\star>0$ such that for each $\epsilon \in(0,\epsilon_\star)$ there exists some $\delta_\epsilon>0$ such that the 
Green's function $G(z)$, defined initially as the unique solution $G(z)\in\ell^2(\Z;\C)$ to \eqref{resolvent} for each $z$ in the resolvent set of 
$\mathcal{L}$, has a unique holomorphic extension (also denoted $G(z)$) to the set
\bqs
\mathcal{S} \, := \, \Big\{ \zeta \in \C \, | \, {\rm e}^{- \, \delta_\epsilon} \, < \, |\zeta| \, \le \, {\rm e}^\pi \, \Big\} \cup \bigcup_{k=1}^K B_\epsilon(\underline{z}_k).
\eqs
Moreover there are constants $C>0$ and $c>0$ such that
\begin{itemize}
\item Whenever $z\in \mathcal{S} \, \setminus \, \bigcup_{k=1}^K B_\epsilon(\underline{z}_k)$,
\bqq
\forall \, j \in \Z \, ,\quad \big| \, G_j(z) \, \big| \, \le \, C \, \exp \big( -c \, |j| \, \big) \, .
\label{uniformexpbound}
\eqq
\item Whenever $z\in B_\epsilon(\underline{z}_k)$ for $k=1,\cdots,K$, the Green's function $G(z)$ satisfies one of the bounds in Lemma
\ref{lem:3} or \ref{lem:4} depending on the cardinal of $\mathcal{I}_k$ and the sign of $\alpha_k$.
\end{itemize}
Furthermore, in the implicit case, the above uniform exponential bound \eqref{uniformexpbound} extends to all $|z| \geq {\rm e}^\pi$.
\end{corollary}

\begin{figure}[t!]
  \centering
\includegraphics[width=.4\textwidth]{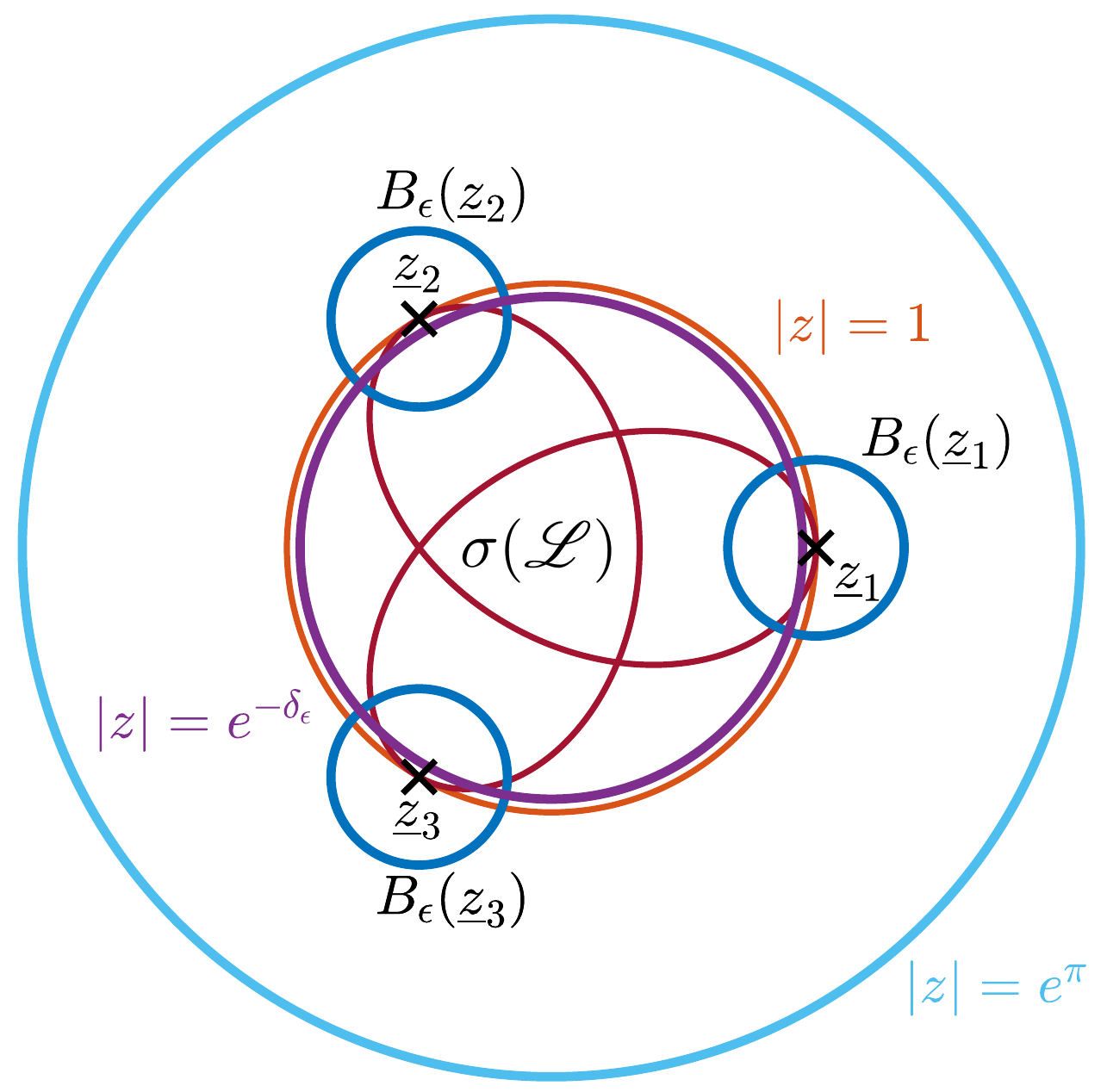}\hspace{2cm}
\includegraphics[width=.35\textwidth]{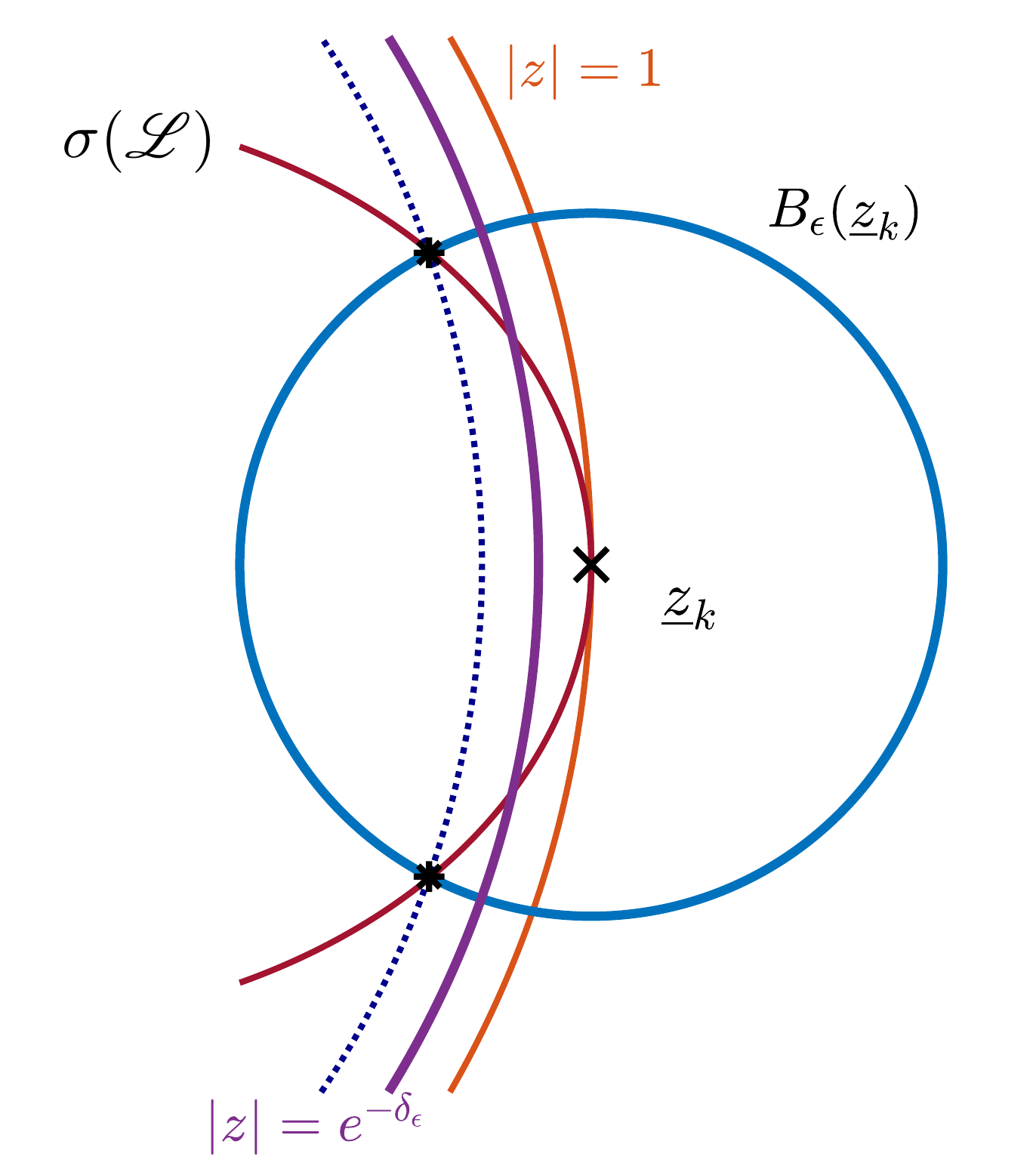}
  \caption{Geometrical illustrations of the set $\mathcal{S}$ given in Corollary~\ref{cor1} in the case $K=3$. The curve of essential spectrum
  $\sigma(\mathscr{L})$ (dark red curve) is tangent to the unit circle precisely at $\underline{z}_k$ for $k=1,2,3$ (black crosses) and otherwise
  strictly contained in the unit disk $\D$. For $\epsilon \in(0,\epsilon_\star)$, each ball $B_\epsilon(\underline{z}_k)$ intersects $\sigma(\mathscr{L})$
  at two points (black stars in the right panel) and the dashed dark blue line represents the circle passing through the point with largest modulus.
  We then fix $\delta_\epsilon>0$ such that the circle $\{ z \, | \, |z|={\rm e}^{-\delta_\epsilon} \}$ (magenta curve) is contained in between the unit
  circle and the circle passing through the point with largest modulus.}
  \label{fig:recouv}
\end{figure}

\section{The temporal Green's function}
\label{section3}

The starting point of the analysis is to use the inverse Laplace transform formula to  express the so-called Green's function $\G^{n} := \mathcal{L}^n
\, \boldsymbol{\delta}$ as the following contour integral
\begin{equation}
\forall \, n \in \N^* \, ,\quad \forall \, j \in \Z \, ,\quad \G^{n}_j \, = \, \left(\mathcal{L}^n \, \boldsymbol{\delta}\right)_j \, = \,
\dfrac{1}{2 \, \mathbf{i} \, \pi} \, \int_{\widetilde\Gamma} z^n \, G_j(z) \, \md z \, ,
\label{lapgreenz}
\end{equation}
where $\widetilde{\Gamma}$ is a closed curve in the complex plane surrounding the unit disk $\D$ and lying in the resolvent set of $\mathcal{L}$. The
idea will be to deform $\widetilde{\Gamma}$ in order to obtain sharp pointwise estimates on the temporal Green's function using our pointwise estimates
on the spatial Green's function given in Lemma~\ref{lem:2}, Lemma~\ref{lem:3}, Lemma~\ref{lem:4} and Lemma~\ref{lem:5} above. To do so, we first
change variable in \eqref{lapgreenz}, by setting $z=\exp(\tau)$, such that we get
\bqq
\G^{n}_j \, = \, \dfrac{1}{2 \, \mathbf{i} \, \pi} \, \int_{\Gamma} {\rm e}^{n \, \tau} \, \mathbf{G}_j(\tau) \, \md \tau \, ,
\label{lapgreentau}
\eqq
where without loss of generality $\Gamma=\left\{ s+\mathbf{i} \, \ell ~|~ \ell \in[-\pi,\pi]\right\}$ for some $s>0$ (and actually \emph{any} $s>0$ thanks to
Cauchy's formula), and $\mathbf{G}_j(\tau)$ is given by
\bqs
\forall \, j \in \Z \, ,\quad \mathbf{G}_j(\tau) \, := \, G_{j}({\rm e}^{\tau}) \, {\rm e}^{\tau} \, .
\eqs
The remaining of this section is devoted to the proof of Theorem~\ref{thm:1}. For the sake of clarity, we first treat the explicit case with $K=1$, and then
deal with the implicit case still with $K=1$. And finally, we explain how the results generalize to $K>1$ in both cases. (Let us recall that $K$ denotes the
number of tangency points of the spectrum of $\mathcal{L}$ within the unit circle $\cercle$).

In the explicit case (when $Q_1$ is the identity), the analysis below uses the fact that each velocity $\alpha_k$ lies in the open interval $(-p,r)$. This fact is
stated in the following lemma, which can be seen as a variation on the so-called Courant-Friedrichs-Lewy condition \cite{cfl} and/or the Bernstein inequality
for trigonometric polynomials. A proof of Lemma \ref{lem:Bernstein} is provided in Appendix \ref{appendix}.

\begin{lemma}
\label{lem:Bernstein}
Under Assumptions \ref{hyp:0} and \ref{hyp:1}, if $Q_1$ is the identity, then there holds:
$$
\forall \, k \, = \, 1,\dots,K \, ,\quad - \, p \, < \, \alpha_k \, < \, r \, ,
$$
where the $\alpha_k$'s are the drift velocities arising in \eqref{hyp:stabilite2}.
\end{lemma}

\noindent We now deal with the proof of Theorem \ref{thm:1}.

\subsection{The explicit case with $K=1$}

We first remark that, since $\mathcal{L}=Q_0$ is a convolution operator with finite stencil, for each $n\geq 1$, there holds
\bqs
\G^n_j \, = \, 0 \, , \, \text{ for } \, j \, > \, r \, n \, \text{ or } \, j \, < \, - \, p \, n \, .
\eqs
As a consequence, throughout this section, we assume that $j$ and $n$ satisfy
\bqs
n \ge 1 \, ,\quad - \, p \, n \, \leq \, j \, \leq \, r \, n \, .
\eqs
We also assume without loss of generality that $\underline{\kappa}_1=\underline{z}_1=1$ together with $\alpha_1>0$ (the case $\alpha_1<0$ being
handled similarly). In that case, we have from \eqref{hyp:stabilite2} that
\bqq
F\Big( {\rm e}^{\, \mathbf{i} \, \xi} \Big) \, = \,
\exp \left( \, - \, \mathbf{i} \, \alpha \, \xi \, - \, \beta \, \xi^{\, 2 \, \mu} \, + \, O \Big( \xi^{\, 2 \, \mu+1} \Big) \right) \, \text{ as } \xi \rightarrow 0,
\label{hyp:stab2mod}
\eqq
where we dropped the index $1$ to simplify our notations. Now, using Lemma~\ref{lem:3}, bounds close to the tangency point $z=1$ for $G_j(z)$ translate
into bounds near the origin $\tau=0$ for $\mathbf{G}_j(\tau)$. More precisely, we have the following lemma which combines Lemma~\ref{lem:3} and
Corollary~\ref{cor1}.

\begin{lemma}
\label{lem:3mod}
There exist $\epsilon_*>0$ and two constants $0<\beta_*<\Re(\beta)<\beta^*$ such that for each $\epsilon\in(0,\epsilon_*)$ there exist some width
$\eta_\epsilon>0$ together with two constants, still denoted $C>0$, $c>0$, such that, for any integer $j\in\Z$, the component $\mathbf{G}_j(\tau)$
extends holomorphically on $B_\epsilon(0)$ with bounds:
\bqq
\forall \, \tau \in B_\epsilon(0) \, ,\quad \forall \, j \in \Z \, ,\quad \big| \, \mathbf{G}_j(\tau) \, \big| \, \le \, \begin{cases}
C \, \exp \big( - \, c \, |j| \, \big) \, ,& \text{\rm if $j \, \le \, 0$,} \\
C \, \exp \big( \, j \, \Re(\varpi(\tau)) \big) \, ,& \text{\rm if $j \, \ge \, 1$,}
\end{cases}
\label{boundnear0}
\eqq
where $\varpi$ is holomorphic on $B_\epsilon(0)$ and has the Taylor expansion:
\bqq
\varpi(\tau) \, = \, - \, \dfrac{1}{\alpha} \, \tau \, + \, (-1)^{\mu+1} \, \dfrac{\beta}{\alpha^{2\, \mu+1}} \, \tau^{2 \, \mu} \,
+ \, O \left( |\tau|^{2 \, \mu+1} \right)\,, \quad \, \forall \, \tau \in B_\epsilon(0) \, ,
\label{eqtaylorexp}
\eqq
together with
\bqq
\Re(\varpi(\tau)) \, \leq \, - \, \dfrac{1}{\alpha} \, \Re(\tau) \, + \, \dfrac{\beta^*}{\alpha^{2\, \mu+1}} \, \Re(\tau)^{2 \, \mu} \, -  \,
\dfrac{\beta_*}{\alpha^{2\, \mu+1}} \, \Im(\tau)^{2 \, \mu} \, , \quad \, \forall \, \tau \in B_\epsilon(0) \, .
\label{eqestimatebeta}
\eqq
Furthermore, we have
\bqq
\forall \, \tau \in \Omega_\epsilon \, := \, \left\{ \, - \, \eta_\epsilon \, < \, \Re(\tau) \, \leq \, \pi \, \right\} \backslash \, B_\epsilon(0) \, ,\quad
\forall \, j \in \Z \, ,\quad \big| \, \mathbf{G}_j(\tau) \, \big| \, \le \, C \, \exp \big( -c \, |j| \, \big) \, .
\label{boundaway0}
\eqq
\end{lemma}

\begin{proof}
The first part of the proof simply relies on writing $\kappa(z)=\exp(\omega(z))$ near $z=1$ and using $z=\exp(\tau)$, such that after identification
we have $\varpi(\tau) = \omega(\exp(\tau))$. Indeed, the function $\kappa$ is holomorphic in the ball $B_{\epsilon_0}(1)$ for some $\epsilon_0>0$.
Upon reducing the size of $\epsilon_0$, we can define a holomorphic function $\omega: B_{\epsilon_0}(1) \to \C$ such that $\kappa(z)=\exp(\omega(z))$
for each $z\in B_{\epsilon_0}(1)$, and $\omega(1)=0$. We now let $\epsilon_*>0$ small enough be such that for each $\epsilon\in(0,\epsilon_*)$ and
$\tau \in B_\epsilon(0)$, we have $\exp(\tau) \in B_{\epsilon_0}(1)$. We can now define $\varpi:B_\epsilon(0)\to\C$ as $\varpi(\tau) := \omega(\exp(\tau))$
which is holomorphic in $B_\epsilon(0)$ by construction. Finally, we remark that $\mathbf{G}_j(\tau)$ extends holomorphically on $B_\epsilon(0)$ for
any $j \in \Z$ since $G_j(z)$ extends holomorphically on $B_{\epsilon_0}(1)$.

Next, we explain how to use the expansion \eqref{hyp:stab2mod} to obtain the desired Taylor expansion \eqref{eqtaylorexp} for $\varpi(\tau)$ near $\tau=0$.
We first remark that for each $\epsilon\in(0,\epsilon_*)$ and $\tau\in B_\epsilon(0)$ we have the identity
\bqs
{\rm e}^\tau=F(\kappa({\rm e}^\tau))=F({\rm e}^{\varpi(\tau)}).
\eqs
As a consequence, we get the expansion
\bqs
\tau = -\alpha \varpi(\tau)+\beta(-1)^{\mu+1}\varpi(\tau)^{2 \, \mu}+O\left( |\varpi(\tau)|^{2 \, \mu+1}\right),
\eqs
as $\tau \to 0$. Since $\varpi$ is holomorphic in $B_\epsilon(0)$ with $\varpi(0)=0$, we can use the above equality to obtain, by identification,  each term
of its Taylor expansion and recover \eqref{eqtaylorexp}. Note that we can always reduce the size of $\epsilon_*$ such that the expansion is valid for each
$\epsilon\in(0,\epsilon_*)$ and $\tau$ in $B_\epsilon(0)$.

To complete the proof we now prove the existence of two positive real numbers $\beta_*$ and $\beta^*$ verifying $0<\beta_*<\Re(\beta)<\beta^*$ such that
inequality \eqref{eqestimatebeta} holds true in $B_\epsilon(0)$ for each $\epsilon\in(0,\epsilon_*)$. First we compute
\begin{align*}
(-1)^{\mu+1} \, \Re \left( \beta \tau^{2\, \mu}\right) &= (-1)^{\mu+1} \, \Re(\beta) \, \Re\left(\tau^{2\, \mu}\right) \, - \,
(-1)^{\mu+1} \, \Im(\beta) \, \Im\left(\tau^{2 \, \mu}\right)\\
&= - \, \Re(\beta) \, \Im(\tau)^{2 \, \mu} \, - \, (-1)^\mu \, \Re(\beta) \, \Re(\tau)^{2 \, \mu} \\
&~~~ - \Re(\beta) \, \sum_{m=1}^{\mu-1} \, (-1)^m \, \left(
\begin{matrix} 2 \, \mu \\ 2 \, m \end{matrix} \right) \, \Re(\tau)^{2 \, m} \, \Im(\tau)^{2 \, (\mu-m)}\\
&~~~ -\Im(\beta) \, \sum_{m=0}^{\mu-1} \, (-1)^{m+1} \, \left(
\begin{matrix} 2 \, \mu \\ 2 \, m +1 \end{matrix} \right) \, \Re(\tau)^{2 \, m+1} \, \Im(\tau)^{2 \, (\mu-m)-1}.
\end{align*}
Next using Young's inequality with some $\delta>0$, we get
\bqs
\Re(\tau)^{k} \, \Im(\tau)^{2 \,\mu-k} \leq \frac{k}{2 \, \mu \, \delta^{\frac{2 \, \mu}{k}}} \, \Re(\tau)^{2 \, \mu}
\, + \, \frac{2\, \mu \, - \, k}{2\, \mu} \, \delta^{\frac{2\, \mu}{2\, \mu - k}} \, \Im(\tau)^{2 \, \mu}, \quad k\, =\, 1, \cdots, 2\, \mu-1\, .
\eqs
And, we also note that the remainder term can be bounded as
\bqs
O\left(|\tau|^{2 \, \mu+1}\right) \leq C \, \epsilon_* \, \left( |\Re(\tau)|^{2 \, \mu}+|\Im(\tau)|^{2 \, \mu}\right), \quad \tau \in B_\epsilon(0) \, ,
\eqs
for each $\epsilon\in(0,\epsilon_*)$ for some constant $C>0$ independent of $\epsilon$. We finally remark that $\delta>0$ can be taken arbitrarily small and
that the leading order term in $\Im(\tau)^{2 \, \mu}$ comes with a negative sign since we assume $\Re(\beta)>0$. As a consequence, upon eventually reducing
the size of $\epsilon_*$, we can find $0<\beta_*<\Re(\beta)<\beta^*$ (depending only on $\epsilon_*$) such that inequality \eqref{eqestimatebeta} holds true in
$B_\epsilon(0)$ for each $\epsilon\in(0,\epsilon_*)$.
\end{proof}

Using Lemma~\ref{lem:3mod}, we readily see that when $-np \leq j\leq 0$, our estimates \eqref{boundnear0}-\eqref{boundaway0} from Lemma~\ref{lem:3mod}
can be combined to
\bqs
\forall \, \tau \in \Omega_\epsilon \cup B_\epsilon(0) \, ,\quad \forall \, j \leq 0 \, ,\quad \big| \, \mathbf{G}_j(\tau) \, \big| \, \le \, C \, {\rm e}^{- \, c \, |j|} \, ,
\eqs
from which we automatically obtain the following estimate, using the contour $\Gamma=\left\{ -\eta +\mathbf{i} \, \ell ~|~ \ell \in[-\pi,\pi]\right\} \subset \Omega_\epsilon
\cup B_\epsilon(0)$ in \eqref{lapgreentau} for any $0<\eta<\eta_\epsilon$. Modifying the contour in \eqref{lapgreentau} is legitimate thanks to Cauchy's formula
and also because the integrals on the segments $\left\{ -\upsilon \pm \mathbf{i} \, \pi ~|~ \upsilon \in[-\eta,\pi] \right\}$ compensate one another.

\begin{lemma}
\label{lem:7} For each $\epsilon \in (0,\epsilon_*)$ there exists constants $C>0$ and $c>0$ such that for all $-n \, p \leq j \leq 0$ with $n\geq 1$, there holds
\bqs
\big| \, \mathcal{G}^n_j \, \big| \, \le \, C \, {\rm e}^{- \, n \, \eta \, - \, c \, |j|} \, ,
\eqs
for any $\eta\in(0,\eta_\epsilon)$ with $\eta_\epsilon>0$ the width given in Lemma~\ref{lem:3mod}.
\end{lemma}

From now on, we assume that $1 \leq j \leq nr$. It turns out that we will need again to divide the analysis in two pieces. We will consider first the medium
range where $1\leq  j\leq n \delta $ where $\delta:=\frac{\alpha}{2}$. In that case we can prove the following lemma.

\begin{lemma}
\label{lem:8}
For each $\epsilon \in \left(0,\min\left( \epsilon_*, \left(\frac{\alpha^{2 \, \mu}}{2\beta^*}\right)^{\frac{1}{2 \, \mu-1}}\right)\right)$ there exists  a constant $C>0$ such
that for $n \geq 1$ and $1 \leq  j \leq n \, \delta$, the following estimate holds:
\bqs
\big| \, \mathcal{G}^n_j \, \big| \, \le \, C \, {\rm e}^{- \, n \, \frac{\eta}{4}} \, ,
\eqs
for each $\eta \in\left(0, \eta_\epsilon \right)$ with $\eta_\epsilon>0$ the width given in Lemma~\ref{lem:3mod}.
\end{lemma}

\begin{proof} For each $\epsilon \in \left(0,\min\left( \epsilon_*, \left(\frac{\alpha^{2 \, \mu}}{2\beta^*}\right)^{\frac{1}{2 \, \mu-1}}\right)\right) $ and for $\eta \in
(0,\eta_\epsilon)$ with $\eta_\epsilon>0$ given in Lemma~\ref{lem:3mod}, we use again the segment $\Gamma=\left\{ -\eta +\mathbf{i} \, \ell ~|~ \ell \in
[-\pi,\pi]\right\} \subset \Omega_\epsilon \cup B_\epsilon(0)$ in \eqref{lapgreentau}. We denote by $\Gamma^{in}$ and $\Gamma^{out}$ the portions of the
segment $\Re(\tau)=-\eta$ which lie either inside $B_\epsilon(0)$ or outside $B_\epsilon(0)$ with $|\Im(\tau)|\leq \pi$. Standard computations (using Lemma
\ref{lem:3mod}) lead to
\bqs
\left| \dfrac{1}{2 \, \mathbf{i} \, \pi} \, \int_{\Gamma^{out}} {\rm e}^{n \, \tau} \, \mathbf{G}_j(\tau) \, \md \tau \right| \, \leq \,
C \, {\rm e}^{- \, n \, \eta \, - \, c \, j} \, ,
\eqs
and
\bqs
\left| \dfrac{1}{2 \, \mathbf{i} \, \pi} \, \int_{\Gamma^{in}} {\rm e}^{n \, \tau} \, \mathbf{G}_j(\tau) \, \md \tau \right| \, \leq \,
C \, {\rm e}^{- \, n \, \eta} \, \int_{\Gamma^{in}} {\rm e}^{\, j \, \Re(\varpi(\tau))} \, \dfrac{|\md \tau|}{2 \, \pi} \, .
\eqs
Next, we recall the estimate \eqref{eqestimatebeta}  on $\Re(\varpi(\tau))$ from Lemma~\ref{lem:3mod}, that is
\bqs
\Re(\varpi(\tau)) \, \leq \, - \, \dfrac{1}{\alpha} \, \Re(\tau) \, +  \, \dfrac{\beta^*}{\alpha^{2\, \mu+1}} \, \Re(\tau)^{2 \, \mu}\, -  \,
\dfrac{\beta_*}{\alpha^{2\, \mu+1}} \, \Im(\tau)^{2 \, \mu} \, , \quad \, \forall \, \tau \in B_\epsilon(0) \, .
\eqs
As a consequence, for all $\tau\in \Gamma^{in}\subset B_\epsilon(0)$, we have
\bqs
\Re(\varpi(\tau)) \, \leq \dfrac{\eta}{\alpha} \, + \, \dfrac{\beta^*}{\alpha^{2\, \mu+1}} \, \eta^{2 \, \mu} \, .
\eqs
Here, we crucially used the fact that the term in $\Im(\tau)^{2 \, \mu}$ in the estimate for $\Re(\varpi(\tau))$ comes with a negative sign. Summarizing, we have
obtained that
\bqs
- \, n \, \eta \, + \, j \, \Re(\varpi(\tau)) \, \leq \, n \, \eta \, \left( - \, \dfrac{1}{2} \, + \, \dfrac{\beta^*}{2\alpha^{2\, \mu}} \, \eta^{2 \, \mu-1} \right) \,
\eqs
for each $\tau\in \Gamma^{in}$ and  $1\leq j \leq \frac{n \alpha}{2}$. Finally, since $0<\eta<\epsilon< \left(\frac{\alpha^{2 \, \mu}}{2\beta^*}\right)^{\frac{1}{2 \, \mu-1}}$,
there holds that $\dfrac{\beta^*}{2\alpha^{2\, \mu}} \, \eta^{2 \, \mu-1} < \frac{1}{4}$, and we get the final estimate
\bqs
{\rm e}^{- \, n \, \eta} \, \int_{\Gamma^{in}} {\rm e}^{\, j \, \Re(\varpi(\tau))} \, \dfrac{|\md \tau|}{2 \, \pi} \, \leq \, {\rm e}^{- \, n \, \frac{\eta}{4}} \, .
\eqs
This concludes the proof of the lemma.
\end{proof}

\begin{figure}[t!]
  \centering
  \includegraphics[width=.45\textwidth]{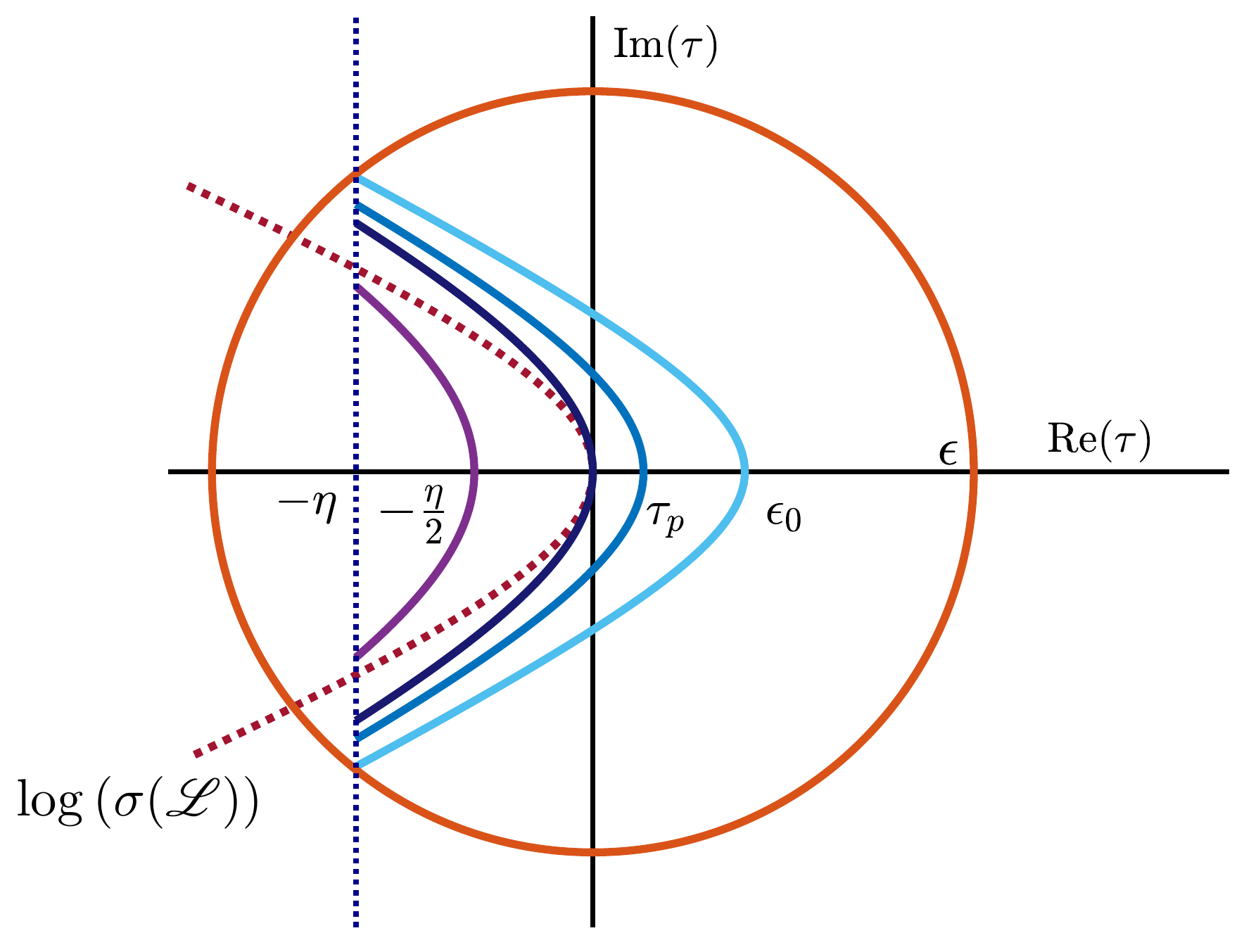}
  \caption{Illustration of the geometry of the family of parametrized curved $\Gamma_p$ within the ball $B_\epsilon(0)$ for different values of $\tau_p \in
  [-\eta/2,\epsilon_0]$. The extremal curves are given when $\tau_p=-\eta/2$ to the left (magenta curve) and when $\tau_p=\epsilon_0$ to the right (light
  blue curve) where $0<\epsilon_0<\epsilon$ is precisely defined such that $\Gamma_p$ with $\tau_p=\epsilon_0$ intersects the segment $\left\{ -\eta
  +\mathbf{i} \, \ell ~|~ \ell \in[-\pi,\pi]\right\}$ on the boundary of $B_\epsilon(0)$. The dashed dark red curve represents the logarithm of the spectrum
  $\sigma(\mathscr{L})$. Note that with our careful choice of parametrization, we have that $\Gamma_p$ with $\tau_p=0$ (dark blue curve) lies to the
  right of the spectral curve with tangency at the origin. }
  \label{fig:contourgamma}
\end{figure}

We now turn to the last case where $n\geq 1$ and $n \, \delta \leq j \leq n \, r$ (recall $\delta=\alpha/2$ and $\alpha<r$). Our generalized Gaussian estimates
will precisely come from this part of the analysis. In order to proceed, we follow the strategy developed in \cite{ZH98} in the fully continuous case (see also
\cite{godillon} in a fully discrete case that corresponds to $\mu=1$ in our notation), and introduce a family of parametrized curves given by
\bqq
\Gamma_p \, := \, \left\{ \, \Re(\tau)-\frac{\beta^*}{\alpha^{2 \, \mu}} \Re(\tau)^{2 \, \mu} \,+ \,\frac{\beta_*}{\alpha^{2 \, \mu}} \Im(\tau)^{2 \, \mu}
\, = \, \Psi\left(\tau_p\right) ~|~ -\eta \leq \Re(\tau)\leq \tau_p \, \right\}
\label{contourGp}
\eqq
with $\Psi\left(\tau_p\right):=\tau_p-\frac{\beta^*}{\alpha^{2 \, \mu}}\tau_p^{2 \, \mu}$. Note that the curves $\Gamma_p$ intersect the real axis at $\tau_p$.
We now explain how we choose $\eta>0$ and $\tau_p>-\eta$ in the above definition of $\Gamma_p$.

First, for each $\epsilon\in(0,\epsilon_*)$, we fix $\eta\in(0,\eta_\epsilon)$ with $\eta_\epsilon>0$ given in Lemma~\ref{lem:3mod} such that the curve
$\Gamma_p$ with $\tau_p=0$ intersects $\left\{-\eta+\mathbf{i} \, \ell ~|~ \ell \in[-\pi,\pi]\right\}$ inside the open ball $B_\epsilon(0)$. Then, we let
$\epsilon_0\in(0,\epsilon)$ which is uniquely defined as the value of $\tau_p$ for which $\Gamma_p$ with $\tau_p=\epsilon_0$ intersects the segment
$\left\{-\eta+\mathbf{i} \, \ell ~|~ \ell \in[-\pi,\pi]\right\}$ precisely on the boundary\footnote{This is possible because the curves $\Gamma_p$ are symmetric
with respect to the real axis.} of $B_\epsilon(0)$ with $\eta$ fixed previously. And finally, the specific value of $\tau_p$ is fixed depending on the ratio
$\frac{\zeta}{\gamma}$ as follows
\bqs
\tau_p \, := \, \left\{
\begin{split}
\rho\left(\frac{\zeta}{\gamma}\right) &\quad \text{ if } \quad  -\frac{\eta}{2}\leq \rho\left(\frac{\zeta}{\gamma}\right) \leq \epsilon_0\,,\\
\epsilon_0 &\quad \text{ if } \quad \rho\left(\frac{\zeta}{\gamma}\right) > \epsilon_0 \, , \\
-\dfrac{\eta}{2} &\quad \text{ if } \quad \rho\left(\frac{\zeta}{\gamma}\right) <- \frac{\eta}{2} \, .
\end{split}
\right.
\eqs
We refer to Figure~\ref{fig:contourgamma}  for an illustration of the geometry of $\Gamma_p$ for different values of $\tau_p$. There only remains to define
$\zeta$, $\gamma$ and the function $\rho$. We let
\bqs
\zeta \, := \, \dfrac{j \, - \, n \, \alpha}{2 \, \mu \, n} \, ,\quad \text{ and } \quad \gamma \, := \, \dfrac{j}{n} \, \dfrac{\beta^*}{\alpha^{2 \, \mu}} \, > \, 0 \, ,
\eqs
and $\rho\left(\frac{\zeta}{\gamma}\right)$ is the unique real root to the equation
\bqs
- \, \zeta \, + \, \gamma \, x^{2 \, \mu-1} \, = \, 0 \, ,
\eqs
that is
\bqs
\rho \left(\frac{\zeta}{\gamma}\right) \, := \, \mathrm{sgn} \left( \dfrac{\zeta}{\gamma} \right) \, \left( \frac{|\zeta|}{\gamma} \right)^{\frac{1}{2 \, \mu-1}} \, .
\eqs

The motivation for introducing such quantities comes from the estimate \eqref{eqestimatebeta} from Lemma~\ref{lem:3mod}. More precisely, for all $\tau
\in \Gamma_p\subset B_\epsilon(0)$, we have
\begin{align*}
j \, \Re(\varpi(\tau)) \, & \, \leq \, j \, \left( - \, \dfrac{1}{\alpha} \, \Re(\tau)
\, + \, \dfrac{\beta^*}{\alpha^{2\, \mu+1}} \, \Re(\tau)^{2 \, \mu}\, - \, \dfrac{\beta_*}{\alpha^{2\, \mu+1}} \, \Im(\tau)^{2 \, \mu} \right) \\
&\, = \, j \, \left( - \, \dfrac{\tau_p}{\alpha}  \, + \, \dfrac{\beta^*}{\alpha^{2\, \mu+1}} \, \tau_p^{2 \, \mu} \right) \\
&\, = \, - \, n \, \tau_p \, + \, \dfrac{n}{\alpha} \, \left( \, - \, 2 \, \mu \, \zeta \, \tau_p \, + \, \gamma \, \tau_p^{2 \, \mu} \right) \, ,
\end{align*}
and our careful choice of $\tau_p$ will always allow us to handle the terms inside the final parenthesis.

We remark that $-\frac{\alpha}{4\mu}\leq\zeta\leq \frac{r-\alpha}{2 \, \mu}$, and our generalized Gaussian estimates will come from those values of
$\zeta \approx 0$. Before proceeding with the analysis, we note that for all $\tau\in\Gamma_p$, we have
\bqq
\label{eqestimateReIm}
\Re(\tau) \, \leq \, \tau_p \, - \, c_* \, \Im(\tau)^{2 \, \mu} \, ,
\eqq
for some constant $c_*>0$. Indeed, we remark that the function $\Psi$, defined as $\Psi(t)=t-\frac{\beta^*}{\alpha^{2 \, \mu}}t^{2 \, \mu}$, satisfies $\Psi'(0)=1$
such that for each $t\in [-\eta,\epsilon]$ one has $\Psi'(t)\leq c_0$ for some $c_0>0$. As a consequence, for each $\tau\in\Gamma_p$ one has
\bqs
- \, \frac{\beta_*}{\alpha^{2 \, \mu}} \, \Im(\tau)^{2 \, \mu} \, = \, \Psi(\Re(\tau)) \, - \, \Psi(\tau_p)
\, = \, - \, \int_{\Re(\tau)}^{\tau_p} \, \Psi'(t) \, \md t \, \geq \, c_0 \, (\Re(\tau)-\tau_p) \, ,
\eqs
which gives the desired estimate \eqref{eqestimateReIm} with $c_*=\frac{\beta_*}{c_0 \, \alpha^{2 \, \mu}}$. Furthermore, a straightforward application of
the implicit function theorem gives the following result on the parametrization of the curves $\Gamma_p$ that we will be using in our estimates below.

\begin{lemma}\label{lemFImp}
There exist $\epsilon_{**} \in(0,\epsilon_*)$, some analytic function $\Phi:(-\epsilon_{**},\epsilon_{**})\times(-\epsilon_{**},\epsilon_{**})\to\R$ and some
constant $C>0$ such that for any $\epsilon\in(0,\epsilon_{**})$ and $\tau_p\in(-\epsilon,\epsilon)$ the curve $\Gamma_p$ can be parametrized as
\bqs
\Gamma_p \, = \, \left\{ \tau \in B_\epsilon(0) ~|~ \Re(\tau) = \Phi(\Im(\tau),\tau_p) \right\} \, ,
\eqs
with
\bqs
\Re(\tau) \, = \, \tau_p \, - \, \frac{\beta^*}{\alpha^{2 \, \mu}} \, \Im(\tau)^{2 \, \mu} \, + \, O\left(|\Im(\tau)|^{2 \, \mu+1}+|\tau_p|^{2 \, \mu+1}\right) \, ,
\eqs
together with
\bqs
\left| \dfrac{\partial \Phi(\Im(\tau),\tau_p)}{\partial \Im(\tau)}\right| \leq C \,,\quad \text{ for each }  |\Im(\tau)| \leq \epsilon  \text{ and  } |\tau_p|\leq \epsilon \, .
\eqs
\end{lemma}

Finally, in what follows, we will use the notation $f \lesssim g$ whenever $f \leq C \, g$ for some constant $C>0$ independent of $j$ and $n$.

\begin{figure}[t!]
  \centering
  \includegraphics[width=.45\textwidth]{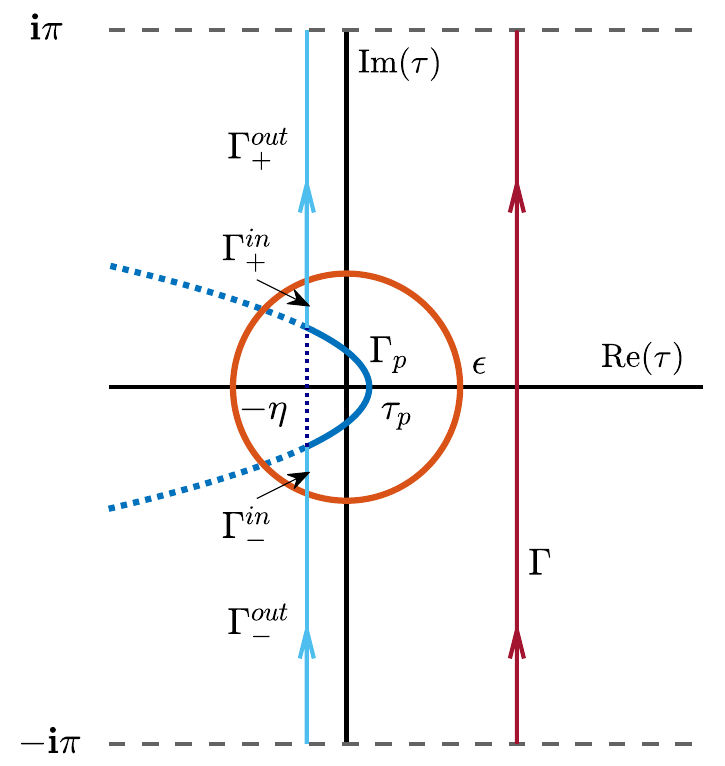}
  \caption{Illustration of the contour used in the case $ -\frac{\eta}{2}\leq \rho\left(\frac{\zeta}{\gamma}\right) \leq \epsilon_0$ when $n \, \delta \leq j
  \leq n \, r$. The contour is composed of $\Gamma_-^{out}\cup\Gamma_-^{in}\cup\Gamma_p\cup\Gamma_+^{in}\cup\Gamma_+^{out}$. The contours
  $\Gamma_\pm^{in}$ and $\Gamma_\pm^{out}$ are the portions of the segment $\Re(\tau)=-\eta$ which lie either inside $B_\epsilon(0)$ or outside
  $B_\epsilon(0)$ with $|\Im(\tau)|\leq \pi$ while $\Gamma_p$ is defined in \eqref{contourGp} and intersects the real axis at $\tau_p$.}
  \label{fig:contour}
\end{figure}

\begin{itemize}
\item We first treat the case $ -\frac{\eta}{2}\leq \rho\left(\frac{\zeta}{\gamma}\right) \leq \epsilon_0$. For all $\tau\in\Gamma_p\subset B_\epsilon(0)$
we obtain
\begin{align*}
n \, \Re(\tau) \, + \, j \, \Re(\varpi(\tau)) \, & \, \leq \, n \, (\Re(\tau)-\tau_p) \, + \, \frac{n}{\alpha} \,
\left( -\, 2 \, \mu \, \zeta \, \tau_p \, + \, \gamma \, \tau_p^{2 \, \mu} \right) \\
&\leq - \, n \, c_* \, \Im(\tau)^{2 \, \mu} \, + \, \frac{n}{\alpha} \, \left( -\, 2 \, \mu \, \zeta \, \tau_p \, + \, \gamma \, \tau_p^{2 \, \mu} \right) \, .
\end{align*}
For the second term, we will use the specific form of $\tau_p=\rho\left(\frac{\zeta}{\gamma}\right)=\mathrm{sgn}(\zeta) \left(
\frac{|\zeta|}{\gamma}\right)^{\frac{1}{2 \, \mu-1}}$ to get that
\bqs
-\, 2 \, \mu \, \zeta \, \tau_p \, + \, \gamma \, \tau_p^{2 \, \mu} \, = \, - \, \gamma \, (2 \, \mu-1) \, \left( \frac{|\zeta|}{\gamma}\right)^{\frac{2 \, \mu}{2 \, \mu-1}}
\, < \, 0 \, .
\eqs
As a consequence, we can derive the following bound
\begin{align*}
\left| \dfrac{1}{2 \, \mathbf{i} \, \pi} \, \int_{\Gamma_p} {\rm e}^{n \, \tau} \, \mathbf{G}_j(\tau) \, \md \tau \right| & \lesssim \int_{\Gamma_p}
{\rm e}^{n\Re(\tau)+j\Re(\varpi(\tau))}|\md \tau| \\
&\lesssim {\rm e}^{-\frac{n}{\alpha} (2 \, \mu-1)\gamma \left(\frac{|\zeta|}{\gamma}\right)^{\frac{2 \, \mu}{2 \, \mu-1}}} \int_{\Gamma_p}
{\rm e}^{-n c_*\Im(\tau)^{2 \, \mu}}|\md \tau|\\
&\lesssim \frac{{\rm e}^{-\frac{n}{\alpha} (2 \, \mu-1)\gamma \left(\frac{|\zeta|}{\gamma}\right)^{\frac{2 \, \mu}{2 \, \mu-1}}}}{n^{\frac{1}{2 \, \mu}}}.
\end{align*}
In the last inequality, assuming that $\epsilon\in(0,\epsilon_{**})$, we have used Lemma~\ref{lemFImp} to get
\bqs
\int_{\Gamma_p}
{\rm e}^{-n c_*\Im(\tau)^{2 \, \mu}}|\md \tau| \lesssim \int_{-\ell_*}^{\ell_*} {\rm e}^{-n c_*x ^{2 \, \mu}}\md x \lesssim \frac{1}{n^{\frac{1}{2 \, \mu}}}\,,
\eqs
where $\ell_*\in(0,\epsilon)$ is defined as $\ell_*:=\left(\frac{\alpha^{2 \, \mu}}{\beta^*}\left(\Psi(\tau_p)-\Psi(-\eta)\right)\right)^{\frac{1}{2 \, \mu}}$, which is
the positive root of
\bqs
- \, \eta \, - \, \frac{\beta^*}{\alpha^{2 \, \mu}} \, \eta^{2 \, \mu} \, + \, \frac{\beta_*}{\alpha^{2 \, \mu}} \, \ell_*^{2 \, \mu}
\, = \, \Psi(\tau_p) \, = \, \tau_p \, - \, \frac{\beta^*}{\alpha^{2 \, \mu}} \, \tau_p^{2 \, \mu} \, .
\eqs

Next we denote by $\Gamma_\pm^{in}$ and $\Gamma_\pm^{out}$ the portions of the segment $\Re(\tau)=-\eta$ which lie either inside $B_\epsilon(0)$
or outside $B_\epsilon(0)$ with $|\Im(\tau)|\leq \pi$. We refer to Figure~\ref{fig:contour} for an illustration. Usual computations lead to
\bqs
\left| \dfrac{1}{2 \, \mathbf{i} \, \pi} \, \int_{\Gamma_\pm^{out}} {\rm e}^{n \, \tau} \, \mathbf{G}_j(\tau) \, \md \tau \right| \leq C \, {\rm e}^{- \, n \, \eta \, - \, c \, j} \, .
\eqs
For all $\tau \in\Gamma_\pm^{in}$, we use that $\Im(\tau)^2\geq \Im(\tau_*)^2$ where $\tau_*=-\eta+\mathbf{i} \, \ell_*$ with $\ell_*=\left(
\frac{\alpha^{2 \, \mu}}{\beta^*}\left(\Psi(\tau_p)-\Psi(-\eta)\right)\right)^{\frac{1}{2 \, \mu}}$. That is $\tau_*=-\eta+\mathbf{i} \, \ell_*$ lies at the intersection
of $\Gamma_p$ and  the segment $\left\{-\eta+\mathbf{i} \, \ell ~|~ \ell \in[-\pi,\pi]\right\}$ with $\tau_*\in B_\epsilon(0)$. As a consequence, for all $\tau \in
\Gamma_\pm^{in}$ we have
\begin{align*}
\Re(\varpi(\tau))& \leq  - \, \dfrac{1}{\alpha} \, \Re(\tau) \, +  \, \dfrac{\beta^*}{\alpha^{2\, \mu+1}} \, \Re(\tau)^{2 \, \mu} \, - \,
\dfrac{\beta_*}{\alpha^{2\, \mu+1}} \, \Im(\tau)^{2 \, \mu} \\
&= \, - \, \frac{\tau_p}{\alpha} \, + \, \frac{\beta^*}{\alpha^{2 \, \mu+1}} \, \tau_p^{2 \, \mu} \, - \, \dfrac{\beta_*}{\alpha^{2\, \mu+1}} \,
\underbrace{\left( \Im(\tau)^{2 \, \mu} -\ell_*^{2 \, \mu}\right)}_{\geq0} \\
&\leq \, - \, \frac{\tau_p}{\alpha} \, + \, \frac{\beta^*}{\alpha^{2 \, \mu+1}} \, \tau_p^{2 \, \mu} \, .
\end{align*}
Thus, we have
\begin{align*}
n\, \Re(\tau) \, + \, j \, \Re(\varpi(\tau))& \leq - \, n \, \eta \, + \, j \, \left( - \, \frac{\tau_p}{\alpha} \, + \, \frac{\beta^*}{\alpha^{2 \, \mu+1}} \, \tau_p^{2 \, \mu}\right) \\
&= \, \frac{n}{\alpha} \, \left[ - \, \eta \, \alpha \, + \, \frac{j}{n} \, \left( -\, \tau_p \, + \, \frac{\beta^*}{\alpha^{2 \, \mu}} \, \tau_p^{2 \, \mu} \right) \right] \\
& \leq \, \frac{n}{\alpha} \, \left[ -\, (\eta+\tau_p) \, \alpha \, - \, 2 \, \mu \, \zeta \, \tau_p \, + \, \gamma \, \tau_p^{2 \, \mu} \right] \\
&= \, \frac{n}{\alpha} \, \left[ -(\eta+\tau_p) \, \alpha \, - \, (2 \, \mu-1) \, \gamma\left(\frac{|\zeta|}{\gamma}\right)^{\frac{2 \, \mu}{2 \, \mu-1}} \right]
\end{align*}
for all $\tau\in\Gamma_\pm^{in}$. Finally,  as $-\frac{\eta}{2} \leq \rho(\frac{\zeta}{\gamma})=\tau_p$ we have $\eta +\tau_p\geq \frac{\eta}{2}$, we obtain an
estimate of the form
\bqs
\left| \dfrac{1}{2 \, \mathbf{i} \, \pi} \, \int_{\Gamma_\pm^{in}} {\rm e}^{n \, \tau} \, \mathbf{G}_j(\tau) \, \md \tau \right| \, \leq \,
C \, {\rm e}^{- \, n \, \frac{\eta}{2} -\frac{n}{\alpha} \, (2 \, \mu-1) \, \gamma \, \left(\frac{|\zeta|}{\gamma}\right)^{\frac{2 \, \mu}{2 \, \mu-1}}} \, .
\eqs
Summarizing, we have obtained
\begin{align*}
\left| \mathscr{G}_j^n \right| & \leq \left| \dfrac{1}{2 \, \mathbf{i} \, \pi} \, \int_{\Gamma_p} {\rm e}^{n \, \tau} \, \mathbf{G}_j(\tau) \, \md \tau \right|
+\left| \dfrac{1}{2 \, \mathbf{i} \, \pi} \, \int_{\Gamma_\pm^{out}} {\rm e}^{n \, \tau} \, \mathbf{G}_j(\tau) \, \md \tau \right|
+\left| \dfrac{1}{2 \, \mathbf{i} \, \pi} \, \int_{\Gamma_\pm^{in}} {\rm e}^{n \, \tau} \, \mathbf{G}_j(\tau) \, \md \tau \right| \\
& \leq C \, \left(\frac{{\rm e}^{-\frac{n}{\alpha} (2 \, \mu-1)\gamma \left(\frac{|\zeta|}{\gamma}\right)^{\frac{2 \, \mu}{2 \, \mu-1}}}}{n^{\frac{1}{2 \, \mu}}}
\, + \,  {\rm e}^{- \, n \, \eta \, - \, c \, j}
\, +\, {\rm e}^{- \, n \, \frac{\eta}{2} -\frac{n}{\alpha} \, (2 \, \mu-1) \, \gamma \, \left(\frac{|\zeta|}{\gamma}\right)^{\frac{2 \, \mu}{2 \, \mu-1}}} \right) \, .
\end{align*}

\item Next, we consider the case $\rho(\zeta/\gamma)>\epsilon_0$ for which we choose $\tau_p=\epsilon_0$. The contour $\Gamma$ is decomposed into
$\Gamma_p\cup\Gamma^{out}_{\pm}$ where $\Gamma_\pm^{out}$ are the portions of the segment $\Re(\tau)=-\eta$ which lie outside $B_\epsilon(0)$ with
$|\Im(\tau)|\leq \pi$. In that case, we have that for all $\tau\in \Gamma_p$
\begin{align*}
n \, \Re(\tau) \, + \, j \, \Re(\varpi(\tau)) \,
&\, \leq \, - \, n \, c_* \, \Im(\tau)^{2 \, \mu} \, + \, \frac{n}{\alpha} \, \left( -\, 2 \, \mu \, \zeta \, \tau_p \, + \, \gamma \, \tau_p^{2 \, \mu} \right) \\
&\, = \, - \, n \, c_* \, \Im(\tau)^{2 \, \mu} \, + \, \frac{n}{\alpha} \, \left( -\, 2 \, \mu \, \zeta \, \epsilon_0 \, + \, \gamma \, \epsilon_0^{2 \, \mu} \right)
\end{align*}
since $\tau_p=\epsilon_0$ in this case.  But as $\rho(\zeta/\gamma)>\epsilon_0$ we get that $\zeta>0$ and $\zeta>\epsilon_0^{2 \, \mu-1}\gamma$, the last
term in the previous inequality is estimated via
\bqs
-\, 2 \, \mu \, \zeta \, \epsilon_0 \, + \, \gamma \, \epsilon_0^{2 \, \mu} \, < \, - \, (2\, \mu-1) \, \gamma \, \epsilon_0^{2 \, \mu} \, < \, 0 \, .
\eqs
As a consequence, we can derive the following bound
\bqs
\left| \dfrac{1}{2 \, \mathbf{i} \, \pi} \, \int_{\Gamma_p} {\rm e}^{n \, \tau} \, \mathbf{G}_j(\tau) \, \md \tau \right|  \lesssim
\frac{{\rm e}^{-\frac{n}{\alpha} (2 \, \mu-1)\gamma \epsilon_0^{2 \, \mu}}}{n^{\frac{1}{2 \, \mu}}}.
\eqs
With our careful choice of $\epsilon_0>0$, the remaining contribution along segments $\Gamma_\pm^{out}$ with $\Re(\tau)=-\eta$ can be estimated as usual as
\bqs
\left| \dfrac{1}{2 \, \mathbf{i} \, \pi} \, \int_{\Gamma_\pm^{out}} {\rm e}^{n \, \tau} \, \mathbf{G}_\tau(j)\md \tau \right| \leq C \, {\rm e}^{- \, n \, \eta \, - \, c \, j} \, ,
\eqs
as $|\tau|\geq \epsilon$ for $\tau\in\Gamma_\pm^{out}$. Summarizing, we have obtained
\bqs
\left| \mathscr{G}_j^n \right| \, \leq  \, \left| \dfrac{1}{2 \, \mathbf{i} \, \pi} \, \int_{\Gamma_p} {\rm e}^{n \, \tau} \, \mathbf{G}_j(\tau) \, \md \tau \right|
+\left| \dfrac{1}{2 \, \mathbf{i} \, \pi} \, \int_{\Gamma_\pm^{out}} {\rm e}^{n \, \tau} \, \mathbf{G}_j(\tau) \, \md \tau \right| \, \leq \, C \,
\left( \frac{{\rm e}^{-\frac{n}{\alpha} (2 \, \mu-1)\gamma \epsilon_0^{2 \, \mu}}}{n^{\frac{1}{2 \, \mu}}} \, + \,  {\rm e}^{- \, n \, \eta \, - \, c \, j} \right) \, .
\eqs

\item It remains to handle the last case $\rho(\zeta/\gamma)<-\eta/2$ for which we choose $\tau_p=-\eta/2$, and we readily note that in this setting $\zeta<0$.
The contour $\Gamma$ is decomposed into $\Gamma_p\cup\Gamma_{\pm}^{out}\cup\Gamma_\pm^{in}$ where once again $\Gamma_\pm^{in}$ and
$\Gamma_\pm^{out}$ are the portions of the segment $\Re(\tau)=-\eta$ which lie either inside $B_\epsilon(0)$ or outside $B_\epsilon(0)$ with $|\Im(\tau)|\leq \pi$.
For all $\tau\in\Gamma_p \subset B_\epsilon(0)$, we find that
$$
n \, \Re(\tau) \, + \, j \, \Re(\varpi(\tau)) \, \leq \, - \, n \, c_* \, \Im(\tau)^{2 \, \mu} \, + \, \frac{n}{\alpha} \,
\left( \mu \, \zeta \, \eta \, + \, \gamma \, \left(\frac{\eta}{2}\right)^{2 \, \mu} \right) \, .
$$
Using that $\rho(\zeta/\gamma)<-\eta/2$ which is equivalent to $\zeta/\gamma<-\left(\frac{\eta}{2}\right)^{2 \, \mu-1}$, we get that
\bqs
\mu \, \zeta \, \eta \, + \, \gamma \, \left(\frac{\eta}{2}\right)^{2 \, \mu} \, < \, - \, (2\, \mu-1) \, \gamma \, \left(\frac{\eta}{2}\right)^{2 \, \mu} \, .
\eqs
As a consequence, we can derive the following bound
\bqs
\left| \dfrac{1}{2 \, \mathbf{i} \, \pi} \, \int_{\Gamma_p} {\rm e}^{n \, \tau} \, \mathbf{G}_j(\tau) \, \md \tau \right|  \lesssim
\frac{{\rm e}^{-\frac{n}{\alpha} (2 \, \mu-1)\gamma\left(\frac{\eta}{2}\right)^{2 \, \mu}}}{n^{\frac{1}{2 \, \mu}}}.
\eqs
As usual, we have that
\bqs
\left| \dfrac{1}{2 \, \mathbf{i} \, \pi} \, \int_{\Gamma_\pm^{out}} {\rm e}^{n \, \tau} \, \mathbf{G}_j(\tau) \, \md \tau \right| \leq C \, {\rm e}^{- \, n \, \eta \, - \, c \, j} \, .
\eqs
It only remains to estimate the contribution on $\Gamma_\pm^{in}$. We proceed as before, and we have that for all $\tau\in\Gamma_\pm^{in}$
\begin{align*}
n \, \Re(\tau) \, + \, j \, \Re(\varpi(\tau)) \, & \, \leq \, \frac{n}{\alpha} \, \left[
- \left( \eta + \tau_p \right) \, \alpha \, - \, 2 \, \mu \, \zeta \, \tau_p \, + \, \gamma \, \tau_p^{2 \, \mu} \right] \\
& \, \le \, \frac{n}{\alpha} \, \left[ - \, \frac{\eta}{2} \, \alpha \, - \, (2\, \mu-1) \, \gamma \, \left(\frac{\eta}{2}\right)^{2 \, \mu} \right] \\
& \, = \, - \, n \, \left( \frac{\eta}{2} \, + \, \frac{(2 \, \mu-1) \, \gamma}{\alpha} \, \left( \frac{\eta}{2} \right)^{2 \, \mu} \right) \, ,
\end{align*}
and this time we obtain
\bqs
\left| \dfrac{1}{2 \, \mathbf{i} \, \pi} \, \int_{\Gamma_\pm^{in}} {\rm e}^{n \, \tau} \, \mathbf{G}_j(\tau) \, \md \tau \right| \leq
C \, {\rm e}^{-n\left( \frac{\eta}{2}+\frac{(2 \, \mu-1)\gamma}{\alpha}\left(\frac{\eta}{2}\right)^{2 \, \mu}\right)} \, .
\eqs
In conclusion, we have obtained
\begin{align*}
\left| \mathscr{G}_j^n \right| & \leq \left| \dfrac{1}{2 \, \mathbf{i} \, \pi} \, \int_{\Gamma_p} {\rm e}^{n \, \tau} \, \mathbf{G}_j(\tau) \, \md \tau \right|
+ \left| \dfrac{1}{2 \, \mathbf{i} \, \pi} \, \int_{\Gamma_\pm^{out}} {\rm e}^{n \, \tau} \, \mathbf{G}_j(\tau) \, \md \tau \right|
+ \left| \dfrac{1}{2 \, \mathbf{i} \, \pi} \, \int_{\Gamma_\pm^{in}} {\rm e}^{n \, \tau} \, \mathbf{G}_j(\tau) \, \md \tau \right| \\
& \leq C \, \left(\frac{{\rm e}^{-\frac{n}{\alpha} (2 \, \mu-1)\gamma\left(\frac{\eta}{2}\right)^{2 \, \mu}}}{n^{\frac{1}{2 \, \mu}}} \, + \,
{\rm e}^{- \, n \, \eta \, - \, c \, j}\, +\, {\rm e}^{-n\left( \frac{\eta}{2}+\frac{(2 \, \mu-1)\gamma}{\alpha}\left(\frac{\eta}{2}\right)^{2 \, \mu}\right)} \right) \, .
\end{align*}
\end{itemize}

As a summary, gathering the above estimates, we can deduce the following result.

\begin{lemma}
\label{lem:9}
For each $\epsilon\in(0,\epsilon_{**})$ there exist constants $C>0$ and $c>0$ such that for $n\geq 1$ and $n \, \delta \leq j \leq n \, r$ the following
estimate holds:
\bqs
\big| \, \mathcal{G}^n_j \, \big| \, \le \, \dfrac{C}{n^{\frac{1}{2 \, \mu}}} \, \exp \left( -c \, \left( \dfrac{|j- \alpha \, n|}{n^{\frac{1}{2 \, \mu}}}
\right)^{\frac{2 \, \mu}{2 \, \mu-1}}\right) \, ,
\eqs
where $\epsilon_{**}>0$ is given in Lemma~\ref{lemFImp}.
\end{lemma}

\begin{proof}
One only needs to check that the purely exponentially decaying in $n$ contributions obtained when $\rho\left(\frac{\zeta}{\gamma}\right)>\epsilon_0$
or $\rho\left(\frac{\zeta}{\gamma}\right)<-\frac{\eta}{2}$ can be subsumed into generalized Gaussian estimates. For example, in the case
$\rho\left(\frac{\zeta}{\gamma} \right)>\epsilon_0$, there exists some small constant $c>0$ such that
\bqs
- \, n \, \leq \, - \, c \, \left( \dfrac{|j- \alpha \, n|}{n^{\frac{1}{2 \, \mu}}} \right)^{\frac{2 \, \mu}{2 \, \mu-1}} \, ,
\eqs
as
\bqs
\dfrac{\beta^*}{2 \, \alpha^{2 \, \mu-1}} \, \leq \, \gamma \, \leq \, \dfrac{\beta^*}{\alpha^{2 \, \mu}} \, r \, ,
\eqs
and
\bqs
\frac{j}{n \, \alpha }-1 \, = \, \dfrac{2 \, \mu \, \zeta}{\alpha} \, > \, \dfrac{2 \, \mu}{\alpha} \, \gamma \, \epsilon_0^{2\, \mu-1} \, \geq \,
\dfrac{\mu \, \beta^*}{\alpha^{2 \, \mu}} \, \epsilon_0^{2 \, \mu-1} \, .
\eqs
All other cases can be dealt with in a similar way.
\end{proof}

\begin{proof}[Proof of Theorem~\ref{thm:1}]
Combining Lemma~\ref{lem:7}, Lemma~\ref{lem:8} and Lemma~\ref{lem:9}  proves our main Theorem~\ref{thm:1} in the explicit case with $K=1$.
Indeed, we first fix $\epsilon \in \left(0,\min\left( \epsilon_{**}, \left(\frac{\alpha^{2 \, \mu}}{2\beta^*}\right)^{\frac{1}{2 \, \mu-1}}\right)\right)$ with $0 <
\epsilon_{**}<\epsilon_*$ from Lemma~\ref{lemFImp}, and then we fix $\eta\in(0,\eta_\epsilon)$ with $\eta_\epsilon>0$ given in Lemma~\ref{lem:3mod}
such that the curve $\Gamma_p$ with $\tau_p=0$ intersects $\left\{-\eta+\mathbf{i} \, \ell ~|~ \ell \in[-\pi,\pi]\right\}$ inside the open ball $B_\epsilon(0)$.
As a consequence, we can apply estimates from Lemma~\ref{lem:7}, Lemma~\ref{lem:8} and Lemma~\ref{lem:9} with this specific choice of $\epsilon$
and $\eta$, which gives the proof since purely exponentially decaying in $n$ bounds obtained in Lemma~\ref{lem:7} and Lemma~\ref{lem:8} can be
subsumed into generalized Gaussian estimates.
\end{proof}

\begin{figure}[t!]
\centering
\includegraphics[width=.48\textwidth]{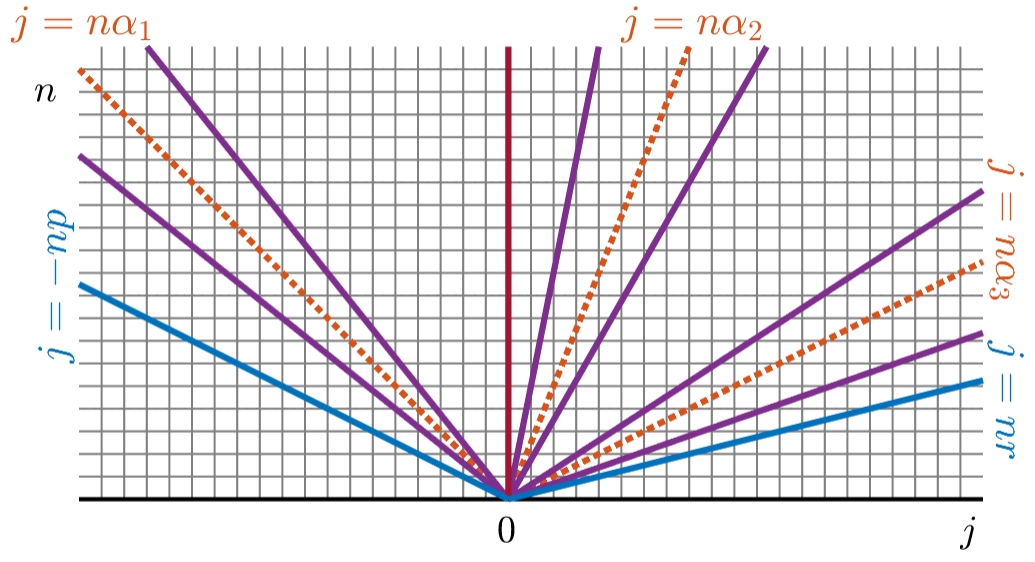}
\includegraphics[width=.48\textwidth]{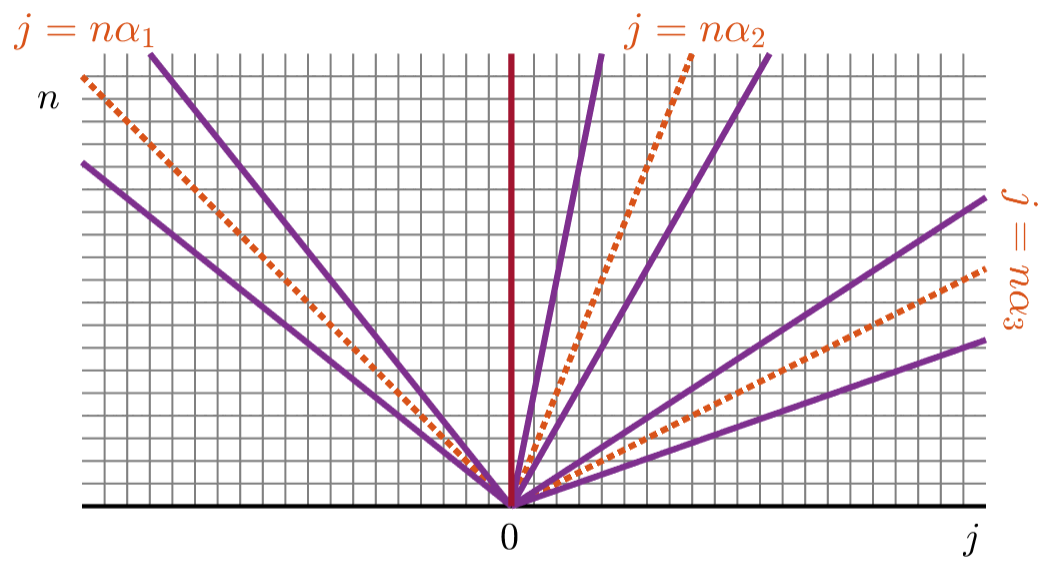}
  \caption{Illustration for the explicit (left) and implicit (right) cases of the different domains in the $(j,n)$ plane where generalized Gaussian estimates are
  obtained; here $K=3$ with $\alpha_1<0$ and $\alpha_{2}\neq \alpha_3>0$. Generalized Gaussian estimates are derived near each $j\approx n \, \alpha_k$,
  $k=\left\{1,2,3\right\}$ in the interior of the sectors delimited by the magenta lines. In the explicit case, below the lines $j=-n \, p$ and $j=n \, r$ (blue), the
  Green's function $\G^n_j$ vanishes.}
  \label{fig:spacetime}
\end{figure}

\subsection{The implicit case with $K=1$}

The main difference compared with the explicit case is that now it is no longer true that $\G^n_j$ vanishes for $j>n \, r$ or $j<-p \, n$. Nevertheless we observe
that the results of Lemma~\ref{lem:7}, Lemma~\ref{lem:8} and Lemma~\ref{lem:9} still hold true. Actually, the proofs of Lemma~\ref{lem:7}, Lemma~\ref{lem:8}
and Lemma~\ref{lem:9} naturally extend to $-n \, L \leq j \leq 0$ (Lemma~\ref{lem:7}), $1 \leq j \leq n \, \delta$ (Lemma~\ref{lem:8}, unchanged) and
$n \, \delta \leq j \leq n \, L$ (Lemma~\ref{lem:9}) for any large constant $L \geq \max( p,r)$ that is fixed a priori. As a consequence, one only needs to consider
the case $n\geq 1$ and $|j| > n \, L$ for some large constant $L>0$ to be determined. To obtain the desired estimate in that case, we use the bound at infinity
obtained in Lemma~\ref{lem:5}. More precisely, there exists $R\geq \pi/2$ and two constants $C>0$, $\underline{c}>0$ such that there holds
\begin{equation}
\label{estiminfini}
\forall \, \tau \in \left\{ \zeta \in \C ~|~ \Re(\zeta) \geq \log R \right\} \, , \quad \forall\, j \in \Z \, , \quad
\big| \, \mathbf{G}_j(\tau) \, \big| \, \le \, C \, \exp \big( -\underline{c} \, |j| \, \big) \, .
\end{equation}
We then have the following result.

\begin{lemma}\label{lem:10}
Let $L \geq \max( p,r) >0$ be large enough such that $L>2 \, \log R \, / \, \underline{c}$ with $R$ and $\underline{c}$ as in \eqref{estiminfini}. Then, there exists
$C>0$, such that for $n\geq 1$ and $|j|>nL$, we have
\bqs
\big| \, \G^n_j \, \big| \, \le \, C \, \exp \left(- \, n \, \dfrac{\underline{c} \, L}{4} \, - \, \dfrac{\underline{c}}{2} \, |j| \, \right) \, .
\eqs
\end{lemma}

\begin{proof}
In \eqref{lapgreentau}, we now use the contour $\Gamma = \left\{ \frac{\underline{c} \, |j|}{2 \, n} +\mathbf{i} \, \ell ~|~ \ell \in[-\pi,\pi] \right\}$. With our choice of $L$, we
have that for all $\tau \in \Gamma$, $\Re(\tau) = \frac{\underline{c} \, |j|}{2 \, n} \geq \frac{\underline{c} L}{2}> \log R$ and so
\bqs
\left| \dfrac{1}{2 \, \mathbf{i} \, \pi} \, \int_{\Gamma} {\rm e}^{n \, \tau} \, \mathbf{G}_j(\tau) \, \md \tau \right| \, \leq \, C \, {\rm e}^{- \, \frac{\underline{c}}{2} \, |j|} \, .
\eqs
Finally, we notice that
\bqs
- \, \dfrac{\underline{c}}{2} \, |j| \, \leq \, - \, n \, \dfrac{\underline{c} \, L}{4} \, - \, \dfrac{\underline{c}}{4} \, |j| \, , \quad \text{ for } |j| \, > \, n \, L \, .
\eqs
This completes the proof of the lemma.
\end{proof}

The proof of Theorem~\ref{thm:1} in the implicit case (for $K=1$) then follows from the combination of the slight extensions of Lemma~\ref{lem:7}, Lemma~\ref{lem:8}
and Lemma~\ref{lem:9} (once the large constant $L$ is fixed as in Lemma~\ref{lem:10}).

\subsection{The explicit and implicit cases with $K>1$}

We now briefly explain how to handle the general case with $K>1$ and refer to Figures~\ref{fig:spacetime}-\ref{fig:contourball} for illustrations. From
Assumption~\ref{hyp:1}, we have the existence of $K$ tangency points $\underline{\kappa}_k$ with associated nonzero real numbers $\alpha_k$. We
will distinguish two cases:
\begin{itemize}
\item[{\bf A}.] All $\alpha_k$ are distinct from one and another.
\item[{\bf B}.] There exist two or more $\alpha_k$ which are equal.
\end{itemize}
We only discuss the explicit case here, as the implicit case does not distinguish between Cases {\bf A} and {\bf B}. First, we let be $\tau_k = \mathbf{i} \, \theta_k
:= \log (\underline{z}_k)$ for $\theta_k \in [-\pi,\pi]$ and $\widetilde{\theta}_k\in]-\pi,\pi]$ be such that $\underline{\kappa}_k=\rm{e}^{\mathbf{i} \, \widetilde{\theta}_k}$
for each $k=1,\cdots,K$. In order to proceed, we need the following lemma which is a direct consequence of Corollary~\ref{cor1} and whose proof is identical to
Lemma~\ref{lem:3mod}.

\begin{lemma}\label{lemkgen}
There exist some $\epsilon_*>0$ and two constants $0<\beta_* <\Re(\beta_k)<\beta^*$ for $k=1,\cdots,K$ such that for each $\epsilon\in(0,\epsilon_*)$
there exist some width $\eta_\epsilon>0$ together with two constants, still denoted $C>0$, $c>0$, such that, for any integer $j\in\Z$, the component
$\mathbf{G}_j(\tau)$ extends holomorphically on each $ B_\epsilon(\tau_k)$ with bounds:
\bqs
\forall \, \tau \in B_\epsilon(\tau_k) \, ,\quad \forall \, j \in \Z \, ,\quad \big| \, \mathbf{G}_j(\tau) \, \big| \, \le \, \begin{cases}
C \, \exp \big( - \, c \, |j| \, \big) \, ,& \text{\rm if $j \, \le \, 0$,} \\
C \, \exp \big( j \, \Re(\varpi_k(\tau)) \big) \, ,& \text{\rm if $j \, \ge \, 1$,}
\end{cases} \quad \mathrm{ (Case~I)}
\eqs
or
\bqs
\forall \, \tau \in B_\epsilon(\tau_k) \, ,\quad \forall \, j \in \Z \, ,\quad \big| \, \mathbf{G}_j(\tau) \, \big| \, \le \, \begin{cases}
C \, \exp \big( j \, \Re(\varpi_k(\tau)) \big) \, ,& \text{\rm if $j \, \le \, 0$,} \\
C  \, \exp \big( -c \, |j| \, \big) \, ,& \text{\rm if $j \, \ge \, 1$,}
\end{cases} \quad \mathrm{ (Case~II)}
\eqs
or
\bqs
\forall \, \tau \in B_\epsilon(\tau_k) \, ,\quad \forall \, j \in \Z \, ,\quad \big| \, \mathbf{G}_j(\tau) \, \big| \, \le \, \begin{cases}
C  \, \exp \big( j \, \Re(\varpi_{\nu_{k,1}}(\tau)) \big) \, ,& \text{\rm if $j \, \le \, 0$,} \\
C  \, \exp \big( j \, \Re(\varpi_{\nu{k,2}}(\tau)) \big) \, ,& \text{\rm if $j \, \ge \, 1$,}
\end{cases} \quad \mathrm{ (Case~III)}
\eqs
where each $\varpi_k$ is holomorphic on $B_\epsilon(\tau_k)$ and has the Taylor expansion:
\bqs
\varpi_k(\tau) \, = \, \mathbf{i} \, \widetilde{\theta}_k \, - \, \dfrac{1}{\alpha_k} \, (\tau-\tau_k) \, + \,
(-1)^{\mu_k+1} \, \dfrac{\beta_k}{\alpha_k^{2 \, \mu_k+1}} \, (\tau-\tau_k)^{2 \, \mu_k}
\, + \, O\left(|\tau-\tau_k|^{2 \, \mu_k+1}\right) \, ,\quad \forall \, \tau \in B_\epsilon(\tau_k) \, ,
\eqs
together with
\bqq
\Re(\varpi_k(\tau)) \, \leq \, - \, \dfrac{\Re(\tau)}{\alpha_k} \, +  \, \dfrac{\beta^*}{\alpha_k^{2 \, \mu_k+1}} \, \Re(\tau)^{2 \, \mu_k}
\, - \, \dfrac{\beta_*}{\alpha_k^{2 \, \mu_k+1}} \, \left(\Im(\tau)-\theta_k\right)^{2 \, \mu_k}  \, ,\quad \forall \, \tau \in B_\epsilon(\tau_k) \, .
\label{eqestimatebetak}
\eqq
Furthermore, we have
\bqs
\forall \, \tau \in \Omega_\epsilon \, := \, \left\{ \, - \, \eta_\epsilon \, < \, \Re(\tau) \, \leq \, \pi \, \right\} \backslash \, \bigcup_{k=1}^K B_\epsilon(\tau_k) \, ,\quad
\forall \, j \in \Z \, ,\quad \big| \, \mathbf{G}_j(\tau) \, \big| \, \le \, C \, \exp \big( -c \, |j| \, \big) \, .
\eqs
\end{lemma}

\paragraph{Case A.} This is precisely the case depicted in Figure~\ref{fig:spacetime} with $K=3$.  Without loss of generality we label the $\alpha_k$
by increasing order such that
\bqs
-p <\alpha_1<\cdots<\alpha_k<\cdots<\alpha_K<r.
\eqs
For each $k=1,\cdots,K$ we define two real numbers $\underline{\delta}_k<\overline{\delta}_k$ such that we have the ordering
\bqs
-p<\underline{\delta}_1<\alpha_1<\overline{\delta}_1<\cdots<\underline{\delta}_k<\alpha_k<\overline{\delta}_k<\cdots<\underline{\delta}_K<\alpha_K
<\overline{\delta}_K<r \, ,
\eqs
with $\mathrm{sgn}(\underline{\delta}_k)=\mathrm{sgn}(\overline{\delta}_k)=\mathrm{sgn}(\alpha_k)$. For each $k=1,\cdots,K$, we define the following
sectors in the $(j,n)$-plane:
\bqs
\mathcal{D}_k \, := \, \left\{ (j,n) \in \Z \times \N^* ~|~ n \, \underline{\delta}_k \, \leq \, j \, \leq \, n \, \overline{\delta}_k\right\} \, ,
\eqs
together with
\bqs
\mathcal{D}_* \, := \, \left\{ (j,n) \in \Z \times \N^* ~|~ - \, n \, p \, \leq \, j \, \leq \, n \, r \right\} \, \backslash \, \bigcup_{k=1}^K \mathcal{D}_k \, .
\eqs
Our first lemma  pertains at obtaining exponential bounds in the region $\mathcal{D}_*$. We introduce two quantities
\bqs
\Lambda_1^* \, := \, \underset{k=1,\cdots,K}{\min} \left( \frac{\alpha_k^{2 \, \mu_k}}{2 \, \beta^*} \right)^\frac{1}{2 \, \mu_k-1} \, > \, 0 \, ,\quad
\text{ and } \quad
\Lambda_2^* \, := \, \underset{k=1,\cdots,K}{\min} \left( \frac{(\overline{\delta}_k-\alpha_k) \, \alpha_k^{2 \, \mu_k}}{2 \, r \, \beta^*} \right)^\frac{1}{2 \, \mu_k-1}
\, > \, 0 \, .
\eqs

\begin{figure}[t!]
\centering
\includegraphics[width=.3\textwidth]{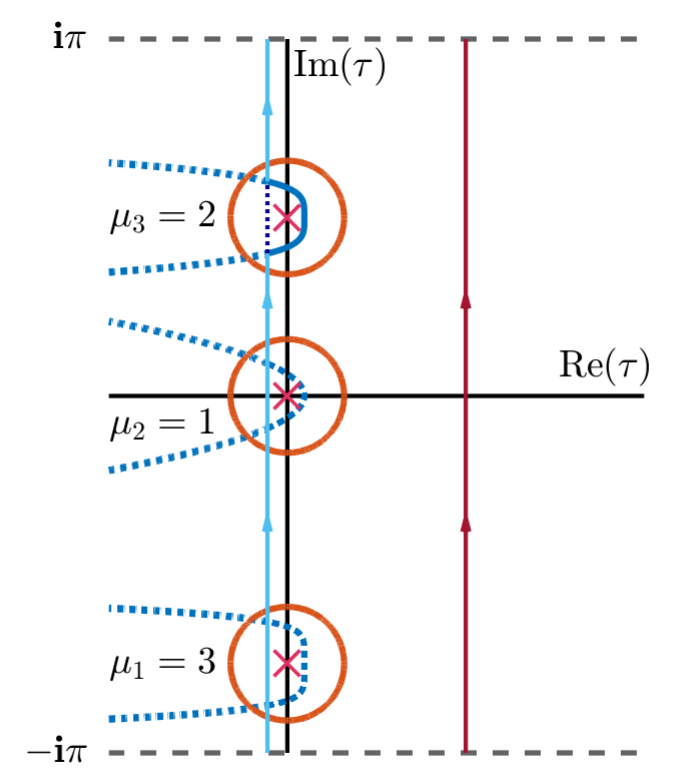}
\includegraphics[width=.3\textwidth]{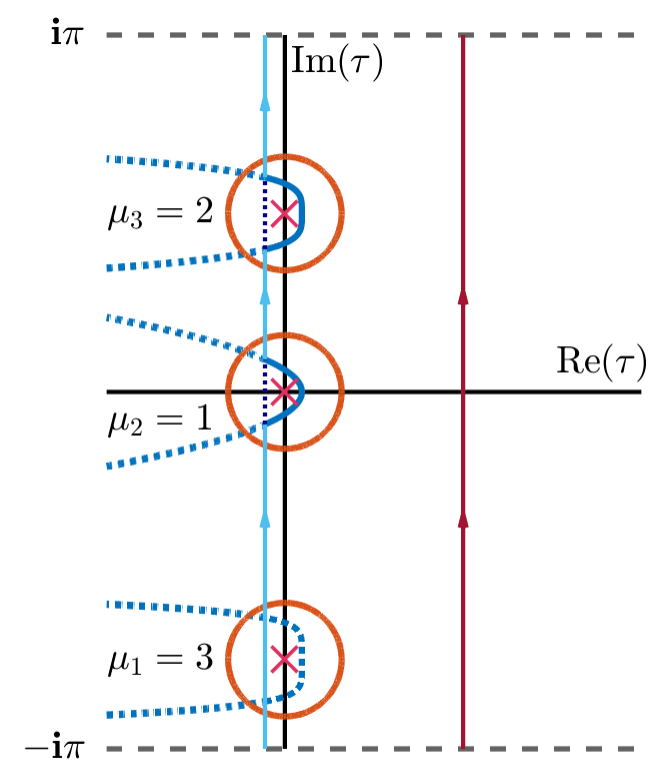}
\includegraphics[width=.3\textwidth]{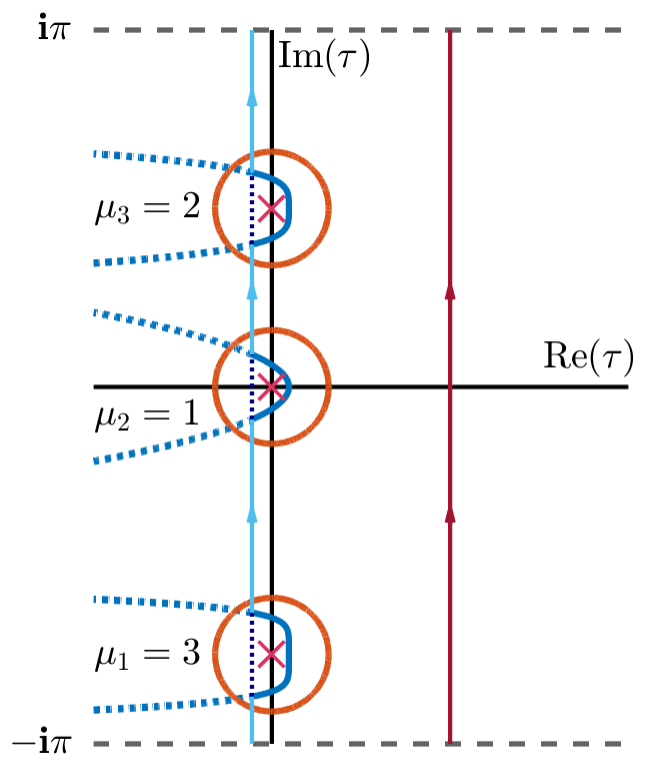}
  \caption{Left: typical contour used when $(j,n) \in\mathcal{D}_3$ in the case depicted in Figure~\ref{fig:spacetime} where all $\alpha_k$ are distinct
  $\alpha_1<\alpha_2<\alpha_3$ (case {\bf A}). Middle: typical contour used when $(j,n)\in\mathcal{D}_2=\mathcal{D}_3$ in the case $\alpha_1<
  \alpha_2=\alpha_3$ (case {\bf B}). Right: typical contour used when $(j,n)\in\mathcal{D}_1=\mathcal{D}_2=\mathcal{D}_3$ in the case $\alpha_1=
  \alpha_2=\alpha_3$ (case {\bf B}). Here $K=3$.}
  \label{fig:contourball}
\end{figure}

\begin{lemma}
\label{lem:11}
For each $\epsilon \in(0,\min(\epsilon_*,\Lambda_1^*,\Lambda_2^*))$, there exist $C>0$ and $\delta>0$ such that for each $(j,n) \in \mathcal{D}_*$ the
following estimate holds:
\bqs
\big| \, \mathcal{G}^n_j \, \big| \, \le \, C \, {\rm e}^{- \, n \, \delta} \, .
\eqs
\end{lemma}

\begin{proof}
We only sketch the proof as it is almost identical to the proofs of Lemma~\ref{lem:7} and Lemma~\ref{lem:8}. Let $\epsilon\in(0,\min(\epsilon_*, \Lambda_1^*,
\Lambda_2^*))$ and consider $\eta\in(0,\eta_\epsilon)$. We select the contour $\Gamma = \left\{ -\eta+\mathbf{i} \, \ell ~|~ \ell \in[-\pi,\pi]\right\}$ and we denote
by $\Gamma^{in}_k$ the portion of $\Gamma$ which lie within $B_\epsilon(\tau_k)$ and $\Gamma^{out}$ the union of the remaining portions. As a consequence,
we have $\Gamma=\Gamma^{in}_1 \cup \cdots \cup \Gamma^{in}_K\cup\Gamma^{out}$, and we get
\bqs
\left| \dfrac{1}{2 \, \mathbf{i} \, \pi} \, \int_{\Gamma} {\rm e}^{n \, \tau} \, \mathbf{G}_j(\tau) \, \md \tau \right| \leq
\sum_{k=1}^K\left| \dfrac{1}{2 \, \mathbf{i} \, \pi} \, \int_{\Gamma^{in}_k} {\rm e}^{n \, \tau} \, \mathbf{G}_j(\tau) \, \md \tau \right|
+\left| \dfrac{1}{2 \, \mathbf{i} \, \pi} \, \int_{\Gamma^{out}} {\rm e}^{n \, \tau} \, \mathbf{G}_j(\tau) \, \md \tau \right| \, .
\eqs
Our objective is to bound each above term separately. Along $\Gamma^{out}$, we get an estimate of the form
\bqs
\left| \dfrac{1}{2 \, \mathbf{i} \, \pi} \, \int_{\Gamma^{out}} {\rm e}^{n \, \tau} \, \mathbf{G}_j(\tau) \, \md \tau \right| \leq
C \, {\rm e}^{- \, n \, \eta \, - \, c \, |j|} \, , \quad \forall \, (j,n) \in \mathcal{D}_* \, ,
\eqs
as along $\Gamma^{out}$ the Green's function $\mathbf{G}_j(\tau)$ enjoys the pointwise exponential bound from Lemma~\ref{lemkgen}.

We now derive pointwise bound for each contour integral along $\Gamma^{in}_k$, $k=1,\cdots,K$. We first handle the case where $\mathcal{I}_k =
\left\{ k \right\}$, and assume without loss of generality that $\alpha_k>0$. Then Case I of Lemma~\ref{lemkgen} reads
\bqq
\forall \, \tau \in B_\epsilon(\tau_k) \, ,\quad \forall \, j \in \Z \, ,\quad \big| \, \mathbf{G}_j(\tau) \, \big| \, \le \, \begin{cases}
C \, \exp \big( - \, c \, |j| \, \big) \, ,& \text{\rm if $j \, \le \, 0$,} \\
C \, \exp \big( j \,\Re(\varpi_k(\tau)) \big) \, ,& \text{\rm if $j \, \ge \, 1$,}
\end{cases}
\label{boundnear0k}
\eqq
with
\bqs
\Re(\varpi_k(\tau)) \, \leq \, - \, \dfrac{\Re(\tau)}{\alpha_k} \, +  \, \dfrac{\beta^*}{\alpha_k^{2 \, \mu_k+1}} \, \Re(\tau)^{2 \, \mu_k}
\, - \, \dfrac{\beta_*}{\alpha_k^{2 \, \mu_k+1}} \, \left(\Im(\tau)-\theta_k\right)^{2 \, \mu_k}  \, ,\quad \forall \, \tau \in B_\epsilon(\tau_k) \, .
\eqs

If $(j,n) \in \mathcal{D}_*$ is such that $j \leq 0$, then we directly get
\bqs
\left| \dfrac{1}{2 \, \mathbf{i} \, \pi} \, \int_{\Gamma^{in}_k} {\rm e}^{n \, \tau} \, \mathbf{G}_j(\tau) \, \md \tau \right| \, \leq \, C \, {\rm e}^{- \, n \, \eta \, - \, c \, |j|} \,.
\eqs
From now on, we therefore consider $(j,n)\in \mathcal{D}_*$ with $j\geq1$. As in the proof of Lemma~\ref{lem:8}, we use the above estimate
\eqref{eqestimatebetak} to get that
\bqs
\Re(\varpi_k(\tau)) \, \leq \, \dfrac{\eta}{\alpha_k} \, +  \, \dfrac{\beta^*}{\alpha_k^{2 \, \mu_k+1}} \, \eta^{2 \, \mu_k}
\eqs
for each $\tau\in \Gamma^{in}_k \subset B_\epsilon(\tau_k)$. As a consequence, for all $(j,n) \in \mathcal{D}_*$ with $1 \leq j \leq n \, \underline{\delta}_k$,
we have
\begin{align*}
-n \, \eta \, + \, j \, \Re(\varpi_k(\tau)) \, &\leq \, n \, \eta \, \left( -1 +\frac{j}{n \, \alpha_k} \, + \,
\frac{j}{n} \, \dfrac{\beta^*}{\alpha_k^{2 \, \mu_k+1}} \, \eta^{2 \, \mu_k-1}  \right) \\
&\leq \, - \, n \, \eta \, \left( \underbrace{1 - \frac{\underline{\delta}_k}{\alpha_k}}_{>0} - \, \dfrac{\beta^*}{\alpha_k^{2 \, \mu_k}} \, \eta^{2 \, \mu_k-1} \right) \\
&\leq \, - \, \frac{n \, \eta}{2} \, \left( 1 - \frac{\underline{\delta}_k}{\alpha_k} \right) \, ,
\end{align*}
since $\eta$ is chosen such that $0<\eta<\eta_\epsilon<\epsilon<\Lambda_1^*$. And we have obtained the estimate
\bqs
\left| \dfrac{1}{2 \, \mathbf{i} \, \pi} \, \int_{\Gamma^{in}_k} {\rm e}^{n \, \tau} \, \mathbf{G}_j(\tau) \, \md \tau \right| \, \leq \, C \, \exp\left( - \, \frac{n \, \eta}{2} \, \left( 1 - \frac{\underline{\delta}_k}{\alpha_k} \right) \right)\,.
\eqs

For the remaining cases $(j,n) \in \mathcal{D}_*$ and $n \, \overline{\delta}_k\leq j \leq n \, r$,  we use a different contour near the ball $B_\epsilon(\tau_k)$.
We refer to Figure~\ref{fig:contourMod} for an illustration. We introduce the contour
\bqs
\Gamma_{\epsilon_{k}}^k \, := \, \left\{ \Re(\tau)-\frac{\beta^*}{\alpha_k^{2 \, \mu_k}}
\Re(\tau)^{2 \, \mu_k} + \frac{\beta_*}{\alpha_k^{2 \, \mu_k}} \left( \Im(\tau)-\theta_k \right)^{2 \, \mu_k} = \Psi_k \left(\epsilon_k \right)~|~
-\eta \leq \Re(\tau)\leq \epsilon_k \right\} \, ,
\eqs
with $\Psi_k\left(\epsilon_k\right):=\epsilon_k-\frac{\beta^*}{\alpha_k^{2 \, \mu_k}}\epsilon_k^{2 \, \mu_k}$ and where $0<\epsilon_k<\epsilon$ is chosen such
that $\Gamma_{\epsilon_k}^k$ intersects $\Gamma_k^{in}$ precisely on the boundary of $B_\epsilon(\tau_k)$. We note that there exists some constant
$c_k>0$ such that for any $\tau\in \Gamma_{\epsilon_k}^k \subset B_\epsilon(\tau_k)$, one has
\bqs
\Re(\tau) \leq \epsilon_k \, - \, c_k \, (\Im(\tau)-\theta_k)^{2 \, \mu_k} \, ,
\eqs
which yields
\begin{align*}
n \, \Re(\tau) \, + \, j \, \Re(\varpi_k(\tau)) \, & \, \leq \, n \, \Re(\tau) \, + \, j \, \left( - \, \dfrac{\Re(\tau)}{\alpha_k} \, + \,
\dfrac{\beta^*}{\alpha_k^{2 \, \mu_k+1}} \, \Re(\tau)^{2 \, \mu_k} \, - \, \dfrac{\beta_*}{\alpha_k^{2 \, \mu_k+1}} \, \left(\Im(\tau)-\theta_k\right)^{2 \, \mu_k} \right) \\
& \, = \, n \, \Re(\tau) \, - \, \frac{j}{\alpha_k} \, \Psi(\epsilon_k) \\
& \, \leq \, - \, n \, c_k \, (\Im(\tau)-\theta_k)^{2 \, \mu_k} \, + \, \frac{n \, \epsilon_k}{\alpha_k} \,
\left( \alpha_k \, - \, \frac{j}{n} \, + \, \frac{j}{n} \, \frac{\beta^*}{\alpha_k^{2 \, \mu_k}} \, \epsilon_k^{2 \, \mu_k-1}\right) \\
& \, \leq \, - \, \frac{n \, \epsilon_k}{\alpha_k} \, \left( \underbrace{\overline{\delta}_k-\alpha_k}_{>0} \, - \, r \, \frac{\beta^*}{\alpha_k^{2 \, \mu_k}}
\, \epsilon_k^{2 \, \mu_k-1} \right)
\end{align*}
for each $\tau\in \Gamma_{\epsilon_k}^k$. Now, since $\overline{\delta}_k-\alpha_k>0$ and $0<\epsilon_k<\epsilon<\Lambda_2^*$, we have
\bqs
\left| \dfrac{1}{2 \, \mathbf{i} \, \pi} \, \int_{\Gamma^{in}_k} {\rm e}^{n \, \tau} \, \mathbf{G}_j(\tau) \, \md \tau \right| \, \leq \,
C \, \exp\left( \, - \, n \, \frac{\epsilon_k \, (\overline{\delta}_k-\alpha_k)}{2 \, \alpha_k} \right) \,,
\eqs
which gives the desired estimate in the region $(j,n) \in \mathcal{D}_*$ and $n \, \overline{\delta}_k\leq j \leq n \, r$.

Let finally comment on the case where $\mathcal{I}_k=\left\{ \nu_{k,1},\nu_{k,2} \right\}$ with $\alpha_{\nu_{k,1}}<0<\alpha_{\nu_{k,2}}$. This time,
Lemma~\ref{lemkgen} gives
\bqs
\forall \, \tau \in B_\epsilon(\tau_k) \, ,\quad \forall \, j \in \Z \, ,\quad \big| \, \mathbf{G}_j(\tau) \, \big| \, \le \, \begin{cases}
C \, \exp \big( j \, \Re(\varpi_{\nu_{k,1}}(\tau)) \big) \, ,& \text{\rm if $j \, \le \, 0$,} \\
C \, \exp \big( j \, \Re(\varpi_{\nu_{k,2}}(\tau)) \big) \, ,& \text{\rm if $j \, \ge \, 1$.}
\end{cases}
\eqs
The analysis for $(j,n)\in \mathcal{D}_*$ and $j\geq 1$ is unchanged, and we apply the same strategy for $(j,n)\in \mathcal{D}_*$ and $j\leq 0$ without any
difficulty. As a conclusion, we have obtained that there exist $C>0$ and $\delta>0$ such that for each $k=1,\cdots,K$ we have
\bqs
\left| \dfrac{1}{2 \, \mathbf{i} \, \pi} \, \int_{\Gamma_k^{in}} {\rm e}^{n \, \tau} \, \mathbf{G}_j(\tau) \, \md \tau \right| \, \leq \, C \, {\rm e}^{-n \, \delta}
\, ,\quad \forall (j,n) \in \mathcal{D}_* \, ,
\eqs
which ends the proof.
\end{proof}

\begin{figure}[t!]
\centering
\includegraphics[width=.5\textwidth]{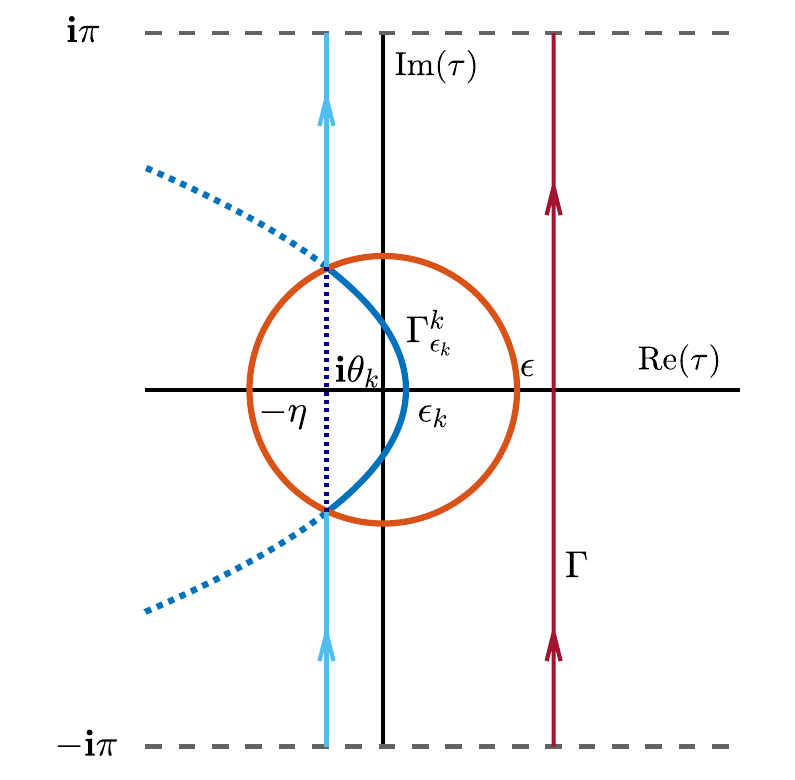}
  \caption{Illustration of the contour $\Gamma^k_{\epsilon_k}$ used in the proof of Lemma~\ref{lem:11} and Lemma~\ref{lem:12}.}
  \label{fig:contourMod}
\end{figure}

\noindent We prove in the next Lemma that we obtain generalized Gaussian estimates in each sector $\mathcal{D}_k$, $k=1,\cdots,K$.

\begin{lemma}
\label{lem:12}
There exists $\widehat{\epsilon}_{*} \in(0,\epsilon_*)$ such that for each $\epsilon\in(0,\widehat{\epsilon}_{*})$ there are constants $C>0$ and $c>0$
such that for any $k=1,\cdots,K$ and $(j,n)\in\mathcal{D}_k$, the following estimate holds:
\bqs
\big| \, \mathcal{G}^n_j \, \big| \, \le \, \frac{C}{n^{\frac{1}{2 \, \mu_k}}} \, \exp \left( - \, c \,
\left( \frac{|j \, - \, \alpha_k \, n|}{n^{\frac{1}{2 \, \mu_k}}} \right)^{\frac{2 \, \mu_k}{2 \, \mu_k-1}} \right) \, .
\eqs
\end{lemma}

\begin{proof}
Let $(j,n)\in\mathcal{D}_k$, that is $n\geq1$ and $n \, \underline{\delta}_k \leq j \leq n \, \overline{\delta}_k$. Assume without loss of generality
that $\alpha_k>0$. We first consider the case where $\mathcal{I}_k=\{k\}$. Once again, we introduce a family of parametrized curves $\Gamma_p^k$
given by
\bqs
\Gamma_p^k \, := \, \left\{ \Re(\tau)-\frac{\beta^*}{\alpha_k^{2 \, \mu_k}} \, \Re(\tau)^{2 \, \mu_k}
+ \frac{\beta_*}{\alpha_k^{2 \, \mu_k}} \, \left(\Im(\tau)-\theta_k\right)^{2 \, \mu_k} = \Psi_k \left(\tau_{p,k}\right)~|~ -\eta \leq \Re(\tau)\leq \tau_{p,k} \right\}
\eqs
with $\Psi_k\left(\tau_{p,k}\right)=\tau_{p,k}-\frac{\beta^*}{\alpha_k^{2 \, \mu_k}} \, \tau_{p,k}^{2 \, \mu_k}$ and $\eta>0$, $\tau_{p,k}>-\eta$ are chosen
as follows. For each $\epsilon\in(0,\epsilon_*)$ we fix $\eta\in(0,\eta_\epsilon)$ such that the curve $\Gamma_p^k$ with $\tau_{p,k}=0$ intersects the
ray $\left\{-\eta+\mathbf{i} \, \ell ~|~ \ell \in[-\pi,\pi]\right\}$ inside the ball $B_\epsilon(\tau_k)$. Furthermore, we let $0<\epsilon_{0,k}<\epsilon$ be defined
as the value of $\tau_{p,k}$ for which $\Gamma_p^k$ with $\tau_{p,k}=\epsilon_{0,k}$ intersects the ray $\left\{-\eta+\mathbf{i} \, \ell ~|~ \ell \in[-\pi,\pi]
\right\}$ on $\partial B_\epsilon(\tau_k)$ with $\eta$ fixed previously. Finally, $\tau_{p,k}$ is now defined as
\bqs
\tau_{p,k} \, := \, \left\{
\begin{split}
\rho_k\left(\frac{\zeta_k}{\gamma_k}\right) &\quad \text{ if } \quad  -\frac{\eta}{2}\leq \rho_k\left(\frac{\zeta_k}{\gamma_k}\right) \leq \epsilon_{0,k}\,,\\
\epsilon_{0,k} &\quad \text{ if } \quad \rho_k\left(\frac{\zeta_k}{\gamma_k}\right) > \epsilon_{0,k} \, , \\
-\frac{\eta}{2} &\quad \text{ if } \quad \rho_k\left(\frac{\zeta_k}{\gamma_k}\right) <- \frac{\eta}{2} \, ,
\end{split}
\right.
\eqs
where $\zeta_k$,  $\gamma_k$ and the function $\rho_k$ are set to
\bqs
\zeta_k \, := \, \dfrac{j \, - \, n \, \alpha_k}{2 \, \mu_k \, n} \, ,\quad \text{ and } \quad
\gamma_k \, := \, \dfrac{j}{n} \, \dfrac{\beta^*}{\alpha_k^{2 \, \mu_k}} \, > \, 0 \, ,
\eqs
with $\rho_k\left(\frac{\zeta_k}{\gamma_k}\right)$  given by
\bqs
\rho_k \left( \dfrac{\zeta_k}{\gamma_k} \right) \, := \, \mathrm{sgn} \left(\frac{\zeta_k}{\gamma_k}\right) \,
\left( \dfrac{|\zeta_k|}{\gamma_k} \right)^{\frac{1}{2 \, \mu_k-1}} \, .
\eqs
In our estimate, we use a contour $\Gamma_k$ which consists of $\Gamma_p^k$ in $B_\epsilon(\tau_k)$ and the ray $\left\{-\eta+\mathbf{i} \, \ell ~|~
\ell \in[-\pi,\pi]\right\}$ otherwise (see Figure~\ref{fig:contourball}, left panel for an illustration in the case $K=3$). Depending on the ratio
$\frac{\zeta_k}{\gamma_k}$, there exists (or not) a portion of the ray $\left\{-\eta+\mathbf{i} \, \ell ~|~ \ell \in[-\pi,\pi]\right\}$ within the ball
$B_\epsilon(\tau_k)$ that we denote $\Gamma^{in}_k$. Note that when $\rho_k\left(\frac{\zeta_k}{\gamma_k}\right) > \epsilon_{0,k}$ we
have $\Gamma^{in}_k=\emptyset$. The analysis along $\Gamma_p^k\cup\Gamma^{in}_k$ is exactly the same as the one conducted in
the proof of Lemma~\ref{lem:9} and we get that there exists $\epsilon_{**}\in(0,\epsilon_*)$ such that for all $\epsilon\in(0,\epsilon_{**})$
one has
\bqs
\left| \dfrac{1}{2 \, \mathbf{i} \, \pi} \, \int_{\Gamma_p^k\cup\Gamma^{in}_k} {\rm e}^{n \, \tau} \, \mathbf{G}_j(\tau) \, \md \tau \right| \leq
\frac{C}{n^{\frac{1}{2 \, \mu_k}}} \, \exp \left( -c \, \left( \dfrac{|j \, - \, \alpha_k \, n|}{n^{\frac{1}{2 \, \mu_k}}} \right)^{\frac{2 \, \mu_k}{2 \, \mu_k-1}}
\right) \, ,\quad (j,n) \in \mathcal{D}_k \, .
\eqs
The fact that one needs to eventually decrease the size of $\epsilon$ comes from Lemma~\ref{lemFImp} which is needed to obtain the generalized
Gaussian bound and prove that
\bqs
\int_{\Gamma_p^k} {\rm e}^{-n c_k^*\left( \Im(\tau)-\theta_k\right)^{2 \, \mu_k}} \, |\md \tau|  \, \leq \, \frac{C}{n^{\frac{1}{2 \, \mu_k}}}\, .
\eqs
Along the ray $\left\{ -\eta + \mathbf{i} \, \ell ~|~ \ell \in[-\pi,\pi]\right\}$, we denote by $\Gamma^{out}$ all portions that lie outside the balls
$B_\epsilon (\tau_\upsilon)$ with $\upsilon \neq k$, and we  get
\bqs
\left| \dfrac{1}{2 \, \mathbf{i} \, \pi} \, \int_{\Gamma^{out}} {\rm e}^{n \, \tau} \, \mathbf{G}_j(\tau) \, \md \tau \right| \, \leq \,
{\rm e}^{- \, n \, \eta \, - \, c \, |j|} \, ,\quad (j,n)\in \mathcal{D}_k \, .
\eqs
Thus, it only remains to estimate the contour integral along the ray $\left\{-\eta+\mathbf{i} \, \ell ~|~ \ell \in[-\pi,\pi]\right\}$ within a ball $B_\epsilon
(\tau_\upsilon)$ with $\upsilon \neq k$, that we denote $\Gamma^{in}_\upsilon$. Let assume first that $\mathcal{I}_\upsilon=\{\upsilon\}$. We split
the analysis in two cases.

\begin{itemize}
\item[(i)] If $\alpha_\upsilon<0$, then we have the estimate
\bqs
\left| \, \mathbf{G}_j(\tau) \, \right| \, \leq \, C \, {\rm e}^{- \, c \, j} \, , \quad \tau \in \Gamma^{in}_\upsilon \, ,
\eqs
as $0 < n \, \underline{\delta}_k \leq j \leq n \, \overline{\delta}_k$, and we obtain
\bqs
\left| \dfrac{1}{2 \, \mathbf{i} \, \pi} \, \int_{\Gamma^{in}_\upsilon} {\rm e}^{n \, \tau} \, \mathbf{G}_j(\tau) \, \md \tau \right| \, \leq \,
C \, {\rm e}^{- \, n \, \eta \, - \, c \, j} \, ,\quad (j,n) \in \mathcal{D}_k \, .
\eqs
\item[(ii)] If $\alpha_\upsilon>0$, then we have the following estimates for each $\tau\in \Gamma^{in}_\upsilon \subset B_\epsilon(\tau_\upsilon)$
\bqs
\left| \, \mathbf{G}_j(\tau) \, \right| \, \leq \, C \, {\rm e}^{\, j \, \Re \left(\varpi_\upsilon(\tau)\right)} \, ,\quad \tau \in \Gamma^{in}_\upsilon \, ,
\eqs
and
\bqs
\Re(\varpi_\upsilon(\tau)) \, \leq \, - \, \dfrac{\Re(\tau)}{\alpha_\upsilon} \, + \, \dfrac{\beta^*}{\alpha_\upsilon^{2 \, \mu_\upsilon+1}} \, \Re(\tau)^{2 \, \mu_\upsilon}
\, - \, \dfrac{\beta_*}{\alpha_\upsilon^{2 \, \mu_\upsilon+1}} \, \left(\Im(\tau)-\theta_\upsilon\right)^{2 \, \mu_\upsilon} \, .
\eqs
 As a consequence, we readily obtain that
\bqs
- \, n \, \eta \, + \, j \, \Re(\varpi_\upsilon(\tau)) \leq n \, \eta \, \left( - \, 1 \, + \, \frac{j}{n \, \alpha_\upsilon}
\, + \, \frac{j}{n} \, \dfrac{\beta^*}{\alpha_\upsilon^{2 \, \mu_\upsilon+1}} \, \eta^{2 \, \mu_\upsilon-1} \right) \, .
\eqs
Thus if $\alpha_\upsilon>\alpha_k$, we get that
\bqs
- \, n \, \eta \, + \, j \, \Re(\varpi_\upsilon(\tau)) \, \leq \, - \, n \, \eta \, \left( \underbrace{1 \, - \, \dfrac{\overline{\delta}_k}{\alpha_\upsilon}}_{>0} \, - \,
\dfrac{ \overline{\delta}_k \, \beta^*}{\alpha_\upsilon^{2 \, \mu_\upsilon+1}} \, \eta^{2 \, \mu_\upsilon-1}  \right) \, \leq \,
- \, \frac{ \, n \, \eta}{2} \, \left( 1 \, - \, \dfrac{\overline{\delta}_k}{\alpha_\upsilon}\right),
\eqs
provided that $\eta$ is chosen small enough, which is always possible by eventually reducing the size of $\epsilon_{**}$. This gives
\bqs
\left| \dfrac{1}{2 \, \mathbf{i} \, \pi} \, \int_{\Gamma^{in}_\upsilon} {\rm e}^{n \, \tau} \, \mathbf{G}_j(\tau) \, \md \tau \right| \, \leq \, C \,
{\rm e}^{- \, n \, \frac{\eta}{2} \, \left( 1-\frac{\overline{\delta}_k}{\alpha_\upsilon} \right)} \, .
\eqs

Finally, if $\alpha_\upsilon<\alpha_k$ we use a different contour inside the ball $B_\epsilon(\tau_\upsilon)$. Namely, we use the contour
\bqs
\Gamma_{\epsilon_{\upsilon}}^\upsilon \, := \, \left\{ \Re(\tau)-\frac{\beta^*}{\alpha_\upsilon^{2 \, \mu_\upsilon}} \,
\Re(\tau)^{2 \, \mu_\upsilon} + \frac{\beta_*}{\alpha_\upsilon^{2 \, \mu_\upsilon}} \, \left(\Im(\tau)-\theta_\upsilon\right)^{2 \, \mu_\upsilon}
\, = \, \Psi_\upsilon\left(\epsilon_\upsilon \right)~|~ -\eta \leq \Re(\tau)\leq \epsilon_\upsilon \right\} \, ,
\eqs
with $\Psi_\upsilon\left(\epsilon_\upsilon\right) := \epsilon_\upsilon-\frac{\beta^*}{\alpha_\upsilon^{2 \, \mu_\upsilon}} \, \epsilon_\upsilon^{2 \, \mu_\upsilon}$
and where $0<\epsilon_\upsilon<\epsilon$ is chosen such that $\Gamma_{\epsilon_\upsilon}^\upsilon$ intersects the segment $\Gamma_\upsilon^{in}$
precisely on the boundary of $B_\epsilon(\tau_\upsilon)$. We note that there exists some constant $c_\upsilon>0$ such that for any $\tau \in
\Gamma_{\epsilon_\upsilon}^\upsilon \subset B_\epsilon(\tau_\upsilon)$, one has
\bqs
\Re(\tau) \, \leq \, \epsilon_\upsilon \, - \, c_\upsilon \, (\Im(\tau)-\theta_\upsilon)^{2 \, \mu_\upsilon} \, ,
\eqs
which yields
$$
n \, \Re(\tau) \, + \, j \, \Re(\varpi_\upsilon(\tau)) \, \leq \, - \, n \, c_\upsilon \, (\Im(\tau)-\theta_\upsilon)^{2 \, \mu_\upsilon}
\, - \, \frac{n \, \epsilon_\upsilon}{\alpha_\upsilon} \, \left(
\underline{\delta}_k \, - \, \alpha_\upsilon \, - \, \overline{\delta}_k \, \frac{\beta^*}{\alpha_\upsilon^{2 \, \mu_\upsilon}}
\, \epsilon_\upsilon^{2 \, \mu_\upsilon-1} \right) \, ,
$$
for each $\tau \in \Gamma_{\epsilon_\upsilon}^\upsilon$. Now, since $\underline{\delta}_k-\alpha_\upsilon>\alpha_k-\alpha_\upsilon>0$ and
$0<\epsilon_\upsilon<\epsilon$, we can always further reduce the size of $\epsilon_{**}$ such that
\bqs
\left| \dfrac{1}{2 \, \mathbf{i} \, \pi} \, \int_{\Gamma^{in}_\upsilon} {\rm e}^{n \, \tau} \, \mathbf{G}_j(\tau) \, \md \tau \right| \, \leq \,
\dfrac{C}{n^{\frac{1}{2 \, \mu_\upsilon}}} \, \exp \left( \, - \, n \, \frac{\epsilon_\upsilon \, (\underline{\delta}_k-\alpha_\upsilon)}{2 \, \alpha_\upsilon}\right) \, ,
\eqs
which gives the desired estimate.
\end{itemize}

If now $\mathcal{I}_\upsilon = \{ \nu_{\upsilon,1},\nu_{\upsilon,2} \}$, then we have $\alpha_{\nu_{\upsilon,1}} < 0 < \alpha_{\nu_{\upsilon,2}}$ and for
$0 < n \, \underline{\delta}_k \leq j \leq n \, \overline{\delta}_k$, we get
\bqs
\left| \, \mathbf{G}_j(\tau) \, \right| \, \leq \, C \, \exp \big( j \, \Re \left(\varpi_{\nu_{\upsilon,2}}(\tau)\right) \big) \, ,\quad \tau \in \Gamma^{in}_\nu \, ,
\eqs
such that the analysis is similar to the above case (ii). Finally, when $\mathcal{I}_k=\{ \nu_{k,1},\nu_{k,2}\}$, we necessarily have that $\alpha_{\nu_{k,1}}
< 0 < \alpha_{\nu_{k,2}}=\alpha_k$ and the analysis remains unchanged. As there exists some $\widehat{\epsilon}_*\in(0,\epsilon_{**})$ such that for all
$\epsilon\in(0,\widehat{\epsilon}_*)$ we have proved the desired generalized Gaussian bound.
\end{proof}

\paragraph{Case B.} In that case, two or more $\alpha_k$ are equal. Note that for $(j,n)\in\mathcal{D}_*$ the analysis remains unchanged and
Lemma~\ref{lem:11} still holds true in that case. Let us assume for simplicity that $\alpha_{\nu_1}=\alpha_{\nu_2}$ for some couple of integers
$\nu_1 \neq \nu_2$, and all other $\alpha_k$'s are distinct. The estimate from Lemma~\ref{lem:12} is still valid for $(j,n)\in\mathcal{D}_k$ for each
$k \not \in \{ \nu_1,\nu_2 \}$. For $(j,n)\in\mathcal{D}_{\nu_1}=\mathcal{D}_{\nu_2}$, in the ball $B_\epsilon(\tau_{\nu_1})$ we use the contour
$\Gamma_p^{\nu_1}$ and in the ball $B_\epsilon(\tau_{\nu_2})$ we use the contour $\Gamma_p^{\nu_2}$. And we refer to Figure~\ref{fig:contourball}
for an illustration of such contours. Reproducing the analysis of Lemma~\ref{lem:12}, we obtain the existence of $C>0$ and $c>0$ such that for
$(j,n) \in \mathcal{D}_{\nu_1}=\mathcal{D}_{\nu_2}$ the following estimate holds:
\bqs
\big| \, \mathcal{G}^n_j \, \big| \, \le \, C \, \sum_{\nu \in \{ \nu_1,\nu_2\}} \, \dfrac{1}{n^{\frac{1}{2 \, \mu_{\nu}}}} \, \exp \left( - \, c \,
\left( \dfrac{|j \, - \, \alpha_\nu \, n|}{n^{\frac{1}{2 \, \mu_{\nu}}}} \right)^{\frac{2 \, \mu_{\nu}}{2 \, \mu_{\nu}-1}} \right) \, .
\eqs

Finally, we remark that Lemma~\ref{lem:10} naturally extends to the case $K>1$ in the implicit setting. This concludes the proof of Theorem~\ref{thm:1}.

\section{Examples and extensions}
\label{section4}

We first give several examples of operators \eqref{operateurs} that fit into the framework of Theorem \ref{thm:1}, and that arise when discretizing
the transport equation:
\begin{equation}
\label{transport}
\partial_t u \, + \, \partial_x u \, = \, 0 \, ,\quad (t,x) \in \R^+ \times \R \, ,
\end{equation}
with Cauchy data at $t=0$. We refer to \cite{gko,Gustafsson-livre} for a detailed analysis and more examples of finite difference schemes in that context.

\subsection{Example 1: the Lax-Friedrichs scheme}

The Lax-Friedrichs scheme is an explicit finite difference approximation of \eqref{transport}, which corresponds to the operators:
\begin{equation}
\label{laxfriedrichs}
Q_1 \, := \, I \, ,\quad Q_0 \, := \, \dfrac{1\, + \, \lambda}{2} \, \, {\bf S}^{-1} \, + \, \dfrac{1\, - \, \lambda}{2} \, \, {\bf S} \, ,
\end{equation}
where here and below, $\lambda$ is a real parameter\footnote{In the theory of finite difference schemes, it is referred to as the
Courant-Friedrichs-Lewy parameter \cite{cfl}.} and ${\bf S}$ still denotes the shift operator defined by:
$$
{\bf S}  \quad : \quad \big( u_j \big)_{j \in \Z} \, \longmapsto \, \big( u_{j+1} \big)_{j \in \Z} \, .
$$
We now restrict to $\lambda \in (0,1)$ so that both coefficients in the definition \eqref{laxfriedrichs} are positive and they sum to $1$. In probability theory,
this corresponds to a random walk with probability $(1+\lambda)/2$ to jump of $+1$ and probability $(1-\lambda)/2$ to jump of $-1$ at each time iteration
(recall our convention on the coefficients $a_\ell$ which differs from the standard convolution product).

In the notation of \eqref{operateurs}, we have $r=p=1$. Since we are dealing here with an explicit scheme, Assumptions \ref{hyp:0} and \ref{hyp:3} are
trivially satisfied. The definition \eqref{defF} reduces here to:
$$
F \big( \, {\rm e}^{\, \mathbf{i} \, \xi} \, \big) \, = \, \cos (\xi )\, - \, \mathbf{i} \, \lambda \, \sin (\xi )\, .
$$
Computing:
$$
\Big| \, F \big( \, {\rm e}^{\, \mathbf{i} \, \xi} \, \big) \, \Big|^2 \, = \, \cos^2 (\xi) \, + \, \lambda^2 \, \sin^2( \xi )\, ,
$$
we find that $F(\kappa)$ belongs to $\Dbar$ for all $\kappa \in \cercle$, and $F(\kappa)$ belongs to $\cercle$ for such $\kappa$ if and
only if $\kappa = \pm 1$. We thus have \eqref{hyp:stabilite1} with $\underline{\kappa}_1 := 1$ and $\underline{\kappa}_2 := -1$, and
the reader can check that \eqref{hyp:stabilite2} is satisfied with:
$$
\alpha_1 \, = \, \alpha_2 \, = \, \lambda \, ,\quad \beta_1 \, = \, \beta_2 \, = \, \dfrac{1 \, - \, \lambda^2}{2} \, , \quad \mu_1  =  \mu_2 = 1\, ,
$$
which means that Assumption \ref{hyp:1} is satisfied. Since the modulus of $F(\kappa)$ attains its maximum at two points of $\cercle$, we cannot
apply the uniform Gaussian bound from \cite{Diaconis-SaloffCoste}. The improvements of Theorem~\ref{thm:1} and \cite[Theorem 1.8]{RSC2} are
relevant in that case. Note that \cite[Theorem 1.8]{RSC2} can be used for the Lax-Friedrichs scheme \eqref{laxfriedrichs} since we are in the case
where $\alpha_1 = \alpha_2$, $\beta_1 = \beta_2$ and $\mu_1  = \mu_2$. The spectral curve $F(\mathbb{S}^1)$ (an ellipse) is illustrated in
Figure~\ref{fig:LF} in the case $\lambda=1/2$.

We now turn to Assumption \ref{hyp:2} and compute (see the general definition \eqref{defAl}):
$$
\A_{-1} (z) \, = \, - \, \dfrac{1\, + \, \lambda}{2} \, ,\quad \A_1 (z) \, = \, - \, \dfrac{1\, - \, \lambda}{2} \, .
$$
Hence Assumption \ref{hyp:2} is satisfied too.

At last, Assumption \ref{hyp:4} is satisfied since we have here $K=2$, $\underline{z}_1=1$ and $\underline{z}_2=-1 \neq \underline{z}_1$,
which means that both sets $\mathcal{I}_1$ and $\mathcal{I}_2$ in \eqref{defIk} are singletons. Overall, the conclusion of Theorem \ref{thm:1}
for the Lax-Friedrichs scheme in \eqref{laxfriedrichs} is the uniform bound\footnote{We do not use the absolute value here since all coefficients
$\mathcal{G}^n_j$ are nonnegative.}:
$$
\mathcal{G}^n_j \, \le \, \dfrac{C}{\sqrt{n}} \, \exp \left( - \, c \, \, \dfrac{(j \, - \, \lambda \, n)^2}{n} \right) \, .
$$
This behavior is illustrated in Figure \ref{fig:LF} in the case $\lambda=1/2$.

\begin{figure}[t!]
\centering
\includegraphics[width=.39\textwidth]{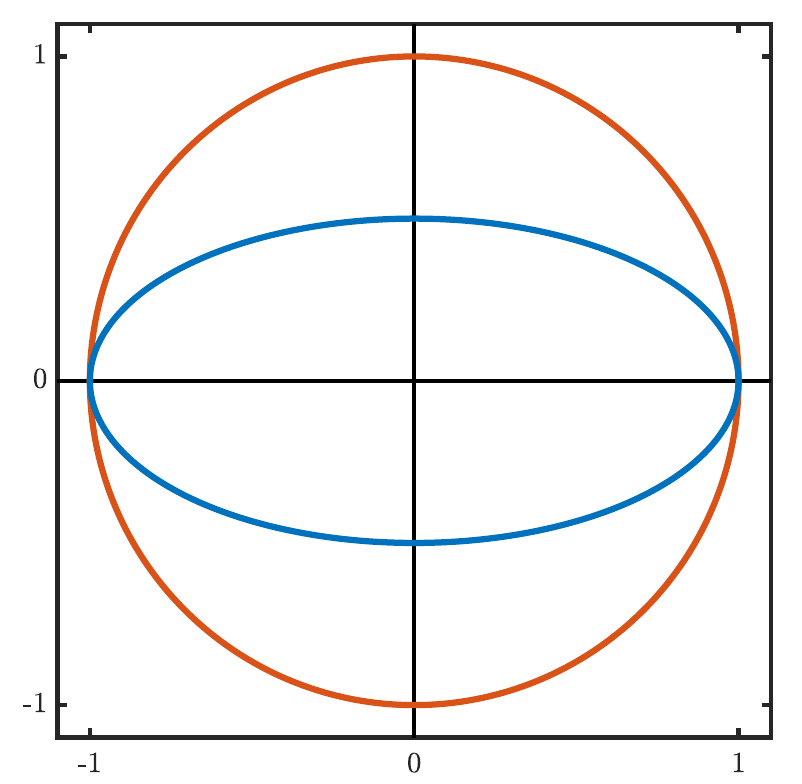}\hspace{0.5cm}
\includegraphics[width=.5\textwidth]{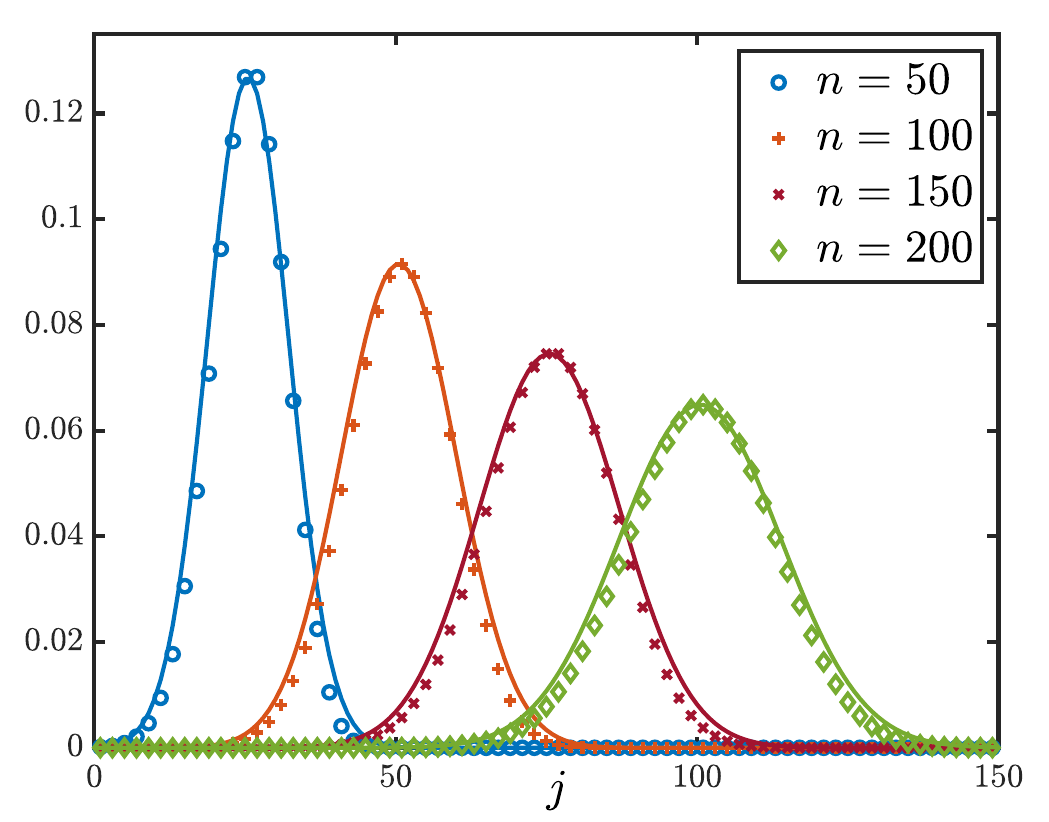}
  \caption{Left: Spectrum (blue curve) $\sigma(\mathcal{L})=F(\mathbb{S}^1)$ for the Lax-Friedrichs scheme \eqref{laxfriedrichs} with $\lambda=1/2$.
  Right: The Green's function (marked points) at different time iterations for the Lax-Friedrichs scheme \eqref{laxfriedrichs} compared with a fixed
  Gaussian profile centered at $j=\lambda n$ (solid lines). We started with an initial condition given by the Dirac mass $\boldsymbol{\delta}$.}
  \label{fig:LF}
\end{figure}

For readers who are familiar with the theory of the transport equation (see \cite{gko} otherwise), the parameter $\lambda$ stands for the ratio
$\Delta t/\Delta x$ of the time and space steps. Hence the bound of Theorem \ref{thm:1} equivalently reads (with new constants that are still
denoted $C$ and $c$):
$$
\mathcal{G}^n_j \, \le \, \dfrac{C}{\sqrt{n}} \, \exp \left( - \, c \, \, \dfrac{(j \, \Delta x - \, n \, \Delta t)^2}{\Delta x \, (n \, \Delta t)} \right) \, ,
$$
which corresponds to the heat kernel at point $j \, \Delta x$, time $n \, \Delta t$ with a diffusion coefficient proportional to $\Delta x$.

\subsection{Example 2: an implicit scheme}

Our second example is based on the so-called method of lines for discretizing \eqref{transport} (see \cite{gko,Gustafsson-livre} for a detailed
exposition of the method and its outcome). Here we first apply the centered finite difference for the spatial derivative and we then apply the
implicit Euler scheme for the time integration. As in the case of the Lax-Friedrichs scheme \eqref{laxfriedrichs}, we introduce a positive
parameter $\lambda>0$ (which plays the role of the ratio $\Delta t/\Delta x$ but its origin is meaningless here), and we use the operators:
\begin{equation}
\label{implicit}
Q_1 \, := \, I \, + \, \dfrac{\lambda}{2} \, \big( \, {\bf S} \, - {\bf S}^{-1} \, \big) \, ,\quad Q_0 \, := \, I \, .
\end{equation}
In the notation of \eqref{operateurs}, this corresponds again to $r=p=1$, but the scheme is now implicit because $Q_1$ is not the identity.
We compute:
$$
\widehat{Q}_1 \big( \, {\rm e}^{\, \mathbf{i} \, \xi} \, \big) \, = \, 1 \, + \, \mathbf{i} \, \lambda \, \sin (\xi )\neq 0 \, ,
$$
which means that $Q_1$ is an isomorphism on $\ell^2(\Z;\C)$. The index condition \eqref{indice} is also satisfied since the complex number
$\widehat{Q}_1(\kappa)$ has positive real part for all $\kappa \in \cercle$ so we can write:
$$
\widehat{Q}_1(\kappa) \, = \, \exp q(\kappa) \, ,
$$
thanks to the standard determination of the logarithm (which implies the validity of \eqref{indice}). The operator $\mathcal{L}$ is given by:
$$
\mathcal{L} \, = \, \dfrac{1}{\sqrt{1+\lambda^2}} \, \left\{ \sum_{\ell \ge 0} \, x^\ell \, {\bf S}^{- \ell} \, +\sum_{\ell \ge 1} \, (-1)^\ell \, x^\ell \, {\bf S}^\ell
\, \right\} \, ,
$$
where $x \in (0,1)$ is given by:
$$
x \, := \, \dfrac{\sqrt{1+\lambda^2} \, - \, 1}{\lambda} \, .
$$
We are thus dealing with a convolution operator with infinite support.

We compute:
$$
F \big( \, {\rm e}^{\, \mathbf{i} \, \xi} \, \big) \, = \, \dfrac{1}{1 \, + \, \mathbf{i} \, \lambda \, \sin (\xi)} \, ,
$$
which means that $F(\kappa)$ belongs to $\Dbar$ for all $\kappa \in \cercle$ and, again, $F(\kappa)$ belongs to $\cercle$ if and only if
$\kappa =\pm 1$ ($K=2$ in the notation of Assumption \ref{hyp:1}). Setting $\underline{\kappa}_1=1$ and $\underline{\kappa}_2=-1$, we
find that the relation \eqref{hyp:stabilite2} is satisfied with:
$$
\alpha_1 \, = \, \lambda \, ,\quad \beta_1 \, = \, \beta_2 \, = \, \dfrac{\lambda^2}{2} \, ,\quad \alpha_2 \, = \, - \, \lambda \, .
$$
Assumption \ref{hyp:1} is thus satisfied but we now have $\underline{z}_1 =\underline{z}_2 =1$, and we immediately see that Assumption
\ref{hyp:4} is also satisfied: both sets $\mathcal{I}_1$ and $\mathcal{I}_2$ equal $\{ 1,2 \}$ and $\alpha_1 \, \alpha_2 =-\lambda^2<0$. The
spectral curve $F(\mathbb{S}^1)$ is illustrated in Figure~\ref{fig:Implicit} in the case $\lambda=1/2$.

As far as Assumption \ref{hyp:2} is concerned, we compute:
$$
\A_{-1} (z) \, = \, - \, \dfrac{\lambda}{2} \, z \, ,\quad \A_1 (z) \, = \, \dfrac{\lambda}{2} \, z \, ,
$$
so Assumption \ref{hyp:2} is satisfied again. We also note that Assumption \ref{hyp:3} is satisfied since we have $a_{-1,1}=-\lambda/2$ and
$a_{1,1}=\lambda/2$. We can therefore apply Theorem \ref{thm:1} which, in the case of \eqref{implicit}, yields the uniform Gaussian bound:
$$
\big| \, \mathcal{G}^n_j \, \big| \, \le \, \dfrac{C}{\sqrt{n}} \, \left( \exp \left( - \, c \, \, \dfrac{(j \, + \, \lambda \, n)^2}{n} \right)
\, + \, \exp \left( - \, c \, \, \dfrac{(j \, - \, \lambda \, n)^2}{n} \right) \right) \, ,\quad | \, j \, | \, \le \, L \, n \, .
$$
This behavior is illustrated in Figure \ref{fig:Implicit} in the case $\lambda=1/2$. Since we have an explicit formula for $\mathcal{G}_j^1$, it is
clear that the bound:
$$
\big| \, \mathcal{G}^1_j \, \big| \, \le \, C \, \exp (- \, c \, | \, j \, | ) \, ,
$$
for large $j$'s cannot be improved to some generalized Gaussian bound. This justifies why we need to distinguish the cases $|j|/n \gg 1$ and
$|j|/n =O(1)$ in \eqref{boundimp}.

\begin{figure}[t!]
\centering
\includegraphics[width=.39\textwidth]{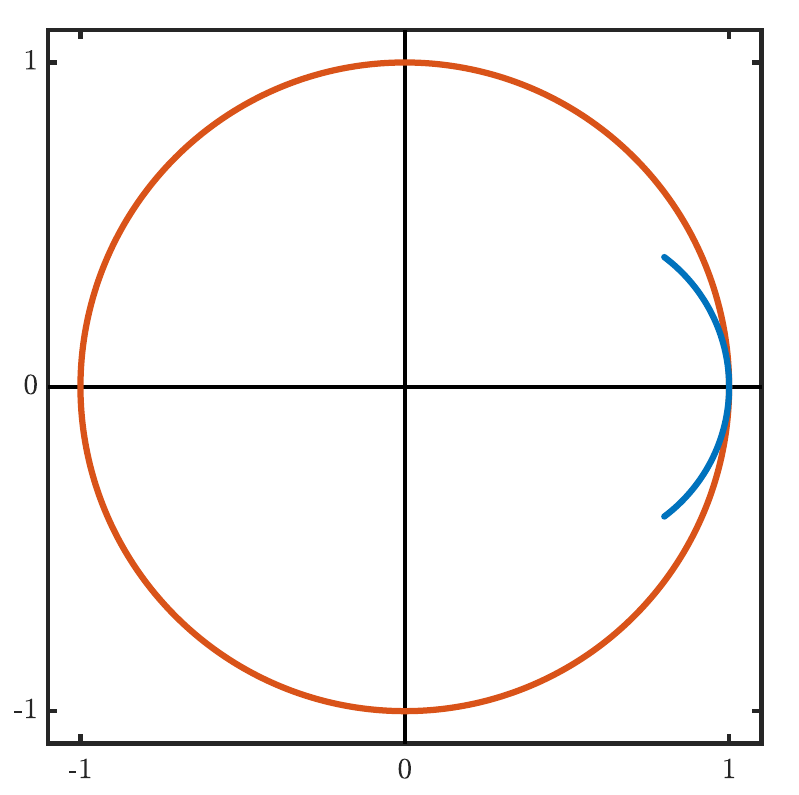}\hspace{0.5cm}
\includegraphics[width=.5\textwidth]{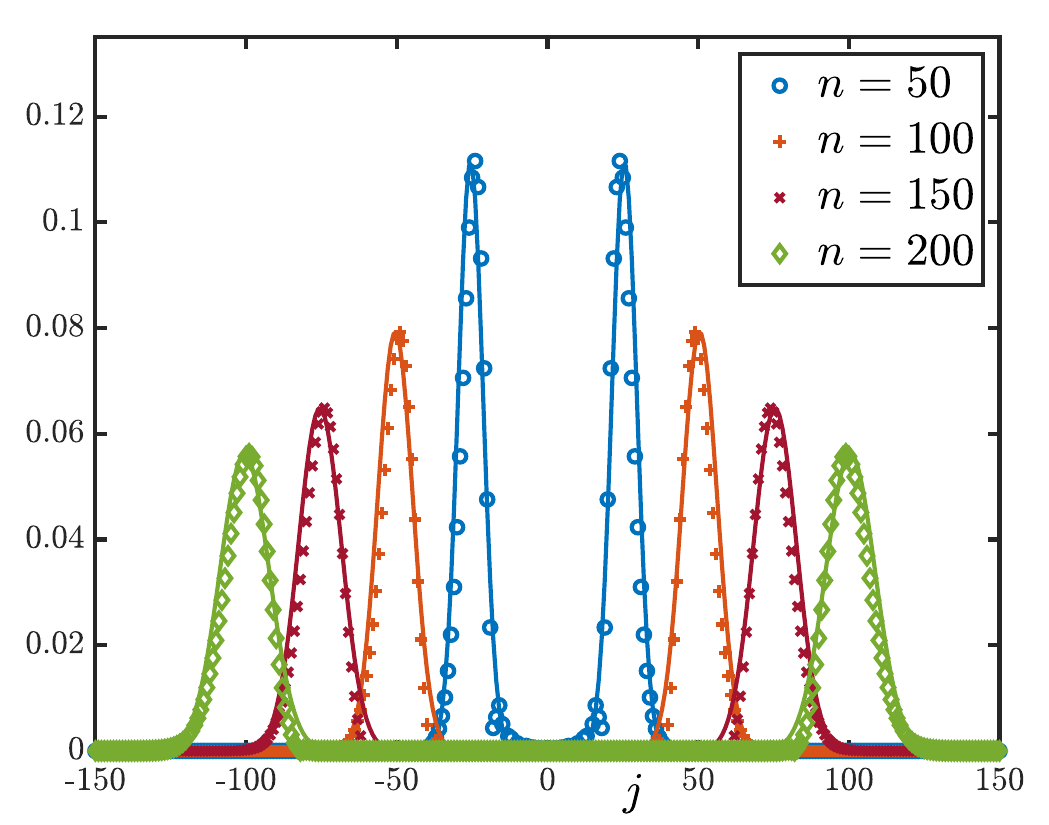}
  \caption{Left: Spectrum (blue curve) $\sigma(\mathcal{L})=F(\mathbb{S}^1)$ for the implicit scheme \eqref{implicit} with $\lambda=1/2$.
  Right: The absolute value of the Green's function (marked points) at different time iterations for the implicit scheme \eqref{implicit} compared
  with two fixed Gaussian profiles centered at $j=\lambda \, n$ and $j=-\lambda \, n$ (solid lines). We started with an initial condition given by
  the Dirac mass $\boldsymbol{\delta}$.}
  \label{fig:Implicit}
\end{figure}

\subsection{Example 3: the $O3$ scheme}

Next, as a third example, we consider the $O3$ scheme \cite{Despres1,Despres2} which is an explicit scheme of order $3$ obtained as the
convex combination of the Lax-Wendroff scheme \cite{LW} and the Beam-Warming scheme \cite{warming}. The Lax-Wendroff scheme is an
explicit finite approximation of \eqref{transport} corresponding to the operators:
\bqs
Q_1^{LW}:=I, \quad Q_0^{LW}:=(1-\lambda^2)I+\frac{\lambda+\lambda^2}{2} \mathbf{S}^{-1}+\frac{-\lambda+\lambda^2}{2} \mathbf{S} \,,
\eqs
for some $\lambda\in(0,1)$. On the other hand, the Beam-Warming scheme is an explicit scheme given by
\bqs
Q_1^{BW}:=I, \quad Q_0^{BW}:=\left(1-\frac{3}{2}\lambda+\frac{1}{2}\lambda^2\right)I+\left(2\lambda-\lambda^2\right) \mathbf{S}^{-1}
+\frac{-\lambda+\lambda^2}{2} \mathbf{S}^{-2} \,,
\eqs
for some $\lambda \in(0,1)$. In both cases, $\lambda$ stands for the ratio $\Delta t / \Delta x$. The $O3$ scheme is then defined as the following
convex combination of the above two schemes
\bqq
Q_1^{O3}:=I, \quad Q_0^{O3}:= \left(1-\delta\right) Q_0^{LW}+\delta Q_0^{BW}, \text{ with } \delta= \frac{1+\lambda}{3}.
\label{O3scheme}
\eqq
The expression for $Q_0^{O3}$ can be simplified and we have
\bqs
Q_0^{O3}=\frac{(2-\lambda)(1-\lambda^2)}{2}I+\frac{\lambda(2-\lambda)(1+\lambda)}{2}\mathbf{S}^{-1}
-\frac{\lambda(1-\lambda^2)}{6}\mathbf{S}^{-2}+\frac{\lambda(2-\lambda)(\lambda-1)}{6}\mathbf{S}.
\eqs
In the notation of \eqref{operateurs}, this corresponds to $r=2$ and $p=1$. Once again, since we are dealing here with an explicit scheme,
Assumptions \ref{hyp:0} and \ref{hyp:3} are  trivially satisfied. The definition \eqref{defF} gives in that case
\bqs
F^{O3} \big( \, {\rm e}^{\, \mathbf{i} \, \xi} \, \big) \, = \, \left(1-\delta\right) F^{LW} \big( \, {\rm e}^{\, \mathbf{i} \, \xi} \, \big)
+\delta F^{BW} \big( \, {\rm e}^{\, \mathbf{i} \, \xi} \, \big),
\eqs
with
\bqs
F^{LW} \big( \, {\rm e}^{\, \mathbf{i} \, \xi} \, \big) \, = \, 1-\lambda^2 + \lambda^2 \cos(\xi)-\mathbf{i} \lambda \sin(\xi),
\eqs
and
\bqs
F^{BW} \big( \, {\rm e}^{\, \mathbf{i} \, \xi} \, \big) \,
= \, 1-\frac{3}{2}\lambda+\frac{1}{2}\lambda^2+\left(2\lambda-\lambda^2\right)\, {\rm e}^{-\, \mathbf{i} \, \xi}
+\frac{-\lambda+\lambda^2}{2} \, {\rm e}^{-\, \mathbf{i} \, 2 \, \xi}.
\eqs
And after some computations, we get
\begin{align*}
\left| F^{LW} \big( \, {\rm e}^{\, i \, \xi} \, \big) \right|^2 \, &= \, 1 -4\lambda^2(1-\lambda^2)\sin^4\left(\frac{\xi}{2}\right) \, , \\
\left| F^{BW} \big( \, {\rm e}^{\, i \, \xi} \, \big) \right|^2 \, &= \, 1 -4\lambda(2-\lambda)(1-\lambda)^2\sin^4\left(\frac{\xi}{2}\right) \,,
\end{align*}
from which we deduce by convexity that $F^{O3} (\kappa)$ belongs to $\overline{\D}$ for all $\kappa\in \mathbb{S}^1$. In fact, further computations lead to
\bqs
\left| F^{O3} \big( \, {\rm e}^{\, \mathbf{i} \, \xi} \, \big) \right|^2 \, = \, 
1- \frac{4}{9}\lambda(2-\lambda)(1-\lambda^2)\sin^4\left(\frac{\xi}{2}\right) \left( 3+4\lambda(1-\lambda)\sin^2\left(\frac{\xi}{2}\right)\right) \, ,
\eqs
such that $F^{O3} (\kappa)$ belongs to $\mathbb{S}^1$ only when $\kappa=1$, that is $K=1$ in the notation of Assumption~\ref{hyp:2} with 
$\underline{\kappa}_1=1$. We find that the relation \eqref{hyp:stabilite2} is satisfied with:
$$
\alpha_1 \, = \, \lambda \, ,\quad \beta_1 \, =  \, \frac{\lambda(2-\lambda)(1-\lambda^2)}{24} \, , \quad \mu_1 \, =\, 2,
$$
and that $\beta_1>0$ for $\lambda\in(0,1)$. Furthermore, note that Assumption~\ref{hyp:4} is trivially satisfied since $\mathcal{I}_1=\left\{\underline{z}_1\right\}$
with $\underline{z}_1=1$. The spectral curve $F^{O3} (\mathbb{S}^1)$ is illustrated in Figure~\ref{fig:O3} in the case $\lambda=1/2$.

\begin{figure}[t!]
\centering
\includegraphics[width=.35\textwidth]{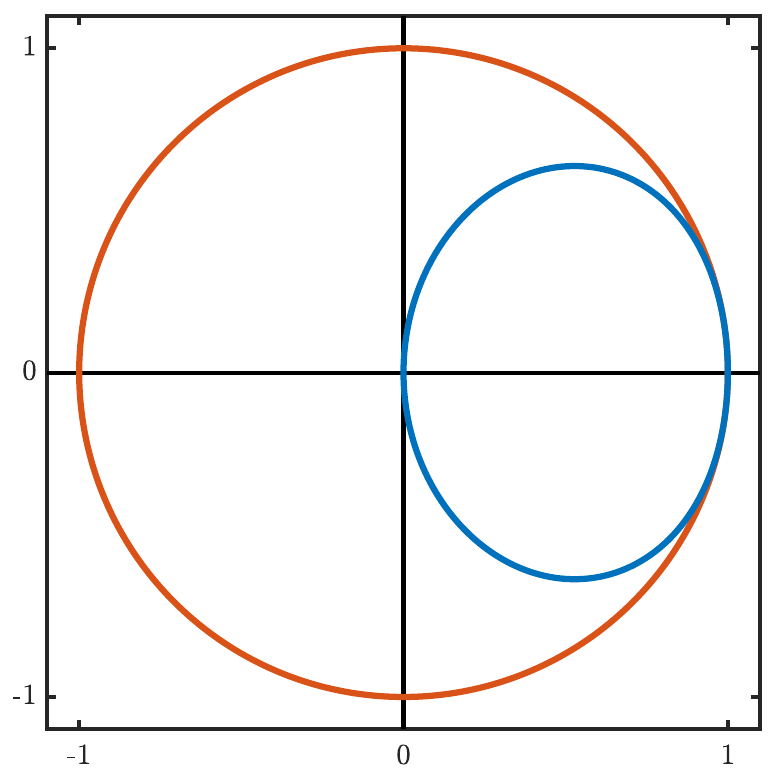}\hspace{0.5cm}
\includegraphics[width=.5\textwidth]{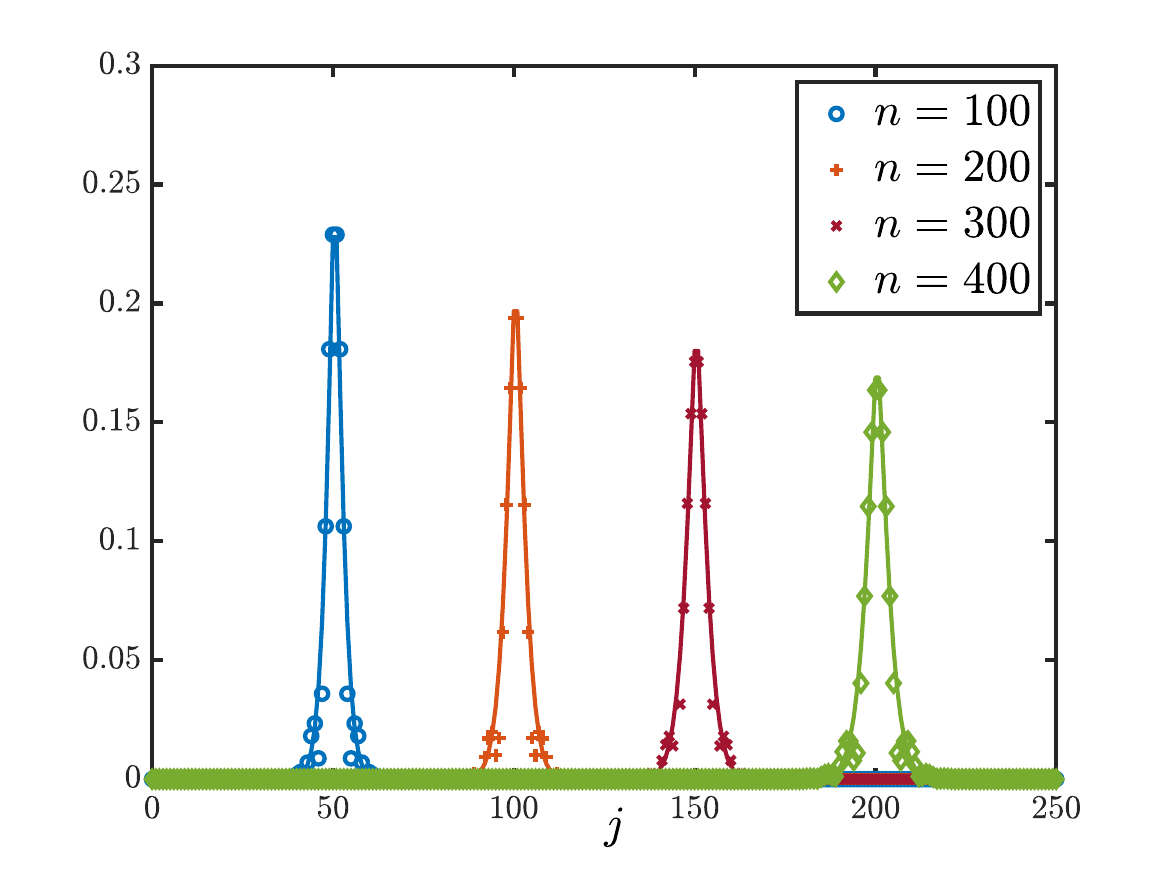}
  \caption{Left: Spectrum (blue curve) $\sigma(\mathcal{L})=F(\mathbb{S}^1)$ for the $O3$ scheme \eqref{O3scheme} with $\lambda=1/2$.
  Right: The absolute value of the Green's function (marked points) at different time iterations for the $O3$ scheme \eqref{O3scheme} compared
  with a fixed generalized Gaussian profile $\mathscr{H}_j^n$ \eqref{defHjn} centered at $j=\lambda \, n$  (solide lines). We started with an initial
  condition given by the Dirac mass $\boldsymbol{\delta}$.}
  \label{fig:O3}
\end{figure}

Next, we have that
\bqs
\mathbb{A}_{-2}(z)=\frac{\lambda(1-\lambda^2)}{6} \, , \quad \text{ and } \quad \mathbb{A}_{1}(z)=\frac{\lambda(2-\lambda)(1-\lambda)}{6} \, ,
\eqs
such that Assumption~\ref{hyp:2} also holds true for each $\lambda\in(0,1)$. We can therefore apply Theorem \ref{thm:1} which, in the case of
\eqref{O3scheme}, yields the uniform generalized Gaussian bound:
$$
\big| \, \mathcal{G}^n_j \, \big| \, \le \, \dfrac{C}{n^{\frac{1}{4}}} \,  \exp \left( - \, c \, \, \dfrac{|j \, - \, \lambda \, n|^{\frac{4}{3}}}{n^{\frac{1}{3}}} \right)\, , \,
\quad n \geq 1\, , \quad j \in \Z \, .
$$
This behavior is illustrated in Figure \ref{fig:O3} in the case $\lambda=1/2$ where we compare the Green's function $\mathscr{G}_j^n$ to the generalized
Gaussian profile $\mathscr{H}_j^n$ defined as
\bqq
\mathscr{H}_j^n:= \dfrac{C}{n^{\frac{1}{4}}} \,  \exp \left( - \, c \, \, \dfrac{|j \, - \, \lambda \, n|^{\frac{4}{3}}}{n^{\frac{1}{3}}} \right)\:,  \, \quad n\geq 1\, , \quad j\in\Z\,,
\label{defHjn}
\eqq
with two fixed constants $C>0$ and $c>0$ independent of $j$ and $n$ which we have set to $C=0.8$ and $c=1.1765$ respectively. Note that
the scaling factor $n^{-\frac{1}{4}}$ in the generalized Gaussian bound is further demonstrated in Figure~\ref{fig:logplot} where we represent
$\sup_{j\in\Z} \big| \, \mathcal{G}^n_j \, \big|$ in logarithmic scale. Using a best linear fit, we numerically obtain a slope of $-0.2496$ which is
in good agreement with the theory.

\begin{figure}[t!]
\centering
\includegraphics[width=.5\textwidth]{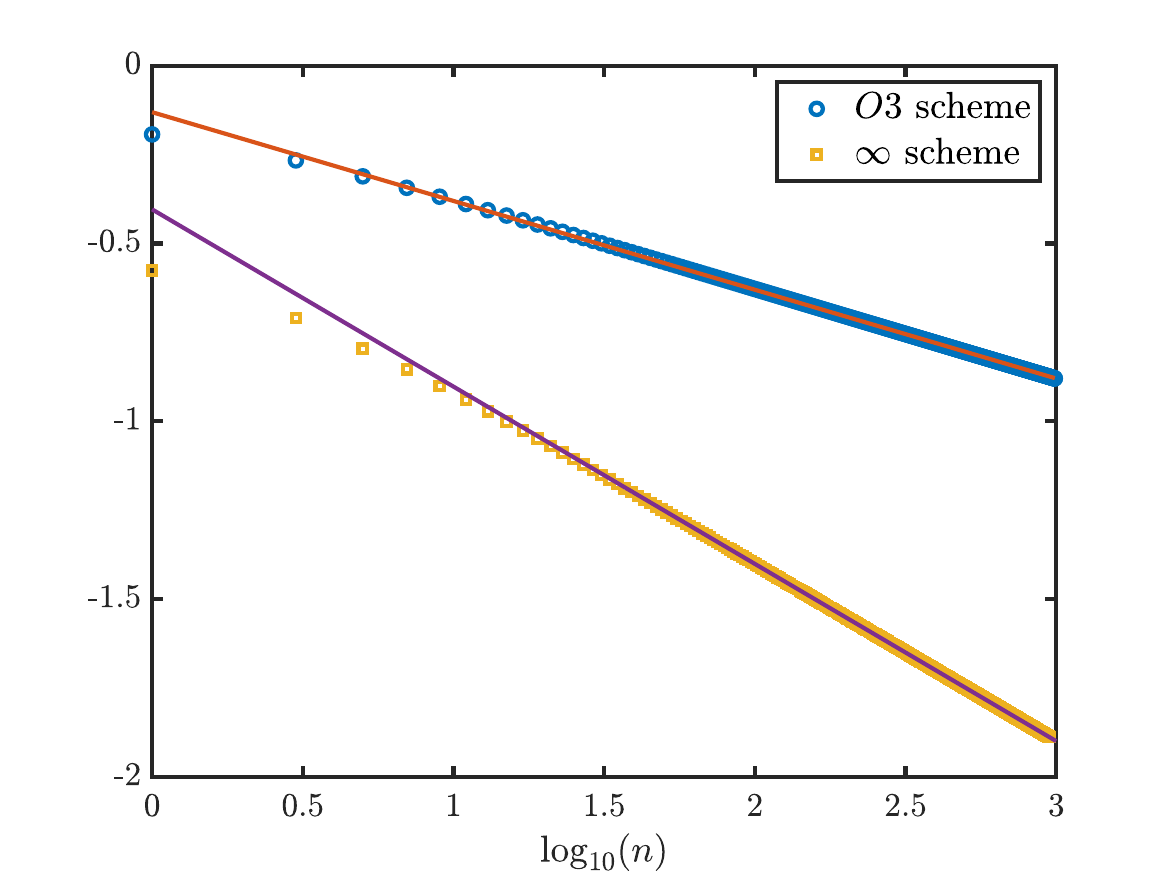}
  \caption{Illustration of the scaling factor in the generalized Gaussian bounds provided by Theorem~\ref{thm:1} in the case of the $O3$ scheme (blue
  circles) and the $\infty$ scheme (orange squares). We plot $\log_{10}\left(\sup_{j\in\Z} \big| \, \mathcal{G}^n_j \, \big|\right)$ as a function of $\log_{10}(n)$
  together with a best linear fit for each scheme for $n$ ranging from $1$ to $10^3$. For the $O3$ scheme we find a slope of $-0.2496$ while for the $\infty$
  scheme we find a slope of $-0.4983$ which compare well with the theoretical $-1/4$ and $-1/2$ scaling factors.}
  \label{fig:logplot}
\end{figure}

\subsection{Example 4: the $\infty$ scheme}

We complete our series of examples with a last explicit scheme, which we shall call the $\infty$ scheme in reference to the spectral curve
associated with it (see Figure~\ref{fig:infini} for an illustration). The $\infty$ scheme corresponds to the operators:
\bqq
Q_1^{\infty}:=I, \quad Q_0^{\infty} :=
\frac{1}{16}\mathbf{S}^{-3}+\frac{1}{4}\mathbf{S}^{-2}+\frac{7}{16} \mathbf{S}^{-1}+\frac{7}{16} \mathbf{S}-\frac{1}{4}\mathbf{S}^{2}+\frac{1}{16}\mathbf{S}^{3} \,,
\label{infinischeme}
\eqq
such that we have $r=p=3$ in that case. The definition \eqref{defF} reduces here to:
$$
F^\infty \big( \, {\rm e}^{\, \mathbf{i} \, \xi} \, \big) \, = \,
\dfrac{7}{8} \, \cos (\xi) \, + \, \dfrac{1}{8} \, \cos(3 \, \xi) \, -  \, \dfrac{\mathbf{i}}{2} \, \sin (2 \, \xi) \, .
$$
Computing:
$$
\Big| \, F \big( \, {\rm e}^{\, \mathbf{i} \, \xi} \, \big) \, \Big|^2 \, = 1 -\sin^2(\xi) \left(1-\frac{\sin^2(2\xi)}{16}\right) \, ,
$$
we find that $F(\kappa)$ belongs to $\Dbar$ for all $\kappa \in \cercle$, and $F(\kappa)$ belongs to $\cercle$ for such $\kappa$ if and
only if $\kappa = \pm 1$. We thus have \eqref{hyp:stabilite1} with $\underline{\kappa}_1 = 1$ and $\underline{\kappa}_2 = -1$, and
the reader can check that \eqref{hyp:stabilite2} is satisfied with:
$$
\alpha_1 \, = \,1 \, ,\, \alpha_2 \, = \,-1 \, ,\quad \beta_1 \, = \, \beta_2 \, := \, \dfrac{1}{2} \, , \quad \mu_1 \, = \, \mu_2 \,=\,1 \,,
$$
which means that Assumption \ref{hyp:1} is satisfied. As a consequence, we also have $\underline{z}_1 =\underline{z}_2 =1$, and we immediately
see that Assumption \ref{hyp:4} is also satisfied: both sets $\mathcal{I}_1$ and $\mathcal{I}_2$ equal $\{ 1,2 \}$ and $\alpha_1 \, \alpha_2 =-1<0$.
This is entirely similar to the behavior of the implicit scheme \eqref{implicit} except that \eqref{infinischeme} corresponds to a finitely supported
convolution operator. The spectral curve $F(\mathbb{S}^1)$ is illustrated in Figure~\ref{fig:infini}. Since the modulus of $F(\kappa)$ attains its
maximum at two points of $\cercle$ and that $\alpha_1\neq \alpha_2$, we cannot apply the uniform Gaussian bound from \cite{Diaconis-SaloffCoste}
or \cite{RSC2}. The improvement of Theorem \ref{thm:1} is thus meaningful here.

As far as Assumption \ref{hyp:2} is concerned, we compute:
$$
\A_{-3} (z) \, = \, - \, \dfrac{1}{16} \, ,\quad \A_3 (z) \, = \,- \dfrac{1}{16}\, ,
$$
so Assumption \ref{hyp:2} is satisfied again. We also note that Assumptions \ref{hyp:0} and \ref{hyp:3} are satisfied since the $\infty$ scheme is explicit.
We can therefore apply Theorem \ref{thm:1} which, in the case of \eqref{infinischeme}, yields the uniform Gaussian bound:
$$
\big| \, \mathcal{G}^n_j \, \big| \, \le \, \dfrac{C}{\sqrt{n}} \, \left( \exp \left( - \, c \, \, \dfrac{(j \, + \,  n)^2}{n} \right)
\, + \, \exp \left( - \, c \, \, \dfrac{(j \, - \, n)^2}{n} \right) \right) \, \quad n\geq1 \, \quad j\in\Z\, .
$$
This behavior is illustrated in Figure \ref{fig:infini}.  Note that the scaling factor $n^{-\frac{1}{2}}$ in the Gaussian bound is further demonstrated in
Figure~\ref{fig:logplot} where we represent $\sup_{j\in\Z} \big| \, \mathcal{G}^n_j \, \big|$ in logarithmic scale. Using a best linear fit, we numerically
obtain a slope of $-0.4983$ which is in good agreement with the theory.

\begin{figure}[t!]
\centering
\includegraphics[width=.39\textwidth]{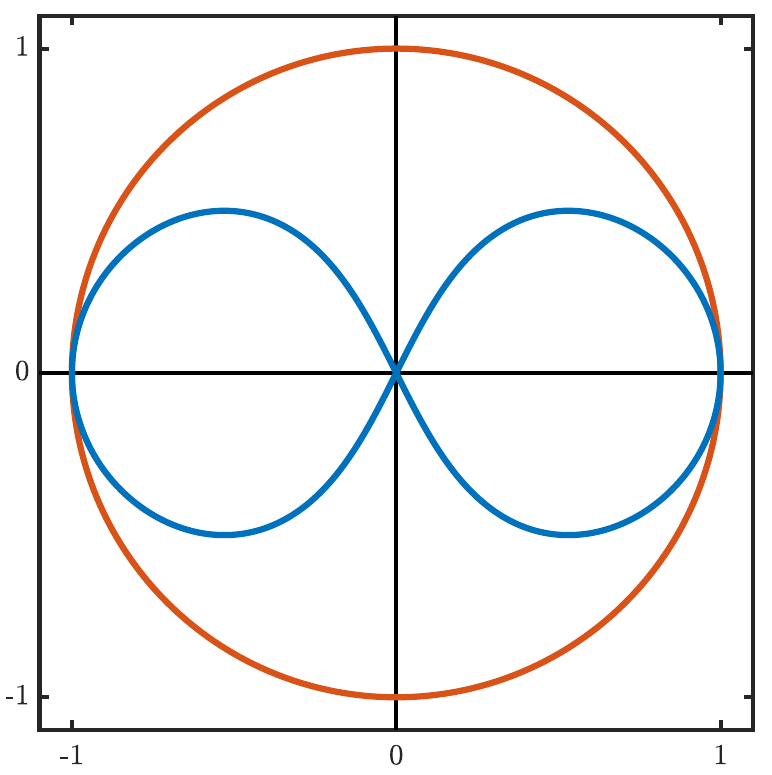}\hspace{0.5cm}
\includegraphics[width=.55\textwidth]{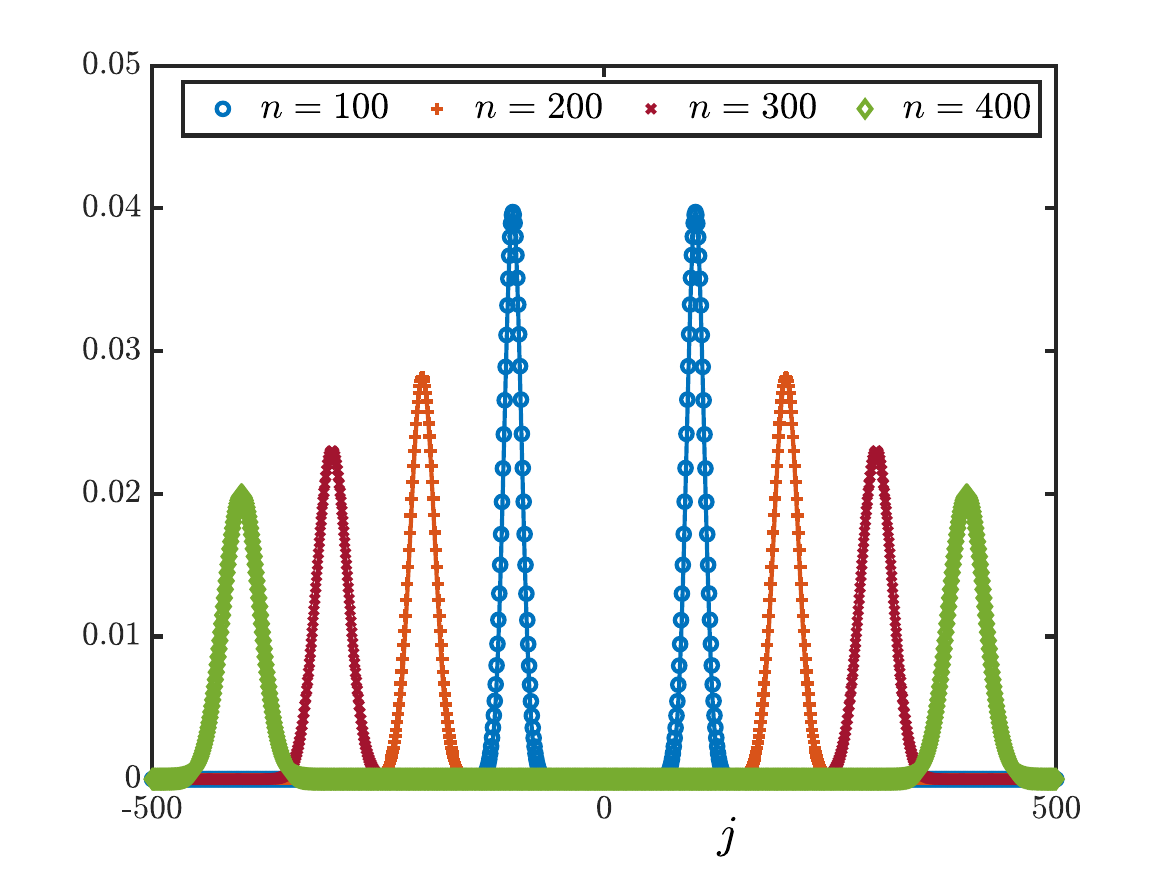}
  \caption{Left: Spectrum (blue curve) $\sigma(\mathcal{L})=F(\mathbb{S}^1)$ for the $\infty$ scheme \eqref{infinischeme}. Right: The absolute value
  of the Green's function (marked points) at different time iterations for the $\infty$ scheme \eqref{infinischeme} compared with two fixed Gaussian
  profiles centered at $j=n$ and $j=-n$ (solid lines). We started with an initial condition given by the Dirac mass $\boldsymbol{\delta}$.}
  \label{fig:infini}
\end{figure}

\subsection{Extending Theorem~\ref{thm:1} to the case where some $\alpha_k$ can be zero}

In the explicit case, it is possible to extend Theorem~\ref{thm:1} to the case where some $\alpha_k$ can be zero. The idea is to shift the operator
$\mathscr{L}=Q_0$ such that for some integer $J\in\Z$ the shifted operator $\widetilde{\mathscr{L}}:=\mathbf{S}^J Q_0$ verifies the assumptions
of Theorem~\ref{thm:1} for which we obtain a generalized Gaussian estimate which is necessarily of the form
\begin{equation*}
\forall \, n \in \N^* \, ,\quad \forall \, j \in \Z \, ,\quad \big| \, \left(\widetilde{\mathscr{L}}^n \, \boldsymbol{\delta}\right)_j \, \big| \, \le \, C \, \, \sum_{k=1}^K \,
\dfrac{1}{n^{1/(2\, \mu_k)}} \, \exp \left( - \, c \, \left( \dfrac{|j \, - \, (\alpha_k+J) \, n|}{n^{1/(2\, \mu_k)}} \right)^{\frac{2\, \mu_k}{2\, \mu_k-1}} \right) \, .
\end{equation*}
But we observe that
\bqs
\forall \, n \in \N^* \, ,\quad \forall \, j \in \Z \, ,\quad \left({\widetilde{\mathscr{L}}}^n\, \boldsymbol{\delta}\right)_j \, = \, 
\left(\mathcal{L}^n \, \boldsymbol{\delta}\right)_{j-nJ}\, ,
\eqs
which gives a generalized Gaussian estimate for $\mathcal{L}^n \, \boldsymbol{\delta}$. More precisely, we introduce a relaxed version of
Assumption~\ref{hyp:1} where some $\alpha_k$ could vanish.

\begin{assumption}
\label{hyp:5}
The function $F$ defined in \eqref{defF} satisfies Assumption~\ref{hyp:1} but possibly with some $\alpha_k$ in \eqref{hyp:stabilite2} that can be zero.
\end{assumption}

\noindent We then have the following corollary.

\begin{corollary}\label{coro2}
Assume that $Q_1$ is the identity, and that $Q_0$ satisfies Assumption~\ref{hyp:5}. If there exists some $J \in \Z$ such that $\mathbf{S}^J \, Q_0$
satisfies Assumptions~\ref{hyp:1},~\ref{hyp:2} and~\ref{hyp:4}, then there exist two constants $C>0$ and $c>0$ such that
\begin{equation*}
\forall \, n \in \N^* \, ,\quad \forall \, j \in \Z \, ,\quad \big| \, \left(\mathcal{L}^n \, \boldsymbol{\delta}\right)_j \, \big| \, \le \, C \, \, \sum_{k=1}^K \,
\dfrac{1}{n^{1/(2\, \mu_k)}} \, \exp \left( - \, c \, \left( \dfrac{|j \, - \, \alpha_k \, n|}{n^{1/(2\, \mu_k)}} \right)^{\frac{2\, \mu_k}{2\, \mu_k-1}} \right) \, .
\end{equation*}
\end{corollary}

Typically, when $Q_1$ is the identity and $K=1$ with  $\alpha_1=0$ in Assumption~\ref{hyp:1}, we can always apply Corollary~\ref{coro2} by choosing
$J$ sufficiently large (in that case, we have $r=0$ and $p>0$ for the shifted operator $\mathbf{S}^J \, Q_0$, $\A_0(z)=z$ and $\A_p(z)$ is a nonzero
constant).

\subsection{Further extensions}

For the sake of clarity, we have focused here on the case of scalar iterations, but the techniques developed in this article should apply to some
multistep iterations of the form
$$
\begin{cases}
Q_{s+1} \, u^{n+s+1} \, = \, Q_s \, u^{n+s} \, + \, \cdots \, + \, Q_0 \, u^n \, ,& n \in \N \, ,\\
u^0,\dots,u^s \in \ell^2(\Z) \, ,&
\end{cases}
$$
where $s \in \N$ is a given fixed integer and there are now $s+2$ convolution operators with finite support involved. Of course, the statement
of the assumptions should be suitably modified (for instance, Assumption \ref{hyp:0} now bears on $Q_{s+1}$ and not on $Q_1$).

We have focused here on numerical schemes for which the modulus of the amplification factor $F$ is not constant on $\cercle$ and such that
the local behavior of $F$ near a point where its modulus attains its maximum is dictated as in \cite{Thomee}. We recall that for explicit operators
($Q_1=I$) of the form \eqref{operateurs}, the main result in \cite{Thomee} shows that Assumption \ref{hyp:1} is necessary and sufficient for $Q_0$
to be power bounded from $\ell^1(\Z;\C)$ to $\ell^1(\Z;\C)$ (or equivalently from $\ell^\infty(\Z;\C)$ to $\ell^\infty(\Z;\C)$). A major advantage of the
approach developed in this article is that it gets rid (more or less) of Fourier analysis. In particular, we have used the above strategy in \cite{CFbord}
for proving a sharp stability result on a discretized transport equation with numerical boundary conditions under a degenerate version of the so-called
Kreiss-Lopatinskii condition (see \cite{gks,kreiss-wu,gko} for some background on numerical boundary conditions for hyperbolic equations).
The problem considered in \cite{CFbord} is set in $\ell^2(\N;\C)$ rather than $\ell^2(\Z;\C)$, which makes many Fourier based techniques useless.
Eventually, the above strategy is used in \cite{coeuret} to sharpen the local limit theorem of \cite{RSC1} and prove generalized Gaussian estimates
for the remainder in the local limit theorem of \cite{RSC1}.

\section*{Acknowledgements}

We are indebted to the referee for an extremely careful reading of an earlier version of the manuscript, which has greatly helped us to improve the
presentation. We also wish to thank Lucas C{\oe}uret for his numerous comments and for suggesting the argument of Corollary \ref{coro2} (which is
also used in \cite{coeuret}).

\appendix
\section{Appendix. Proof of intermediate and related results}
\label{appendix}

\subsection{Behavior of the amplification factor on the unit circle}

The aim of this subsection is to prove the following result, which generalizes the classification obtained by Thom\'ee \cite[page 280]{Thomee} for
trigonometric polynomials. Lemma \ref{lem:comportementF} below shows that this classification only depends on the holomorphy of the considered
function on a neighborhood of the unit circle.

\begin{lemma}
\label{lem:comportementF}
Let $\delta>0$ and let $f$ be a holomorphic function on the annulus:
$$
\big\{ \, \zeta \in \C \, | \, {\rm e}^{\, - \, \delta} \, < \, |\zeta| \, < \, {\rm e}^{\, \delta} \big\} \, ,
$$
that satisfies:
\begin{equation}
\label{conditionstabilite}
\sup_{\kappa \in \cercle} \, |f(\kappa)| \, = \, 1 \, .
\end{equation}
Then one of the following is satisfied:
\begin{itemize}
 \item $f(\kappa)$ has modulus $1$ for any $\kappa \in \cercle$,
 \item there exists a finite set of points $\{ \underline{\kappa}_1,\dots,\underline{\kappa}_K \}$, $K \ge 1$, in $\cercle$ such that
 $f(\underline{\kappa}_k)$ has modulus $1$ for any $k \in \{ 1,\dots,K \}$ and:
$$
\forall \, \kappa \in \cercle \setminus \big\{ \underline{\kappa}_1,\dots,\underline{\kappa}_K \big\} \, ,\quad \big| \, f(\kappa) \, \big| \, < \, 1 \, .
$$
\end{itemize}
\end{lemma}

\noindent Assumption \ref{hyp:1} in our work thus excludes the case where the rational function $F$ in \eqref{defF} has modulus $1$ on the whole
unit circle $\cercle$ (an example of such functions are the so-called Blaschke products \cite{rudin}).

\begin{proof}[Proof of Lemma \ref{lem:comportementF}]
We first consider a point $\underline{\kappa} \in \cercle$ such that $f(\underline{\kappa})$ has modulus $1$. Writing:
$$
f \big( \underline{\kappa} \, {\rm e}^{\, \zeta} \big) \, = \, f(\underline{\kappa}) \, \exp \, ( \, g(\zeta) \, ) \, ,
$$
for some holomorphic function $g$ on a neighborhood of $0$, we can conclude that there exists a power series $\sum a_n \, x^n$ with real
coefficients and a positive radius of convergence, such that for any sufficiently small real number $\xi$, there holds:
$$
\big| \, f \big( \underline{\kappa} \, {\rm e}^{\, \mathbf{i} \, \xi} \big) \, \big| \, = \, \exp \, \left( \, \sum_{n=0}^{+ \, \infty} \, a_n \, \xi^{\, n} \, \right) \, .
$$
Since $f(\underline{\kappa})$ has modulus $1$, we have $a_0=0$. Using now the condition \eqref{conditionstabilite}, we can conclude that
either all the coefficients $a_n$, $n \in \N$, are zero or there exists a smallest nonzero \emph{even} integer $2 \, p$ such that:
$$
a_{2 \, p} \, < \, 0 \qquad \text{\rm and } \qquad a_0 \, = \, a_1 \, = \, \cdots \, = \, a_{2\, p \, - \, 1} \, = \, 0 \, .
$$
In particular, for any $\xi$ in some interval $(-\alpha,\alpha)$ with $\alpha>0$, there holds:
$$
\big| \, f(\underline{\kappa} \, {\rm e}^{\, \mathbf{i} \, \xi}) \, \big| \, \le \, 1 \, - \, \dfrac{|a_{2\, p}|}{2} \, \xi^{\, 2 \, p} \, .
$$
If all the coefficients $a_n$ are zero, we can conclude that there exists an interval $(-\alpha,\alpha)$ with $\alpha>0$, such that for
any $\xi \in (-\alpha,\alpha)$, the modulus of $f(\underline{\kappa} \, \exp (\mathbf{i} \, \xi))$ equals $1$. We have thus classified the two
possible behaviors of the modulus of $f$ near any point $\underline{\kappa} \in \cercle$ at which the modulus of $f$ attains its maximum.
\bigskip

With this preliminary fact at our disposal, let us now consider the set:
$$
\mathcal{O} \, := \, \Big\{ \, \kappa \in \cercle \, | \, \exists \, \alpha \, > \, 0 \, , \, \forall \, \xi \in (-\alpha,\alpha) \, , \quad
\big| \, f \big( \underline{\kappa} \, {\rm e}^{\, \mathbf{i} \, \xi} \big) \, \big| \, = \, 1 \, \Big\} \, .
$$
The set $\mathcal{O}$ is clearly open and the previous argument on the local behavior of $|f|$ near any point of $\cercle$ where it attains its
maximum shows that the set $\mathcal{O}$ is closed. Since $\cercle$ is connected, $\mathcal{O}$ is either empty or equal to $\cercle$.

Let us now prove the claim of Lemma \ref{lem:comportementF}. The case $\mathcal{O} = \cercle$ corresponds to the first possibility where
$f(\kappa)$ has modulus $1$ for any $\kappa \in \cercle$. We thus assume that the modulus of $f$ is non-constant on $\cercle$ and therefore
$\mathcal{O}$ is empty. Then any point $\underline{\kappa}$ admits an open neighborhood $\mathcal{V}$ in $\cercle$ such that:
$$
\forall \, \kappa \in \mathcal{V} \setminus \{ \underline{\kappa} \} \, ,\quad |\, f(\kappa) \, | \, < \, 1 \, .
$$
The conclusion follows from the compactness of $\cercle$.
\end{proof}

\subsection{The Bernstein type inequality}

This subsection is devoted to the proof of Lemma \ref{lem:Bernstein}. A proof of Lemma \ref{lem:Bernstein} is provided in \cite{CFbord} in
the particular case where the coefficients $a_{\ell,0}$ are real and $K=1$, $\underline{\kappa}_1=1$. We explain below why the result holds
in the broader context of complex valued sequences and an arbitrary number $K$ of tangency points. The result of Lemma \ref{lem:Bernstein}
is a variation on the so-called Courant-Friedrichs-Lewy condition for numerical approximations of hyperbolic equations \cite{cfl}. This condition,
which bears on continuous and numerical domains of dependence, is known to be related to the von Neumann stability condition and the
Bernstein inequality for trigonometric polynomials, see, for instance, \cite{QueffelecZarouf} for a proof and historical comments on the Bernstein
inequality. We refer to \cite[page 152]{strang2} for the link between the CFL condition and the Bernstein inequality.

\begin{proof}[Proof of Lemma \ref{lem:Bernstein}]
Since $Q_1$ is the identity, we have $F=\widehat{Q}_0$. Given an integer $k \in \{ 1,\dots,K \}$, we introduce a polynomial function $P_k$ defined by:
$$
\forall \, z \in \C \, ,\quad P_k(z) \, := \, z^{\, r}\, \dfrac{F (z \, \underline{\kappa}_k)}{F (\underline{\kappa}_k)} \, = \,
\dfrac{1}{\sum_{\ell=-r}^p \, a_{\ell,0} \, \underline{\kappa}_k^\ell} \,
\sum_{\ell=-r}^p \, a_{\ell,0} \, \underline{\kappa}_k^\ell \, z^{\, \ell \, + \, r} \, .
$$
Assumption \ref{hyp:1} implies that $P_k$ is a non-constant holomorphic function on $\C$ and, furthermore, the modulus of $P_k$ is not larger than $1$
on $\cercle$. By the maximum principle for holomorphic functions \cite[chapter 12]{rudin}, $P_k$ maps $\D$ onto $\D$. In particular, there holds:
$$
\forall \, \epsilon \in (0,1) \, ,\quad |P_k(1-\epsilon)| \, < \, 1 \, ,
$$
and we easily obtain $P_k(1)=1$. We now use \eqref{hyp:stabilite2} and compute:
\begin{equation}
\label{DLPk}
P_k({\rm e}^{\, \mathbf{i} \, \xi}) \, = \, \exp \left( \, \mathbf{i} \, (r-\alpha_k) \, \xi \, - \, \beta_k \, \xi^{\, 2 \, \mu_k} \, + \, O \Big( \xi^{\, 2 \, \mu_k+1} \Big) \right) \, ,
\end{equation}
from which we get $P_k'(1)=r-\alpha_k$ and the asymptotic expansion:
$$
|P_k(1-\epsilon)| \, = \, 1 - \, (r-\alpha_k) \, \epsilon \, + \, O(\epsilon^2) \, ,
$$
as (the real number) $\epsilon$ tends to zero. Arguing by contradiction, this gives $r-\alpha_k \ge 0$ since $P_k$ maps $\D$ onto $\D$. We now show
that $\alpha_k$ can not equal $r$ and argue again by contradiction. Assuming $\alpha_k=r$, \eqref{DLPk} reduces to:
$$
P_k({\rm e}^{\, \mathbf{i} \, \xi}) \, = \, \exp \left( \, - \, \beta_k \, \xi^{\, 2 \, \mu_k} \, + \, O \Big( \xi^{\, 2 \, \mu_k+1} \Big) \right) \, ,
$$
and we thus obtain\footnote{The crucial fact here is that \eqref{DLPk} holds for either real or complex values of $\xi$.}, for real positive values of $\epsilon$
tending to zero:
$$
P_k({\rm e}^{\, \mathbf{i} \, \epsilon \, \exp (\mathbf{i} \, \pi /(2\, \mu_k))}) \, = \,
\exp \left( \, \beta_k \, \epsilon^{\, 2 \, \mu_k} \, + \, O \Big( \epsilon^{\, 2 \, \mu_k+1} \Big) \right) \, .
$$
This leads to a contradiction because $\exp(\mathbf{i} \, \epsilon \, \exp (\mathbf{i} \, \pi /(2\, \mu_k)))$ belongs to $\D$ for any $\epsilon>0$ and
$\beta_k$ has positive real part.

In order to prove the inequality $\alpha_k>- \, p$, we introduce the complex reciprocal polynomial $Q_k$ of $P_k$ (see, again, \cite{QueffelecZarouf}):
$$
Q_k(z) \, := \, z^{\, p \, + \, r} \, \overline{P_k(1/\overline{z})} \, = \, \dfrac{1}{\sum_{\ell=-r}^p \, \overline{a_{\ell,0}} \, \overline{\underline{\kappa}_k}^\ell} \,
\sum_{\ell=-r}^p \, \overline{a_{\ell,0}} \, \overline{\underline{\kappa}_k}^\ell \, z^{\, p \, - \, \ell} \, .
$$
By the same argument as the one used for $P_k$, we have $Q_k(1)=1$ and $Q_k$ maps $\D$ onto $\D$. We also compute:
$$
Q_k'(1) \, = \, p \, + \, r \, - \, \overline{P_k'(1)} \, = \, p \, + \, \alpha_k \, ,
$$
and we therefore have $p \, + \, \alpha_k \ge 0$ because $Q_k$ maps $\D$ onto $\D$. We now show that $\alpha_k$ can not equal $- \, p$ and argue
again by contradiction. Assuming $\alpha_k =- \, p$, we use \eqref{DLPk} and compute:
\begin{align*}
Q_k({\rm e}^{\, \mathbf{i} \, \epsilon \, \exp (\mathbf{i} \, \pi /(2\, \mu_k))}) &= \,
{\rm e}^{\, \mathbf{i} \, (p \, + \, r) \, \epsilon \, \exp (\mathbf{i} \, \pi /(2\, \mu_k))} \,
\overline{P_k \big( {\rm e}^{\, \mathbf{i} \, \epsilon \, \exp (- \, \mathbf{i} \, \pi /(2\, \mu_k))} \big)} \\
&= \, {\rm e}^{\, \mathbf{i} \, (p \, + \, r) \, \epsilon \, \exp (\mathbf{i} \, \pi /(2\, \mu_k))} \,
\overline{\exp \left( \mathbf{i} \, (r-\alpha_k) \, \epsilon \, \exp (- \, \mathbf{i} \, \pi /(2\, \mu_k)) + \beta_k \, \epsilon^{\, 2 \, \mu_k}
+ O \Big( \epsilon^{\, 2 \, \mu_k+1} \Big) \right)} \\
&= \, \exp \left( \, \overline{\beta_k} \, \epsilon^{\, 2 \, \mu_k} \, + \, O \Big( \epsilon^{\, 2 \, \mu_k+1} \Big) \right) \, .
\end{align*}
We are again led to a contradiction. The proof of Lemma \ref{lem:Bernstein} is thus complete.
\end{proof}

\newpage

\bibliographystyle{alpha}
\bibliography{CF}
\end{document}